%% file: main.tex
\documentclass[11pt]{article}
\usepackage[dvipsnames]{xcolor}

\usepackage{algpseudocode}

\input{style}
\newcommand{\dm}{decision maker}  
\newcommand{\Dm}{Decision maker}  
\usepackage{authblk}

\newcommand{\setalldist}[1]{\mathcal{P}(\R^{#1})}
\newcommand{\setmoments}[1]{{\mathcal{M}_2^{#1}}}
\newcommand{\PPz}[1]{\PP^{#1}_z}  
\newcommand{\g}{g}    
\newcommand{\gz}{\g} 
\newcommand{\gradgz}{\nabla \gz} 

\newcommand{\norm}[1]{\left\lVert #1 \right\rVert}

\usepackage{dsfont}

\usepackage{xcolor}
\hypersetup{colorlinks, breaklinks,
linkcolor={[RGB]{174,83,139}},
citecolor={[RGB]{97,62,173}},
urlcolor={[RGB]{115,135,195}}}
\usepackage{cleveref}
\setcitestyle{authoryear}

\newcommand{\revision}[2][black]{\textcolor{#1}{#2}}

\begin{document}

\title{Optimality of Linear Policies in Distributionally Robust Linear Quadratic Control}

\author[$\dagger$]{Bahar Ta{\c{s}}kesen}
\affil[$\dagger$]{\small Booth School of Business, University of Chicago \authorcr \href{mailto:bahar.taskesen@chicagobooth.edu}{\texttt{bahar.taskesen@chicagobooth.edu}}}

\author[$\ddagger$]{Dan A. Iancu}
\affil[$\ddagger$]{Operations, Information \& Technology, Graduate School of Business, Stanford University \authorcr \href{daniancu@stanford.edu}{\texttt{daniancu@stanford.edu}}}

\author[$\ast$]{\c{C}a\u{g}{\i}l Ko\c{c}yi\u{g}it}
\affil[$\ast$]{Luxembourg Centre for Logistics and Supply Chain Management, University of Luxembourg \authorcr \href{cagil.kocyigit@uni.lu}{\texttt{cagil.kocyigit@uni.lu}}}

\author[$\ast\ast$]{Daniel Kuhn}
\affil[$\ast\ast$]{Risk Analytics and Optimization Chair, EPFL \authorcr \href{daniel.kuhn@epfl.ch}{\texttt{daniel.kuhn@epfl.ch}}}

\maketitle

\begin{abstract}%
  We study a generalization of the classical discrete-time, Linear-Quadratic-Gaussian (LQG) control problem where the \revision{distributions of the noise terms} affecting the states and observations are unknown and chosen adversarially from divergence-based ambiguity sets centered around a known nominal distribution. \revision{The noise terms are allowed to have nonzero mean but are required to have finite second moments that satisfy an orthogonality condition, which is equivalent to uncorrelatedness if noise means are zero.} For a finite horizon model with \revision{zero-mean} Gaussian nominal noise and a structural assumption on the divergence that is satisfied by many examples -- including 2-Wasserstein distance, Kullback-Leibler divergence, moment-based divergences, entropy-regularized optimal transport, or Fisher (score-matching) divergence -- we prove that a control policy that is \emph{affine} in the observations is optimal and the adversary's corresponding worst-case optimal distribution is Gaussian. \revision{Under a weak condition satisfied by all these examples, we then prove} that the adversary should optimally set the distribution's mean to zero and the optimal control policy becomes \emph{linear}. Moreover, the adversary should optimally ``inflate" the noise by choosing covariance matrices that dominate the nominal covariance in Loewner order. Exploiting these structural properties, we develop a Frank-Wolfe algorithm whose inner step solves standard LQG subproblems via Kalman filtering and dynamic programming and show that the implementation consistently outperforms semidefinite-programming reformulations of the problem. All structural and algorithmic results extend to an infinite-horizon, average-cost formulation, yielding stationary linear policies and a time-invariant Gaussian distribution for the adversary. Lastly, we show that when the divergence is 2-Wasserstein, the entire framework remains valid when the nominal distributions are elliptical rather than Gaussian.
\end{abstract}

\section{Introduction}
\label{sec:introduction}
The Linear Quadratic Gaussian (LQG) control problem has served as a fundamental building block for a wide range of applications in management \citep{bensoussan2007partially,ref:holt_1955_LDR}, economics \citep{Hansen-Sargent-2005}, finance \citep{abeille2016lqg}, engineering \citep{Auger_eta_al_2013,Chen_2012_survey_robotic_vision}, or medicine \citep{patek2007lqg,chakravarty2020clad,kazemian2019glaucoma,todorov2002optimal}.

The discrete-time, finite-horizon formulation considers the problem of minimizing the expected costs incurred when controlling a linear dynamical system over a finite number of periods $t\in \{0,1,\dots,T-1\}$. The system evolves according to the equations
\label{system}
\begin{equation}
    \label{eq:dynamics}
    x_{t+1} = A_t  x_t + B_t  u_t +  w_t \quad \forall t  \in \{0,1,\dots,T-1\},
\end{equation}
where $x_t\in\R^n$ denotes the system states, $u_t \in\R^m$ denotes the control inputs, $w_t\in\R^n$ denotes an exogenous noise process, and the system matrices~$A_t\in \R^{n\times n}$ and~$B_t \in \R^{n \times m}$ are known. The \dm{} only has access to imperfect state measurements
\begin{equation}
    \label{eq:observation}
    y_t = C_t x_t + v_t \quad \forall t \in \{0,1,\dots,T-1\},
\end{equation}
corrupted by exogenous observation noise~$v_t\in \R^p$, where~$C_t \in \R^{p \times n}$ and usually~$p\leq n$ (so that observing~$y_t$ does not allow perfectly reconstructing~$x_t$ even without observation noise). The control inputs~$u_t$ are \emph{causal}, i.e., depend on the past observations~$y_0,\ldots, y_t$ but not on the future observations~$y_{t+1},\ldots, y_{T-1}$, so that the set of feasible control inputs $\mathcal{U}_y$ is the set of random vectors $u = (u_0,u_1,\dots,u_{T-1})$ such that $u_t = \varphi_t(y_0,\ldots,y_t)$ for every~$t$, where~$\varphi_t:\R^{p(t+1)}\rightarrow\R^m$ is a measurable control policy. Controlling the system generates quadratic costs:
\begin{equation}
    \label{eq:lqr-cost-function}
    J(u) = \sum\limits_{t=0}^{T-1}(  x_t^\top Q_t x_t +   u_t^\top R_t   u_t) +   x_{T}^\top Q_T   x_T,
\end{equation}
where~$Q_t \in \R^{n \times n}$ are positive semidefinite matrices governing the state costs, and $R_t \in \R^{m \times m}$ are positive definite matrices governing the input costs. Under the assumption that the joint probability distribution $\PP$ for the noise terms is known, the classical LQG problem seeks causal control inputs 
that minimize the expected costs under the distribution $\PP$, i.e., $\inf_{u \in \mathcal{U}_y} \EE_{\PP}[J(u)]$.

To prove structural results and design tractable algorithms for solving this problem, several assumptions on the probability distribution $\PP$ are typically needed. 
Under the premise that all noise terms have zero means and are mutually independent, it is known that the problem admits an optimal control policy of the form $u\opt_t=K_t \hat x_t$ for every $t\in \{0,\dots,T-1\}$. Here, the feedback gain matrices $K_t\in\R^{m\times n}$ only depend on the system and cost matrices $\{A_\tau,B_\tau,Q_\tau,R_\tau\}_{\tau \geq t}$ and can be obtained by solving a set of recursive equations for a system without noise, and $\hat{x}_t=\mathbb E_{\PP}[x_t|y_0,\ldots,y_t]$ is the minimum mean-squared-error (MMSE) estimator of the state $x_t$ given the history of observations $y_0,\dots,y_t$. This \emph{separation principle} holds regardless of the specific probability distribution $\PP$, but does not readily lead to tractable algorithms because calculating the MMSE estimator $\hat{x}_t$ is intractable for general probability distributions $\PP$. As such, the LQG model also makes the additional assumption that $\PP$ is \emph{Gaussian}\footnote{Note that if $\PP$ is a multivariate Gaussian distribution, the requirement that noise terms are independent can be relaxed to only requiring that they are uncorrelated, i.e., $\EE_{\PP}[z' z^\top]=0$ for all $z\neq z' \in \{x_0, w_0, \ldots, w_{T-1}, v_0, \ldots, v_{T-1}\}$.}, in which case the optimal state estimator $\hat{x}_t$ depends linearly on the history of observations $y_0,\dots,y_t$ and can be obtained efficiently with Kalman filtering techniques. (We refer the reader to Appendix~\ref{sec: appx: optimal-solution-classic-LQG} for an overview and to \citet{Bertsekas_2017} for a detailed discussion of these classical results.)

Motivated by practical settings where noise distributions may not be readily available or may not be Gaussian, we consider a generalization of the discrete-time LQG model where an adversary chooses the noise distributions from an ambiguity set $\mathcal{B}$ characterized by a divergence $\mathds{D}$ and centered around a known nominal distribution $\hat{\PP}$, and the \dm{}'s goal is to minimize the costs incurred under the worst-case distribution, $\sup_{\PP \in \mathcal{B}} \EE_{\PP}[J(u)]$. 

This optimization problem -- which we refer to as the distributionally-robust linear quadratic (DRLQ) problem -- is challenging: both the \dm{} and nature are optimizing over infinite-dimensional spaces, and the ambiguity set~$\mathcal{B}$ contains many non-Gaussian distributions, so it is not obvious which structural results from the LQG model would continue to hold or how one could compute an optimal control policy.

\subsection*{Main Contributions}
We first consider the finite-horizon case when the nominal distribution $\hat{\PP}$ is Gaussian \revision{with zero mean, as in the classical LQG model}. We construct ambiguity sets containing all distributions whose ``distance" from the nominal distribution -- measured according to a divergence $\mathds{D}$ -- is not too large. We restrict attention to distributions under which the exogenous noise terms \revision{are allowed to have non-zero means, but are required to have finite second moments that satisfy an orthogonality condition requiring cross second moments to vanish. When noise terms are zero-mean, this second-moment orthogonality (SMO) condition is equivalent to requiring the noise terms to be uncorrelated, which is also a standard requirement in the classical LQG model.} 
We require the divergence $\mathds{D}$ to satisfy a key condition, which we verify for several important examples such as 2-Wasserstein distance, Kullback-Leibler divergence, moment-based divergences, entropy-regularized optimal transport, and Fisher divergence. (The first three examples are discussed in the main text and the last two in the Appendix.)

Within this framework, we prove that an optimal control policy exists that is \emph{affine} in the observations, $u_t\opt = q_t + \sum_{\tau=0}^t U_{t,\tau} y_\tau$ for $q_t \in \R^{m}$ and $U_{t,\tau} \in \R^{m \times p}$, and that the associated worst-case optimal distribution $\PP\opt$ is \emph{Gaussian}. Our proof is novel and does not rely on traditional recursive dynamic programming arguments. Instead, we re-parameterize the control policy using purified observations and derive an upper bound for the resulting minimax formulation by relaxing the ambiguity set (to an outer approximation determined by the first two moments) while simultaneously restricting the \dm{} to affine policies. We then use convex duality to show that the upper bound matches a lower bound obtained by restricting the ambiguity set (to Gaussian distributions) in the dual of the minimax formulation. The matching bounds then certify the optimality of affine output-feedback policies for the \dm{} and of Gaussian distributions for the adversary.

Under two mild and intuitive assumptions that hold for  every divergence we consider, we derive additional structural results that yield sharp managerial insights and facilitate computation.

The first result concerns the means of the noise terms. \revision{We prove that the adversary's worst-case distribution $\PP\opt$ sets the noise mean to zero, and thus the worst-case exogenous noise terms are uncorrelated.} Whereas the vast majority of papers formulating robust LQG models restrict attention to \revision{zero-mean (and uncorrelated) noise for simplicity or in keeping with the classical LQG assumptions}, our findings provide a different justification: this assumption/choice is \emph{conservative}, because allowing the adversary to use zero means gives the adversary more power and results in the worst-case costs for the \dm{}.
Moreover, we prove that when noise is zero-mean, the optimal control policy becomes purely \emph{linear} in the outputs, $u_t\opt = \sum_{\tau=0}^t U_{t,\tau} y_\tau$. The intuition is straightforward: any deterministic bias that nature may introduce can be anticipated and neutralized by a suitable affine shift in the control policy, so it offers the adversary no advantage. In equilibrium, neither player employs predictable offsets, so when the nominal means are zero, the \dm{} also sets the intercepts to zero without incurring any optimality loss.

The second result pertains to the covariance of the noise terms. Restricting attention to zero-mean distributions and linear control policies, we prove that nature's optimal choice of covariance matrix $\Sigma\opt$ dominates the covariance matrix of the nominal distribution $\hat{\Sigma}$ in Loewner order, $\Sigma\opt \succeq \hat{\Sigma}$. Nature therefore spends its ambiguity budget by suitably ``inflating" the nominal covariance matrix $\hat{\Sigma}$, which increases the noise level and, consequently, the decision maker's optimal costs. This finding formalizes the familiar principle that higher variance entails greater uncertainty, extending it to the dynamic setting of distributionally robust LQG control. A practical implication follows immediately: when model misspecification is a concern, a simple yet effective safeguard (even against adversarial distributional ambiguity) is to up-scale the nominal covariance matrix and solve a nominal model under the resulting noisier Gaussian distribution.

We leverage these structural results to design efficient algorithms for finding optimal control policies in the DRLQ problem. We propose an algorithm based on a Frank-Wolfe first-order method that solves at each iteration sub-problems corresponding to classical LQG control problems, using Kalman filtering and dynamic programming. We show that this algorithm enjoys a sublinear convergence rate and is susceptible to parallelization. Our PyTorch implementation, which relies on automatic differentiation, yields uniformly lower runtime than a direct method based on semidefinite programming, outperforming it across every problem horizon and instance we tested. Moreover, the optimal robust policy significantly reduces worst-case costs while exhibiting virtually no performance loss when the nominal distribution is in fact correct.

We then extend our structural results to an infinite-horizon formulation of the DRLQ problem with average-cost objective, time-invariant system matrices, and time-invariant nominal distribution $\hat{\PP}$. Importantly, we do not require the control policies or all distributions in the ambiguity set to be stationary. This setting raises additional technical hurdles that require strengthening the assumptions of the finite-horizon case. However, under assumptions that mirror the classical LQG assumptions for infinite-horizon models, we prove that a time-invariant, \emph{stationary}, linear control policy $u\opt$ is optimal and that the worst-case distribution $\PP\opt$ is a \emph{time-invariant} Gaussian distribution. The result not only generalizes our finite-horizon findings but also aligns seamlessly with the structural insights long-observed in traditional infinite-horizon LQG settings with known distributions.

The Appendix elaborates on several extensions of the framework. \S\ref{sec:entropy_reg_opt_transport} and \S\ref{sec:Fisher_divergence_extension} confirm that all our structural results hold when the divergence $\mathds{D}$ is chosen as the entropy-regularized optimal transport divergence or as the Fisher divergence (also known as score-matching distance), respectively. \S\ref{sec: elliptical} then replaces the nominal Gaussian distribution with an \emph{elliptical} nominal distribution $\hat{\PP}$ with finite second moments -- a family that includes many non-Gaussian laws such as the 
Laplace, logistic, or hyperbolic distributions. Focusing on the 2-Wasserstein distance, we show that all our structural results hold and our scalable Frank-Wolfe algorithm is applicable.

\subsection{Literature Review}
\label{sec:literature review}
Our work is related to the literature on distributionally robust control, which seeks control policies that minimize expected costs under the worst-case system evolution (see, e.g.,  \citealp{Bertsimas_Iancu_Parrilo_2011,Kim_Yang_2023,Kotsalis_2021, Petersen_2000, vanParys_2016,Yang_2021} and references therein). \cite{Kim_Yang_2023} prove the optimality of linear state-feedback control policies for a related minimax LQR model with a Wasserstein distance but with {\it perfect} state observations. With perfect observations and zero-mean noise, the optimal policies in the classical LQR formulation are independent of the noise distribution and are thus inherently robust, 
so considering imperfect observations and non-zero mean noise is what makes the problem more challenging in our case. Closest to our work, \citet{taskesen2023drlqc} study the finite-horizon version of our model with a 2-Wasserstein distance and all noise distributions \emph{required} to be zero-mean and \citet{lanzetti2025optimality} consider an infinite-horizon formulation with 2-Wasserstein distance and all noise distributions required to be \emph{stationary and zero-mean}, and both papers prove that \emph{affine} output-feedback policies are optimal. Our work unifies and extends these previous results. We provide a unifying description of the ambiguity set via a divergence $\mathds{D}$ that is required to satisfy a key property, which we verify for all the aforementioned examples considered in the literature and new examples that we identify. (Indeed, to the best of our knowledge, this is the first paper to consider entropy-regularized optimal transport or Fisher divergence for distributionally robust control problems.) Moreover, we do not restrict noise distributions to be zero mean in the finite-horizon case or stationary in the infinite-horizon case; instead, we allow the adversary more freedom and we prove -- under mild assumptions -- that such restrictions are without loss of optimality. Lastly, we draw a sharper distinction between affine and linear control policies and we characterize precisely when each class is optimal and how this is related to the adversary's choices. We note that several papers in the literature have also considered robust formulations (with imperfect observations) for \emph{constrained} systems, e.g., \cite{Ben-Tal_Boyd_Nemirovski_2005,ben2006extending,vanParys_2016,Kotsalis_2021, brouillon2024drinf}; these models are more challenging and the common approach is to restrict attention to affine feedback policies for computational tractability and without proving their optimality.

Our work is also related to the literature on distributionally robust filtering for linear dynamical systems, which considers formulations without controls and focuses on the problem of estimating states. \cite{zorzi_2017}  studies a model based on $\tau$-divergences (a class that includes Kullback-Leibler divergence as a special instance) and \cite{Han2023} considers a model based on 2-Wasserstein distance, for which tractable convex reformulations are derived. Within this stream, the closest works to ours are \cite{shafieezadeh2018wasserstein}, \cite{nguyen2023bridging}, which consider the problem of minimax mean-squared-error estimation when ambiguity is modeled with a 2-Wasserstein distance from a nominal Gaussian distribution, and \cite{kargin2024distributionally}, which relaxes the assumption that noise terms are iid and investigates both finite and infinite-horizon models, proving the optimality of linear filters when the nominal distribution is Gaussian and providing tractable convex reformulations using frequency-domain techniques. For the case with Wasserstein distance, our proof relies on some ideas from these papers (such as using the Gelbrich distance to construct upper bounds), which we combine with ideas from control theory on purified output feedback to obtain the construction.


Our paper is also related to literature that documents the optimality of linear/affine policies in (distributionally) robust dynamic optimization models. \cite{bertsimas2010optimality,iancu2013supermodularity} prove optimality for one-dimensional linear systems affected by additive noise and with perfect state observations, but with general convex state and/or control costs. \cite{ref:hadjiyiannis2011affine,vanParys_Goulart_2013} provide computationally tractable approaches for quantifying the suboptimality of affine control policies in finite- or infinite-horizon settings, and \cite{bertsimas2012power,elhousni2021optimality,georghiou2021optimality} characterize the performance of affine policies in two-stage (distributionally) robust dynamic models.

Our proposed algorithm for solving the DRLQ problem is a variant of the classical Frank-Wolfe algorithm for solving convex optimization problems. The original paper introducing the ideas is \citet{ref:frank1956algorithm}, and \citep{taskesen2023drlqc} also rely on these ideas to construct a tractable algorithm for the finite-horizon, 2-Wasserstein DRLQ formulation. We extend that construction and generalize it to other ambiguity sets.

\subsection*{Notation}
All random objects are defined on a measurable space $(\Omega, \mathcal F)$. If this measurable space is equipped with a probability measure~$\PP$, then the distribution of any random vector~$z:\Omega\rightarrow\R^{d_z}$ is given by the pushforward distribution ${\mathbb P}_z={\mathbb P}\circ z^{-1}$ of~${\mathbb P}$ with respect to~$z$. The expectation under~$\PP$ is denoted by~$\mathbb E_{{\mathbb P}}[\cdot]$. We use $\mathcal{P}(\R^d)$ to denote the set of probability measures supported on $\R^d$ with finite second moments. The sets of all natural numbers with and without~$0$ are denoted by~$\mathbb N$ and~$\mathbb N_+$, respectively. For any $t \in \mathbb N$, we set $[t] = \{0, \ldots, t\}$.  
We use $\Vert A \Vert_F$ to denote the Frobenius norm and $\Vert A \Vert_2$ to denote the spectral norm of a matrix~$A \in \R^{n\times m}$.
We let $\setmoments{d} = \{ (\mu,M) \,: \mu \in \R^d, \, M \in \mathbb{S}_+^d, \, M \succeq \mu \mu^\top \}$ denote the set of valid pairs of first and second moments for a $d$-dimensional random vector. For $(\mu,M) \in \setmoments{d}$, we use $\mathcal{N}(\mu, M)$ to denote a Gaussian distribution with mean $\mu$ and second-moment matrix $M$. (Occasionally and when no confusion can arise, we also refer to the Gaussian distribution in terms of its mean $\mu$ and covariance matrix $\Sigma = M - \mu \mu^\top$.)
 
%
%
\section{Ambiguity Model, Assumptions, and Examples}
\label{sec:problem-definition}
We consider a discrete-time, linear dynamical system like the one described in \S\ref{sec:introduction} and assume that the initial state~$x_0$ and the noise terms~$\{w_t\}_{t=0}^{T-1}$ and~$\{v_t\}_{t=0}^{T-1}$ are exogenously determined 
and governed by an unknown probability distribution. Because all random vectors appearing in our model are functions of these exogenous uncertainties, we set the sample space without loss of generality as~$\Omega = \R^n \times \R^{n\times T}  \times \R^{p \times T}$. 
We use $\mathcal F$ to denote the Borel $\sigma$-algebra on~$\Omega$ and~$\PP$ to denote the joint probability distribution of these random vectors. The joint distribution~$\PP$ is only known to belong to an ambiguity set~$\mathcal{B}$. To model $\mathcal{B}$, we consider a known \emph{nominal} distribution $\hat{\PP}$ with marginals $\hat \PP_{z}$ for every $z \in \mathcal Z = \{x_0, w_0, \ldots, w_{T-1}, v_0, \ldots, v_{T-1}\}$ and a \emph{divergence} $\mathds D : \setalldist{d_z} \times \setalldist{d_z} \to [0,+\infty]$ satisfying $\mathds D(\PP_z,\PP_z)=0$ whenever $\PP_z \in \setalldist{d_z}$, and we construct the ambiguity set~$\mathcal{B}$ as:
\begin{equation}
    \begin{aligned}
        &\mathcal{B} = \left\{\PP \in \mathcal P(\R^{n+ T(n+ p)}): \PP_{z} \in \mathcal{B}_z, ~ \EE_{\PP}[z' z^\top] = 0 ~\forall z\neq z'\in \mathcal Z \right\},
    \end{aligned}
    \label{eq:ambiguity_set_B_def}
\end{equation}
where, for all $z \in \mathcal Z$ and for finite $\rho_z \geq  0$,
\begin{equation}
    \begin{aligned}
        &\mathcal{B}_z = \left\{\PP_z \in \setalldist{d_z}: \mathds D(\PP_{z}, \hat \PP_{z}) \leq \rho_{z} \right\}.\\
    \end{aligned}
    \label{eq:ambiguity_set_Bz_def}
\end{equation}
\revision{We refer to the requirement that $\EE_{\PP}[z' z^\top]=0$ for any $z \neq z' \in \mathcal{Z}$ as the \emph{second moment orthogonality} (SMO) condition. Note that if $\EE_\PP[z]=\EE_\PP[z']=0$, this is equivalent to requiring that $z,z'$ are uncorelated.} Subsequently, for every $z \in \mathcal{Z}$, we use $\mu_z, M_z$, and $\Sigma_z$ to denote the mean, second moment, and covariance matrix, respectively, under a generic distribution $\PP \in \mathcal{B}$, and use $\hat{\mu}_z, \hat{M}_z$, and $\hat{\Sigma}_z$ to denote these quantities, respectively, under the nominal distribution $\hat{\PP}$.

\subsection{Assumptions for Tractability}
\label{subsec:assumptions}
To maintain tractability and rule out uninteresting cases, we impose a few assumptions on the nominal distribution $\hat{\PP}$ and on the structure of the ambiguity set $\mathcal{B}$.

\begin{assumption}
\label{ass:Gaussian}
 $\hat {\PP}$ is a Gaussian distribution satisfying \revision{$\hat\mu_z =0$ for all $z\in \mathcal Z$ and }$\hat{\Sigma}_{v_t} \succ 0$ for all $t \in [0,T-1]$.
\end{assumption}
Requiring $\hat{\PP}$ to be Gaussian renders our model computationally tractable for several cases of practical interest and is consistent with the assumptions in the classical LQG model. Appendix~\S\ref{sec: elliptical} shows that when the divergence $\mathds{D}$ corresponds to a 2-Wasserstein distance, our results also hold for any \emph{elliptical} nominal distribution $\hat{\PP}$ with finite second moments -- a class that includes many \emph{non}-Gaussian distributions such as the Laplace, logistic, or hyperbolic distributions. In general, computing the optimal control policy for an \emph{arbitrary} distribution $\PP$ would be very difficult because even computing the state estimator $\hat x_t$ is hard in that case, as formalized in the following result.
\begin{theorem}[Computational complexity of state estimation]\label{theorem: complexity result}
Computing the state estimator~$\hat x_t=\mathbb E_{\PP}[x_t|y_0,\ldots,y_t]$ is $\#$P-hard even if $t= 0$, $A_0 = B_0 =C_0 = 0$, $v_0$ and $w_0$ are mutually independent, and $\PP_{w_0}$ is a uniform distribution on some nonempty polytope $P \subseteq [0,1]^n$ given as an intersection of finitely many half-spaces. 
\end{theorem}

Requiring the observation noise terms $v_t$ to have positive definite covariance matrices, $\hat{\Sigma}_{v_t} \succ 0$, is consistent with the LQG framework and allows using the Kalman filter recursions, which compute inverses of these matrices (see Appendix \S\ref{sec: appx: optimal-solution-classic-LQG}).  All our structural results in \S\ref{sec:sec3_purified_observations}-\S\ref{sec:sec3_optimality_linear_zero_mean} hold without this assumption, so this is only required in \S\ref{sec:sec3_worst_case_covariance}. In practice, even if the assumption is not satisfied, one can construct a positive definite covariance matrix that provides an $\epsilon$-approximation to the optimal value, so the assumption does not introduce significant optimality gaps.

Regarding the ambiguity set $\mathcal{B}$, note that our model already includes two requirements: under all valid distributions $\PP$, the exogenous noise terms $z \in \mathcal{Z}$ should have finite second moments (by the definition of $\setalldist{d_z}$) and \revision{should satisfy the pairwise second-moment orthogonality condition} by~\eqref{eq:ambiguity_set_B_def}. Requiring finite second moments (which is encoded in the definition of $\mathcal{P}$) ensures that the DRLQ problem has a finite objective value; without this assumption, the adversary could choose a distribution $\PP$ resulting in an unbounded objective under any finite control policy, because the LQG cost depends quadratically on the noise terms $z \in \mathcal{Z}$. \revision{We require second-moment orthogonality (SMO) rather than uncorrelatedness for tractability. When noise means are allowed to be nonzero, a constraint that would require uncorrelatedness would be bilinear in means and covariances, making the feasible set of first and second moments non-convex. In contrast, the SMO condition is linear in the second moments, which gives rise to convex optimization problems. Because the SMO condition is equivalent to uncorelatedness for zero-mean distributions, once we establish that the worst-case means are zero, we also indirectly establish that the worst-case distribution is uncorrelated, which recovers the classical LQG structure ex post, while maintaining tractability ex ante.} 

We also include two assumptions on the divergence $\mathds{D}$ characterizing the ambiguity sets 
$\mathcal{B}_z$ in~\eqref{eq:ambiguity_set_Bz_def}.

\begin{assumption}
\label{ass:general_assumption_about_M_function}
The divergence $\mathds{D}$ satisfies the following properties:
\begin{enumerate}[noitemsep,nolistsep,label=(\roman*)] 
\item For every $(\mu_z, M_z) \in \setmoments{d_z}$, a Gaussian distribution minimizes the divergence~$\mathds D$ from $\hat \PP_z$ among all distributions with mean~$\mu_z$ and second moment matrix $M_z$, that is,
    \begin{equation*}
    \mathds D(\mathcal N(\mu_z, M_z), \hat \PP_z) = \left\{
    \begin{array}{ccll}  
         &\inf\limits_{\PP_z \in \setalldist{d_z}} &\mathds D(\PP_z, \hat \PP_z) \\
         &\st&\EE_{\PP_z}[z] = \mu_z, &\EE_{\PP_z}[z z^\top ] = M_z.
    \end{array}\right.
   \end{equation*}
    \item The set $\mathcal{M}_{(\mu_z, M_z)} = \{(\mu_z, M_z) \in \setmoments{d_z} : \mathds D(\mathcal N(\mu_z, M_z), \hat{\PP}_z) \leq  \rho_z \}$ is 
    convex and compact.
    \label{prop:2wasserstein_for_elliptical}
\end{enumerate}
\end{assumption}
Assumption~\ref{ass:general_assumption_about_M_function} holds in several important cases (see \S\ref{sec:examples}) and admits an intuitive interpretation. Requirement~(i) readily holds if the divergence $\mathds{D}(\PP_z, \hat{\PP}_z)$ depends only on the first two moments of the distributions $\PP_z, \hat{\PP}_z$. If the divergence is based on an information-theoretic principle related to uncertainty, the Gaussian distribution may satisfy requirement~(i) because it is the ``most uncertain" (maximum-entropy) distribution for a given set of first and second moments; this happens in two of our examples, corresponding to the Kullback-Leibler and Fisher divergences. More broadly,~(i) can be thought of as restricting attention to ambiguity sets $\mathcal{B}_z$ that are ``dense in Gaussians": the requirement is satisfied if for any $0 \leq \rho \leq \rho_z$, the set of distributions in $\mathcal{B}_z$ with ``distance" of at most $\rho$ from the nominal distribution -- if nonempty -- contains at least one Gaussian distribution.  Part~(ii) requires that the set of first and second moments characterizing all Gaussian distributions in the ambiguity set $\mathcal{B}_z$ is ``well behaved," i.e., it is convex and compact. This enables us to evaluate worst-case expectations of \emph{quadratic} functions of $z$ (prominent in the LQG model) over the ambiguity set $\mathcal{B}_z$ by solving finite-dimensional, convex optimization problems.

\subsection{Examples}
\label{sec:examples}
Assumption~\ref{ass:general_assumption_about_M_function} holds in several important instances, which we describe below. The formal results and proofs that help verify these properties are all included in Appendix~\S\ref{app:proofs_for_section2}.

\subsection*{Wasserstein Ambiguity Sets}
Consider an ambiguity set where $\mathds{D}$ corresponds to the 2-Wasserstein distance~$\mathds{W}$, defined as follows. 

\begin{definition}[2-Wasserstein distance]
\label{def:wasserstein-distance}
The 2-Wasserstein distance between two distributions $\PP_z, \hat\PP_z \in \setalldist{d_z}$ is given by
\begin{equation*}
    \begin{aligned}
    \mathds W (\PP_z, \hat \PP_z) = \left(\inf_{\pi \in \Pi(\PP_z, \hat\PP_z)} \int_{\R^{d_z}\times\R^{d_z}} \Vert z - \hat z \Vert_2^2 \, \diff \pi(z,\hat z)\right)^{\frac{1}{2}},
    \end{aligned}
\end{equation*}
where $\Pi(\PP_z, \hat \PP_z)$ denotes the set of all couplings of $\PP$ and $\hat\PP$, that is, all joint distributions of the random vectors $z$ and $\hat z$ with marginal distributions $\PP_z$ and $\hat \PP_z$, respectively.
\end{definition}

Ambiguity sets based on Wasserstein distance have become popular in the DRO literature due to their advantageous statistical and computational properties \citep{ref:kuhn2019wasserstein,blanchet_tutorial}.

Appendix~\S\ref{appendix:assumption2_wasserstein} shows that the Wasserstein distance $\mathds{W}$ satisfies Assumption~\ref{ass:general_assumption_about_M_function}. Leveraging known results in the literature, \Cref{prop:WassersteinSubsetGelbrich} proves that for any two distributions $\PP_z, \hat \PP_z \in \setalldist{d_z}$ with first and second moment pairs given by $(\mu_z,M_z) \in \setmoments{d_z}$ and $(\hat\mu_z, \hat M_z) \in \setmoments{d_z}$, respectively,
\begin{align} 
\mathds{W}\bigl(\PP_z,\hat \PP_z \bigr) \geq  \mathds{G}\left((\mu_z, M_z - \mu_z \mu_z^\top),(\hat \mu_z, \hat M_z - \hat \mu_z \hat \mu_z^\top)\right),
\label{eq:inequality_wasserstein_gelbrich}
\end{align}
with equality holding if $\PP_z$ and $\hat \PP_z$ are Gaussian. Here, $\mathds{G} ((\mu_z, \Sigma_z), (\hat\mu_z, \hat \Sigma_z))$ is the Gelbrich distance between two mean-covariance pairs, 
\begin{equation}\label{eq: Gelbrich distance definition}
    \mathds G \bigl((\mu_z, \Sigma_z), (\hat \mu_z, \hat\Sigma_z) \bigr) = \sqrt{\|\mu_z - \hat \mu_z\|^2+ \operatorname{Tr}\left(\Sigma_z + \hat\Sigma_z - 2\left( \hat \Sigma_z^{1/2} \Sigma_z \hat \Sigma_z^{1/2}\right)^{1/2}\right)}.
\end{equation}
This implies that~Assumption~\ref{ass:general_assumption_about_M_function}-(i) is satisfied. Assumption~\ref{ass:general_assumption_about_M_function}-(ii) is also satisfied because the set $\mathcal{M}_{(\mu_z, M_z)}$ can be written as $\{(\mu_z, M_z )\in \setmoments{d_z} : \mathds{G}((\mu_z, M_z-\mu_z\mu_z^\top), (\hat\mu_z, \hat \Sigma_z) )^2 \leq \rho_z^2 \}$, which is known to be convex and compact \citep[Proposition~3.17]{ref:nguyen2019adversarial}.

\subsection*{Kullback-Leibler Ambiguity Sets}
Next, consider an ambiguity set where the divergence $\mathds{D}$ corresponds to  the Kullback-Leibler (KL) divergence $\mathds{K}$, defined as follows.
\begin{definition}[KL divergence]
 The KL divergence from distribution~$\PP_z \in \mathcal{P}(\R^{d_z})$ to distribution $\hat \PP_z \in \mathcal{P}(\R^{d_z})$ is defined as
 \[\mathds K(\PP_z, \hat \PP_z) =  \int_{\R^{d_z}} \log\left(\frac{\diff\PP_z}{\diff \hat\PP_z}(z)\right)\diff\PP_z(z)  \]
 if $\PP_z$ is absolutely continuous with respect to $\hat \PP_z$, and $\mathds K(\PP_z, \hat \PP_z) = \infty$ otherwise.
\end{definition}
The KL divergence is a well-established measure of discrepancy between probability distributions that has found applications in statistics, information theory, computer science, economics, and many other fields, and ambiguity sets based on KL divergence have been frequently considered in  distributionally robust optimization and in robust control -- see, e.g., \citet{Whittle_risk_sensitive,ElGhaoui_2002_Var,ref:hong_2012_wp,Hansen-Sargent-2005}.


Appendix~\S\ref{appendix:assumption2_KL} leverages known results in the literature to argue that the KL divergence satisfies Assumption~\ref{ass:general_assumption_about_M_function}. \Cref{prop: KL bound} formalizes the key result that for any two distributions $\PP_z, \hat \PP_z \in \setalldist{d_z}$ with first and second moments $(\mu_z, M_z) \in \setmoments{d_z}$ and $(\hat \mu_z, \hat M_z) \in \setmoments{d_z}$, respectively,
\begin{align} 
\mathds{K}(\PP_z, \hat \PP_z) \geq  \mathds{T}\left((\mu_z, M_z - \mu_z \mu_z^\top),(\hat \mu_z, \hat M_z - \hat \mu_z {\hat \mu_z}^\top)\right),
\label{eq:inequality_KL}
\end{align}
with equality holding if $\PP_z$ and $\hat \PP_z$ are Gaussian.
Here, $\mathds{T}\bigl((\mu_z, \Sigma_z), (\hat \mu_z, \hat \Sigma_z)\bigr)$ is a (KL-type) divergence between two mean-covariance pairs, 
\begin{equation*}
 \mathds T\bigl((\mu_z,\Sigma_z), (\hat \mu_z, \hat \Sigma_z) \bigr) =  \frac{1}{2}\left( (\mu_z - \hat \mu_z)^\top \hat \Sigma^{-1} (\mu_z -\hat \mu_z) + \operatorname{Tr}\left(\Sigma_z \hat \Sigma_z^{-1} \right) - \log\det\left(\Sigma_z \hat \Sigma_z^{-1} \right) - d \right).
\end{equation*}

Result~\eqref{eq:inequality_KL} implies that~Assumption~\ref{ass:general_assumption_about_M_function}-(i) is satisfied. Assumption~\ref{ass:general_assumption_about_M_function}-(ii) is also satisfied because the set $\mathcal{M}_{(\mu_z, M_z)}$ can be expressed as
\( \bigl\{(\mu_z, M_z) \in \setmoments{d_z} : M_z - \mu_z \mu_z^\top \in \mathbb S_{++}^{d_z},~\mathds T\bigl((\mu_z,\Sigma_z),(\hat \mu_z,\hat  \Sigma_z)\bigr) \leq \rho_z \bigr\} \). This representation follows from~\eqref{eq:inequality_KL} and because we are interested in $\rho_z$ finite (which implies that $\mathds{T}$ must be finite, so any relevant $\PP_z \in \mathcal{B}_z$ must satisfy $\Sigma_z = M_z - \mu_z \mu_z^\top\succ 0$ due to the $- \log \det \Sigma_z$ term in~$\mathds{T}$). The latter set is known to be convex and compact ~\cite[Lemma~A.3]{taskesen2021sequential}. 
%
%
\subsection*{Moment Ambiguity Sets}
Lastly, we consider moment ambiguity sets where the divergence $\mathds{D}$ between two probability distributions relies only on the first two moments of the distributions. Specifically, for $\PP_z, \hat \PP_z \in \setalldist{d_z}$ with first and second moments $(\mu_z, M_z)$ and $(\hat \mu_z, \hat M_z)$, respectively, we take 
\begin{align*}
 \mathds D(\PP_z, \hat \PP_{z}) = \mathds{M}\bigl((\mu_z, M_z),(\hat \mu_z, \hat M_z) \bigr),
\end{align*}
where $\mathds{M}: \setmoments{d_z} \times \setmoments{d_z} \to [0, +\infty]$ is any divergence between pairs of first two moments satisfying $\mathds{M}(m_z,m_z)=0$ for all $m_z \in \setmoments{d_z}$. The ambiguity set $\mathcal{B}_z$ for the random variable $z$ with nominal distribution $\hat \PP_z$ (with mean~$\hat{\mu}_z$ and second moment matrix $\hat{M}_z$) can therefore be expressed as:
\begin{equation*}
\begin{aligned}
    \mathcal{B}_z = \left\{\PP_z \in \setalldist{d}: \EE_{\PP_z}[z] = \mu_z,\,\EE_{\PP_z}[z z^\top ] = M_z,~ \mathds{M}\left((\mu_z, M_z), (\hat{\mu}_z, \hat{M}_{z})\right) \leq \rho_{z}\right \}.
\end{aligned}
\end{equation*}
Assumption~\ref{ass:general_assumption_about_M_function}-(i) holds because every distribution (including a Gaussian distribution) with the same mean and second moment would yield the same divergence from the nominal distribution $\hat{\PP}_z$ and would therefore minimize the divergence from $\hat{\PP}_z$. Assumption~\ref{ass:general_assumption_about_M_function}-(ii) holds if the sublevel sets 
of the function $\mathds{M}\bigl(\cdot,\hat{m}_z\bigr) $ restricted to its first variable are convex and compact (for the given $\hat{m}_z = (\hat{\mu}_z,\hat{M}_z)$). Any restriction of $\mathds{M}$ that is quasiconvex and coercive would satisfy the requirement.

\section{Nash Equilibrium and Optimality of Linear Policies and Gaussians}
\label{sec:Nash}
This section proves our main structural results concerning the DRLQ problem. We view this problem as a game between the \dm{}, who chooses causal control inputs, and nature, which chooses a distribution~$\PP\in\mathcal{B}$. We show that this game admits a Nash equilibrium under our standing assumptions, wherein nature's strategy is a {\em Gaussian} distribution, ~$\PP_z\opt = \mathcal{N}(\mu_z\opt,M_z\opt)$ for any $z \in \mathcal{Z}$, and the \dm{}'s strategy is an {\em affine} output feedback policy. Under mild additional assumptions, we also prove that nature's optimal distribution has a \revision {mean of zero, in which case} the \dm{}'s optimal strategy becomes \emph{linear} and nature's optimal strategy entails choosing a covariance matrix for the noise terms $\Sigma_z\opt$ that dominates the nominal covariance matrix $\hat{\Sigma}_z$ in Loewner order, $\Sigma_z\opt \succeq \hat{\Sigma}_z$.

\subsection{Reformulation with Purified Observations} 
\label{sec:sec3_purified_observations}
We first simplify the problem formulation by re-parametrizing the control inputs in a more convenient form. (This follows ideas similar to~\citealp{Ben-Tal_Boyd_Nemirovski_2005, ben2006extending, ref:hadjiyiannis2011affine}.) Note that the control inputs in the formulation are subject to cyclic dependencies, as $u$ depends on~$y$, while~$y$ depends on~$x$ through~\eqref{eq:observation}, and~$x$ depends again on~$u$ through~\eqref{eq:dynamics}, etc. Because these dependencies make the problem hard to analyze, it is preferable to instead consider the controls as functions of a new set of so-called {\em purified} observations instead of the actual observations~$y_t$. 

Specifically, we first introduce a fictitious {\em noise-free} system
\begin{equation*}
    x'_{t+1} = A_t  x'_t + B_t  u_t \quad \forall t  \in [T-1]\quad\text{and}\quad y'_t = C_t x'_t \quad \forall t \in [T-1]
\end{equation*}
with states~$x'_t\in\R^n$ and outputs~$ y'_t\in\R^p$, which is initialized with~$x'_0=0$ and controlled by the {\em same} inputs~$u_t$ as the original system~\eqref{system}. We then define the purified observation at time~$t$ as~$\eta_t= y_t - {y}'_t$ and we use~$\eta = (\eta_0, \dots, \eta_{T-1})$ to denote the trajectory of {\em all} purified observations.

Because the inputs~$u_t$ are causal, the \dm{} can compute the fictitious state~$x'_t$ and output~$y'_t$ from the observations~$y_0,\ldots,y_t$. Thus, $\eta_t$ is representable as a function of~$y_0,\ldots, y_t$. Conversely, one can show by induction that~$y_t$ can also be represented as a function of~$\eta_0,\ldots, \eta_t$. Moreover, any measurable function of~$y_0,\ldots,y_t$ can be expressed as a measurable function of~$\eta_0,\ldots, \eta_t$ and vice-versa \cite[Proposition~II.1]{ref:hadjiyiannis2011affine}. So if we define~$\mathcal U_\eta$ as the set of all control inputs~$(u_0,u_1,\dots,u_{T-1})$ so that $u_t = \phi_t(  \eta_0, \dots,   \eta_{t})$ for some measurable function~$\phi_t : \R^{p(t+1)}\rightarrow\R^m$ for every~$t\in[T-1]$, the above reasoning implies that~$\mathcal U_\eta = \mathcal U_y$. Moreover, the class of causal control policies that are affine (respectively, linear) in $\eta$ is equivalent to the class of causal control policies that are affine (respectively, linear) in $y$ (see \citet{Ben-Tal_Boyd_Nemirovski_2005, ref:skaf2010design} and \Cref{lemma:linear-rel-u-eta} in \S\ref{subsec:purified_outputs_classical_LQG} for a concise proof). Therefore, in interpreting all our results, the existence of optimal affine (linear) policies $u\opt \in \mathcal{U}_\eta$ is equivalent to the existence of optimal affine (respectively, linear) policies  $u\opt \in \mathcal{U}_y$.

In view of this, we can rewrite the DRLQ problem equivalently as:
\begin{align}\label{eq:DRLQG}
    p^\star &=
    \left\{
    \begin{array}{cll}
    \min\limits_{  x,   u,   y} &\max\limits_{\mathbb P \in \mathcal{B}} \mathbb E_{\mathbb P} \left[   u^\top R   u +   x^\top Q  x \right] \notag\\
    \st &  u \in \mathcal U_{y},{x} = H u + G w,  y = C  x + v
    \end{array}\right.\\
    &= 
    \left\{
    \begin{array}{cll} \min\limits_{x,   u} &\max\limits_{\mathbb P \in \mathcal{B}} \mathbb E_{\mathbb P} \left[   u^\top R   u +   x^\top Q  x \right]\\
    \st &  u \in \mathcal U_{\eta},~{x} = H u + G w,
    \end{array}\right.
\end{align}
where $x = (x_0, \dots, x_T)$, $u = (u_0, \dots, u_{T-1})$, $y = (y_0, \dots, y_{T-1})$, $w = (x_0, w_0,  \dots,  w_{T-1})$, $v = (v_0, \dots, v_{T-1})$, $\eta = (\eta_0, \dots, \eta_{T-1})$, and $R$, $Q$, $H$, $G$ and~$C$ are suitable block matrices (see Appendix~\S\ref{sec: appx: stacked-matrices} for their precise definitions). 

The latter reformulation involving the purified observations~$\eta$ is useful for two reasons. First, the purified outputs are {\em independent} of the inputs $u$. Indeed, by recursively combining the equations of the original and the noise-free systems, one can show that
$\eta = D  w + v$ for some block triangular matrix~$D$ (see Appendix~\S\ref{sec: appx: stacked-matrices} for the construction). So the purified observations depend (affinely) on the exogenous uncertainties but do {\em not} depend on the control inputs $u$, and hence, the cyclic dependencies complicating the original system are eliminated in~\eqref{eq:DRLQG}. Second and more importantly, the purified observations will allow us to reformulate the non-convex DRLQ problem as a \emph{convex} optimization problem, as will become obvious subsequently.

Our results also rely on the dual of~\eqref{eq:DRLQG}, defined as
\begin{equation}\label{DRCPdual}
    d^\star = \left\{
    \begin{array}{ccll} 
    \max\limits_{\mathbb P \in \mathcal{B}} &\min\limits_{  x,   u} &\mathbb E_{\mathbb P} \left[   u^\top R   u +   x^\top Q  x \right]\\
    &\st  &u \in \mathcal U_{\eta},{x} = H u + G w.
    \end{array}\right.
\end{equation} 

We prove that $p\opt=d\opt$ and that a Nash equilibrium for our game exists wherein $\PP\opt$ is Gaussian and $u\opt$ is affine. The classical minimax inequality implies that~$p\opt\geq  d\opt$. To prove that~$p\opt=d\opt$, we construct an upper bound for $p^\star$ and a lower bound for $d^\star$, and then argue that these bounds match.

\subsection{Upper Bound for Primal}
\label{sec:sec3_upper_bound_p}
We obtain an upper bound for $p^\star$ by suitably \emph{enlarging} the ambiguity set~$\mathcal{B}$ and \emph{restricting} the control policies $u_t$ to be affine. We first define the following ambiguity set for the noise terms:
\begin{equation*}
    \begin{aligned}
        &\overline{\mathcal{B}} = \{\PP \in \mathcal P(\R^{n+ T(n+ p)}): \PP_{z} \in \overline{\mathcal{B}}_z, ~ \EE_{\PP}[z' z^\top] = 0 ~\forall z \in \mathcal Z, \forall z'\neq z\in \mathcal Z\},
    \end{aligned}
\end{equation*}
where, for all $z \in \mathcal Z$,
\begin{multline*}
\overline{\mathcal{B}}_z \!= \!\bigl\{\PP_z \in \setalldist{d_z}\!:\! \exists (\mu_z , M_z) \in \setmoments{d_z} \text{ with } \EE_{\PP_z}\![z] \!= \!\mu_z,\EE_{\PP_z}\![z z^\top ]\! =\! M_z, \mathds{D}\bigl(\mathcal N(\mu_z,M_z), \hat{\PP}_z\bigr) \leq \rho_z  \bigr\}.
\end{multline*}
To see that~$\overline{\mathcal{B}}$ constitutes an outer approximation for the ambiguity set~$\mathcal{B}$, note first that the random vectors~$x_0$, $\{w_t\}_{t=0}^{T-1}$ and~$\{v_t\}_{t=0}^{T-1}$ \revision{satisfy the SMO condition} 
and have finite second moments under any~$\PP\in \overline{\mathcal{B}}$.
Moreover, any $\PP_z \in \mathcal{B}_z$ satisfies 
$\mathds{D}(\PP_z, \hat{\PP}_z) \leq \rho_z$ and because 
$\mathds{D}\bigl(\mathcal{N}(\mu_z,M_z), \hat{\PP}_z\bigr) \leq \mathds{D}(\PP_z, \hat{\PP}_z)$ by Assumption~\ref{ass:general_assumption_about_M_function}-(i), we have that $\PP_z \in \overline{\mathcal{B}}_z$ and therefore $\mathcal{B} \subseteq \overline{\mathcal{B}}$. 

To finalize our construction of the upper bound on $p^\star$, we focus on affine policies of the form $u = q + U   \eta = q + U(Dw + v)$, where $q = (q_0, \dots, q_{T-1})$, and $U$ is a block lower triangular matrix
\begin{equation}
    \label{eq:block-lower-triangular}
    \begin{aligned}
    U = \begin{bmatrix}
    U_{0,0}  & & & \\
    U_{1,0} & U_{1,1} & & \\
    \vdots & &\ddots & \\
    U_{T-1,0} &\dots &\dots &U_{T-1,T-1} 
\end{bmatrix}.
    \end{aligned}
\end{equation}
The block lower triangularity of~$U$ ensures that the corresponding control policy is causal, which in turn ensures that~$u\in\mathcal U_\eta$. In the following, we denote by $\mathcal U$ the set of all block lower triangular matrices of the form~\eqref{eq:block-lower-triangular}. 

An upper bound on problem~\eqref{eq:DRLQG} can now be obtained by \emph{restricting} the \dm{}'s feasible set to causal control policies that are \emph{affine} in the purified observations~$\eta$ and by \emph{relaxing} nature's feasible set to the outer approximation~$\overline{\mathcal{B}}$ of~$\mathcal{B}$. The resulting upper bound is given by:
\begin{equation}
\overline p^\star=\left\{
    \begin{array}{cll}
     \min\limits_{U, q,   x,   u} &\max\limits_{\mathbb{P} \in \overline{\mathcal{B}}} \mathbb E_{\mathbb P} \left[   u^\top R   u +   x^\top Q   x \right]\\
    \st & U \in \mathcal U,  u = q + U(Dw + v),  {x} = H u + G w.
    \end{array}\right.
    \label{upper bound problem 1}
\end{equation}
Because we obtained~\eqref{upper bound problem 1} by restricting the feasible set of the outer minimization problem and relaxing the feasible set of the inner maximization problem in~\eqref{eq:DRLQG}, it is clear that~$\overline p^\star\geq  p\opt$. 

Although problem~\eqref{eq:DRLQG} is still an infinite-dimensional, zero-sum game because nature's choices are over distributions $\PP$, important simplifications are possible. Specifically, note that for any fixed $U,q$, the control policies $u$ and induced states $x = Hu + Gw$ are affine functions on the noise terms $w,v$, and therefore the expected value of the objective, $\mathbb{E}_{\mathbb{P}} \bigl[   u^\top R   u +   x^\top Q   x \big]$, only depends on the first two moments of the random vector $(w,v)$ under distribution $\PP$. This implies that problem~\eqref{upper bound problem 1} can be rewritten as a finite-dimensional zero-sum game, as formalized in the following result.
\begin{proposition}\label{prop: min max SDP upper bound problem}Problem~\eqref{upper bound problem 1} has the same optimal value as the optimization problem:
\begin{equation}
\overline p^\star=\!\!\left\{
\begin{array}{llr}
    \min\limits_{\substack{q \in \R^{pT}\\ U \in \mathcal U}} \max\limits_{\substack{(\mu_w, M_w) \in \mathcal{M}_{(\mu_w, M_w)}\\ (\mu_v, M_v) \in \mathcal{M}_{(\mu_v, M_v)}}} \!\!\!\!\!\!\!\! &\Tr\Big( \bigl( (UD)^\top R UD + (G+HUD)^\top Q (G+HUD) \bigr) M_w  +  U^\top \bar R U M_v\Big ) \\[-2ex]
    &+2 q^\top(\bar R UD + G^\top QH) \mu_w + 2 q^\top \bar R U \mu_v + q^\top \bar R q,
\end{array}\right.
\label{eq: distributionally robust control problem -- affine simplified min max}
\end{equation}
where $\bar R = R + H^\top Q H$ and 
\begin{align*}
    \mathcal{M}_{(\mu_w, M_w)} = \mathcal{M}_{(\mu_{x_0}, M_{x_0})} \times \prod_{t=0}^{T-1} \mathcal{M}_{(\mu_{w_t}, M_{w_t})}, \qquad 
    \mathcal{M}_{(\mu_v, M_v)} = \prod_{t=0}^{T-1} \mathcal{M}_{(\mu_{v_t}, M_{v_t})}.
\end{align*}
\end{proposition}
For this and subsequent results with omitted proofs, we refer the reader to \S\ref{app:proofs_for_section3}.
%
%
\subsection{Lower Bound for Dual}
\label{sec:sec3_lower_bound_d}
To derive a tractable lower bound on $d^\star$, we restrict nature's feasible set to the family~$\mathcal{B}_{\mathcal N}$ of all {\em Gaussian} distributions in the ambiguity set~$\mathcal{B}$. 
The resulting bounding problem is thus given by
\begin{equation}
\label{eq: dual distributionally robust control problem restriction}
    \underline{d}^\star = \left\{
    \begin{array}{ccl} \max\limits_{\mathbb P \in \mathcal{B}_{\mathcal N}} &\min\limits_{x,   u} &\mathbb E_{\mathbb P} \left[   u^\top R   u +   x^\top Q  x \right]\\
    &\st &  u \in \mathcal U_{\eta}, {x} = H u + G w.
    \end{array}\right.
\end{equation}
As we obtained~\eqref{eq: dual distributionally robust control problem restriction} by restricting the feasible set of the outer maximization problem in~\eqref{DRCPdual}, it is clear that $\underline{d}^\star \leq d\opt$. Next, by leveraging the fact that the inner minimization problem in~\eqref{eq: dual distributionally robust control problem restriction} is solved by an affine control policy for any fixed Gaussian distribution in~$\mathcal{B}_{\mathcal N}$, we show that~\eqref{eq: dual distributionally robust control problem restriction} can be recast as a finite-dimensional zero-sum game. 

\begin{proposition}
\label{prop:dual-sdp}
Problem~\eqref{eq: dual distributionally robust control problem restriction} has the same optimal value as the optimization problem:
\begin{equation}
\underline d^\star=\!\left\{\!\!\!\!\!\!\!
\begin{array}{llll} &\max\limits_{\substack{(\mu_w, M_w)  \in \mathcal{M}_{(\mu_w, M_w)}\\ (\mu_v, M_v)  \in \mathcal{M}_{(\mu_v, M_v)}}}\!\!\!\!\! &\min\limits_{\substack{q \in \R^{pT}\\ U \in \mathcal U}} \hspace{-2.1ex}
    &\Tr\Big( \bigl( (UD)^\top R UD + (G\!+\!HUD)^\top Q (G+HUD) \bigr) M_w  \!+\!  U^\top \bar R U M_v\Big ) \\[-2ex]
    &&&+2 q^\top(\bar R UD + G^\top QH) \mu_w + 2 q^\top \bar R U \mu_v + q^\top \bar R q,
\end{array}\right.
\label{eq:dr-affine-max-min}
\end{equation}
where $\bar R$, $\mathcal{M}_{(\mu_w,M_w)}$ and $\mathcal{M}_{(\mu_v,\Sigma_v)}$ are defined exactly as in \Cref{prop: min max SDP upper bound problem}. 
\end{proposition}

\subsection{Optimality of Affine Policies and Gaussian Distributions}
\label{sec:sec3_optimality_affine_gaussian}
The next result leverages the primal and dual relaxations to prove our main result that the primal DRLQ problem in~\eqref{eq:DRLQG} and its dual in~\eqref{DRCPdual} actually have the same optimal value.
\begin{theorem}[Strong duality]
The optimal value in problem~\eqref{eq:DRLQG} equals the optimal values in problems~\eqref{eq: distributionally robust control problem -- affine simplified min max}, \eqref{eq:dr-affine-max-min}~ and~\eqref{DRCPdual}, i.e., $p^\star = \bar{p}^\star = \underline{d}^\star = d^\star$, and all optimal values are attained.
\label{theorem:lower-equal-upper}
\end{theorem}
\begin{proof}{Proof.} 
By weak duality and the construction of the bounding problems~\eqref{eq: distributionally robust control problem -- affine simplified min max} and~\eqref{eq:dr-affine-max-min}, we readily have that~ $\underline d^\star\leq d^\star\leq p^\star\leq \overline{p}^\star$. To complete the argument, we prove that $\underline{d}^\star = \overline{p}^\star$. Consider problem~\eqref{eq: distributionally robust control problem -- affine simplified min max}, with optimal value $\overline{p}^\star$, and problem~\eqref{eq:dr-affine-max-min}, with optimal value $\underline{d}^\star$. Note that these problems are dual to each other, that is, they can be transformed into one another by interchanging minimization and maximization. In both problems, the set $\cal U$ appearing in the minimization is convex and closed, and the feasible sets  $\mathcal{M}_{(\mu_w,M_w)}$ and $\mathcal{M}_{(\mu_v,M_v)}$ in the maximization are convex and compact under Assumption~\ref{ass:general_assumption_about_M_function}~\ref{prop:2wasserstein_for_elliptical}. Moreover, the objective in both problems is convex quadratic in the minimization variables~$(U,q)$ and is linear in the maximization variables $(\mu_w, M_w)$ and $(\mu_v, M_v)$. Therefore, the conditions of Sion's minimax theorem~\citep{ref:sion1958minimax} are satisfied and we conclude that $\underline{d}^\star = \overline{p}^\star$.  Moreover, the optimal values in~\eqref{eq:DRLQG} and~\eqref{eq: distributionally robust control problem -- affine simplified min max} are attained because of our assumption that $R \succ 0$, whereas the optimal values in \eqref{eq:dr-affine-max-min}~ and~\eqref{DRCPdual} are attained because the feasible sets $\mathcal{M}_{\mu_w,M_w}$ and $\mathcal{M}_{\mu_v,M_v}$ are compact.
\end{proof}

\Cref{theorem:lower-equal-upper} has several important implications. From a mathematical perspective, it establishes strong duality between two {\em infinite-dimensional} zero-sum games, which is not a priori expected to hold. From a practical perspective, it implies that a \dm{} faced with solving the DRLQ problem~\eqref{eq:DRLQG} can restrict attention to policies that depend \emph{affinely} on the purified outputs $\eta$ (or, equivalently, on the original outputs $y$).
\begin{corollary}[\Dm{}'s Nash strategy is affine]
    \label{corollary:affine-controllers-are-optimal}
    The primal DRLQ problem~\eqref{eq:DRLQG} admits an optimal \textup{affine} policy of the form~$u\opt=q\opt+U\opt \eta$ for some~$U\opt\in\mathcal U$ and~$q\opt\in\R^m$.
\end{corollary}
Lastly, \Cref{theorem:lower-equal-upper} implies that the worst-case distribution in the DRLQ problem is Gaussian.
\begin{corollary}[Nature's Nash strategy is a Gaussian distribution]
   The dual DRLQ problem~\eqref{DRCPdual} admits an optimal solution that is a Gaussian distribution,~$\PP\opt\in\mathcal{B}_{\mathcal N}$.
    \label{corollary:normal-distributions-are-optimal}
\end{corollary}
\Cref{corollary:normal-distributions-are-optimal} follows from the equality~$\underline d^\star= d^\star$. Note that the optimal Gaussian distribution~$\PP\opt$ is uniquely determined by the first and second moments~$(\mu\opt_w, M\opt_w)$ and~$(\mu\opt_v, M\opt_v)$ of the exogenous uncertain parameters, which can be computed by solving problem~\eqref{eq:dr-affine-max-min}. That the worst-case distribution is actually Gaussian is not a-priori expected and is surprising given that the ambiguity set~$\mathcal{B}$ contains many non-Gaussian distributions.

At an intuitive level, \Cref{theorem:lower-equal-upper} delivers two key insights. First, it confirms the merits of affine output feedback policies, which extend from the classical LQG setting to the distributionally robust LQG setting: even when the noise distribution is unknown, a more elaborate control policy cannot outperform an affine one. Second, the result supplies a novel justification for the standard focus on Gaussian noise in the LQG model: beyond analytical tractability, this assumption is also \emph{conservative} because it allows capturing the worst-case costs that the \dm{} might face within the ambiguity set (provided the nominal distribution remains Gaussian).

\subsection{Optimality of Linear Policies and Zero-Mean Distributions}
\label{sec:sec3_optimality_linear_zero_mean}
Under some additional mild assumptions on the ambiguity sets, we can further refine the structural results concerning the \dm{}'s and nature's Nash strategies. We first state an additional assumption on the set of first two moments $\mathcal{M}_{(\mu_z, M_z)}$ defined in ~Assumption~\ref{ass:general_assumption_about_M_function}.
\begin{assumption}
    \label{ass:setting_mean_zero_feasible}
    For any $z \in \mathcal{Z}$, the set $\mathcal{M}_{(\mu_z, M_z)} = \{(\mu_z, M_z) \in \setmoments{d_z} : \mathds D(\mathcal N(\mu_z, M_z), \hat{\PP}_z) \leq  \rho_z \}$ defined in ~Assumption~\ref{ass:general_assumption_about_M_function} is such that $(\mu_z,M_z) \in \mathcal{M}_{(\mu_z, M_z)}$ implies that  $(0,M_z) \in \mathcal{M}_{(\mu_z, M_z)}$.
\end{assumption}
The assumption bears an intuitive interpretation. Formulated as a feasibility condition on the ambiguity set $\mathcal{B}_z$, it states that if a Gaussian distribution $\PP_z = \mathcal{N}(\mu_z,M_z)$ belongs to $\mathcal{B}_z$, then the ``centered" Gaussian $\PP'_z = \mathcal{N}(0,M_z)$ -- obtained by setting the mean to zero while keeping the second moment unchanged -- should also be feasible, $\PP'_z \in \mathcal{B}_z$. In other words, the adversary may always transfer deterministic bias into additional variance without leaving $\mathcal{B}_z$. Because this transformation raises the covariance from $M_z - \mu_z \mu_z^\top$ to $M_z$ and the latter dominates the former in the Loewner (positive semidefinite) order, the key intuition behind the assumption is to allow the adversary to ``inflate" uncertainty (while holding the second moment fixed).

The validity of Assumption~\ref{ass:setting_mean_zero_feasible} depends on three key inputs: the divergence $\mathds{D}$, the nominal distribution $\hat{\PP}_z = \mathcal{N}(\hat{\mu}_z,\hat{M}_z)$, and the radius $\rho_z$ of the ambiguity set. The following result shows that all our examples from~\S\ref{sec:examples} satisfy Assumption~\ref{ass:setting_mean_zero_feasible} under very mild conditions if the nominal distribution has zero mean.
\begin{proposition}
    \label{prop:assumption3_for_W_KL_M}
    If $\hat{\mu}_z=0$, the divergences $\mathds{W}, \mathds{K}$ defined in \S\ref{sec:examples} satisfy~Assumption~\ref{ass:setting_mean_zero_feasible} and the moment-based ambiguity set based on divergence $\mathds{M}$ also satisfies the assumption if 
    \begin{align}
        \mathds{M}\left( (0, M_z), (0, \hat{M}_z) \right) \leq \mathds{M}\left( (\mu_z, M_z), (0, \hat{M}_z) \right) , ~ \forall (\mu_z,M_z) \in \setmoments{d_z}.
        \label{eq:Assumption3_condition_M_amb_set}
    \end{align}
\end{proposition}
To see that the conditions are mild, note that requiring the nominal mean to be zero is very common in LQR/LQG models. Textbook treatments of the classical LQG model restrict attention to zero-mean noise without loss of generality \citep{Bertsekas_2017}, and the vast majority of distributionally-robust LQR/LQG formulations require that \emph{all} distributions in the ambiguity set be zero-mean (see, for instance, \citet{taskesen2023drlqc,Kim_Yang_2023,Han2023} and our literature review in \S\ref{sec:literature review}). For the moment-based ambiguity set, the condition holds if the divergence $\mathds{M}$ penalizes deviations from the nominal mean $\hat{\mu}_z=0$, which is a sensible requirement (for instance, this holds if $\mathds{M}$ penalizes separately the distance between the first moments and the second moments).
Assumption~\ref{ass:setting_mean_zero_feasible} allows refining our structural results on the solution to the DRLQ problem.
\begin{theorem}[Worst-case nominal mean and linear controls]\label{thm:worst-case-mean-zero}
    Under Assumptions~\ref{ass:Gaussian}-\ref{ass:setting_mean_zero_feasible}, 
    problem~\eqref{DRCPdual} admits an optimal solution $\PP\opt$
    that is Gaussian and has zero mean, $\EE_{\PP_z\opt}[z]=0$, for all $z \in \mathcal{Z}$. Moreover, under such $\PP\opt$, the inner minimization problem in~\eqref{DRCPdual} admits an optimal \emph{linear} control policy, $u\opt = U\opt \eta$ for some $U \in \mathcal{U}$.
\end{theorem}
\Cref{thm:worst-case-mean-zero} extends the classical LQG insight -- that \emph{linear} output feedback policies are optimal -- to our DRLQ model and \emph{proves} that zero-mean (Gaussian) distributions are optimal for nature.

To appreciate the intuition behind this result, recall that under nature's optimal choice of Gaussian distribution $\PP\opt$, the \dm{} can restrict attention to an affine policy, $u = q + U \eta$. The proof of \Cref{thm:worst-case-mean-zero} shows that under the \emph{optimal} choice of intercept $q$, the objective becomes a quadratic function of the means $\mu_z$ that is maximized by setting $\mu_z = 0$. This arises because nature can increase the costs achieved with a (Gaussian) distribution with nonzero mean $\mu_z$ by instead choosing a zero-mean Gaussian distribution and increasing the covariance matrix by $\mu_z \mu_z^\top$. This change -- which leads to feasible distributions, by Assumption~\ref{ass:setting_mean_zero_feasible} -- magnifies the noise and increases the \dm{}'s costs. In equilibrium, when nature uses zero-mean policies, the \dm{} can restrict attention to \emph{linear} -- rather than \emph{affine} -- policies. This conveys a simple intuition: nature cannot gain by adding a predictable offset/bias to its distribution (through the mean), because that could be corrected by the \dm{} through an appropriate choice of intercept $q$, so at optimality, neither nature nor the \dm{} add predictable offsets to their actions.

The result in \Cref{thm:worst-case-mean-zero} is also important because it provides a novel rationale for considering zero-mean noise distributions. Whereas the vast majority of papers formulating robust LQG models restrict attention to zero-mean noise for simplicity or in keeping with the classical LQG setting (as discussed in our earlier examples and in \S\ref{sec:literature review}), 
the result in \Cref{thm:worst-case-mean-zero} provides a different justification: this assumption/choice is \emph{conservative}, because allowing the adversary to use zero means gives the adversary more power and results in the worst-case costs for the \dm{}.

\revision{In view of the structural result in \S\ref{sec:sec3_optimality_linear_zero_mean} and to simplify exposition, we assume throughout the rest of the paper that the mean of the noise under any distribution in the ambiguity set is 0, i.e., $\EE_{\PP_z}[z] = \mu_z = 0$ for all $z \in \mathcal Z$ and all $\PP_z \in \mathcal{B}_z$.}
\subsection{Worst-Case Covariance Matrix}
\label{sec:sec3_worst_case_covariance}
Our final structural result further develops the intuition above and shows that nature's optimal distribution $\PP\opt$ entails suitably ``inflating" the covariance matrix of the nominal distribution $\hat{\PP}$. Because any zero-mean Gaussian distribution $\PP_z$ is fully specified through the second moment/covariance matrix $M_z\opt=\Sigma_z\opt$, we can shorten notation by using $\mathcal{M}_{\Sigma_z}$ to denote the sets $\mathcal{M}_{(\mu_z=0, M_z)}$ defined in~Assumption~\ref{ass:general_assumption_about_M_function}.

Our final result requires a mild condition on the sets $\mathcal{M}_{\Sigma_z}$ that further refines Assumption~\ref{ass:setting_mean_zero_feasible}.
\begin{assumption}
\label{ass:M-gradients}
Fix $\mu_z = \hat{\mu}_z=0$ for all $z\in \mathcal{Z}$. There exists $\gz: \mathbb S_+^{d_z} \to \R$ such that the sets defined in ~Assumption~\ref{ass:general_assumption_about_M_function} can be represented as \(\mathcal{M}_{\Sigma_z} = \{ \Sigma_z \in \mathbb{S}_+^{d_z} \,: \, \gz(\Sigma_z) \leq  0\}\) for any $z \in \mathcal{Z}$, where $\gz$ is convex on $\mathbb{S}_+^{d_z}$, differentiable on $\mathbb{S}_{++}^{d_z}$, and satisfies: (i) $\gz(\hat{\Sigma}_z) < 0$, (ii)~$\hat{\Sigma}_z \in \argmin_{\Sigma_z \succeq 0} \gz(\Sigma_z)$, and (iii) $\gradgz(\Sigma_1) \succeq \gradgz(\Sigma_2)$ implies $\Sigma_1 \succeq \Sigma_2$, for any $\Sigma_1, \Sigma_2 \in \mathbb S_+^{d_z}$.
\end{assumption}
%
%
%
\Cref{prop:assumption4_for_base_examples} in Appendix~\S\ref{sec: appx: verification-assumption-m-gradients} shows that Assumption~\ref{ass:M-gradients} is  satisfied by all examples in \S\ref{sec:examples} if $\rho_z > 0$ (and, for the case of the moment-based ambiguity, if the divergence $\mathds{M}$ satisfies a mild condition). The key requirements are intuitive if one takes $\gz(\Sigma_z)$ as a (convex, increasing) transformation of the distance $\mathds{D}(\PP_z, \hat{\PP}_z)$ between a distribution $\PP_z = \mathcal{N}(0,\Sigma_z)$ and the nominal $\hat{\PP}_z = \mathcal{N}(0,\hat{\Sigma}_z)$. Requirement~(i) simply asks that $\hat{\Sigma}_z$ is an interior point of the set of valid covariances $\mathcal{M}_z$, which holds in all our examples provided there is ambiguity, $\rho_z > 0$. Requirement~(ii) is readily satisfied because $\mathds{D}(\PP_z, \hat{\PP}_z) \geq \mathds{D}(\hat{\PP}_z, \hat{\PP}_z)$ for any divergence $\mathds{D}$. Lastly,~(iii) asks that the gradient map $\gradgz$ be order-reflecting, i.e., that gradients that dominate in the Loewner (positive semidefinite) order should correspond to (distributions whose) covariances also dominate in the Loewner order. Put more intuitively, this means that ``noisier" gradient maps should come from ``noisier" distributions.

\begin{theorem}
\label{thm:optimal-covs-are-higher}
Under Assumptions~\ref{ass:Gaussian}-\ref{ass:M-gradients}, problem~\eqref{DRCPdual} admits an optimal solution $\PP\opt$ such that $\PP_z\opt=\mathcal{N}(0,\Sigma_z\opt)$ and the optimal covariance satisfies $\Sigma\opt_{z}\succeq \hat\Sigma_{z}$, for every $z \in \mathcal{Z}$. 
\end{theorem}

\Cref{thm:optimal-covs-are-higher} offers a clear and powerful insight: when nature selects a zero-mean Gaussian distribution that maximizes the cost, its optimal strategy is to \emph{inflate} the nominal covariance matrix $\hat{\Sigma}$, so that the optimal covariance $\Sigma\opt$ dominates $\hat{\Sigma}$ in the Loewner (positive semidefinite) ordering. 
This result generalizes a familiar principle -- that increasing a random variable's variance creates ``more uncertainty" -- to the significantly more complex, dynamic environment provided by the LQG. This also suggests an important practical take-away: in LQG problems where model ambiguity is a concern, a simple yet effective rule of thumb is to inflate the covariance matrix of the noise, which will provide additional protection against uncertainty.

%
%
\section{Efficient Numerical Solution of DRLQ Problems}
\label{sec:DR-LQG-algorithm}
Our duality results in \Cref{theorem:lower-equal-upper} lead to immediate algorithms for computing optimal strategies: the optimal value in the DRLQ problem is the same as the optimal value in problems~\eqref{eq: distributionally robust control problem -- affine simplified min max} and~\eqref{eq:dr-affine-max-min}, and the latter problems are finite-dimensional, convex-concave problems with smooth objectives, which are amenable to saddle-point methods \citep{Juditsky_Nem_2022,schiele2024disciplined}. However, such approaches would fail to exploit the temporal structure of the original control problem and would generally result in large-dimensional optimization problems.

We next leverage our results from~\S\ref{sec:Nash} to develop a set of more efficient algorithms to solve the DRLQ problem via Kalman filtering and DP techniques. Our algorithms will rely on all structural results from~\S\ref{sec:Nash}, as formalized in the next result.
\begin{corollary}[Solving DRLQ via Kalman Filtering]
\label{corol:worst-case-Kalman-elliptical}
Under Assumptions~\ref{ass:Gaussian}-\ref{ass:M-gradients}, the solution to the DRLQ problem~\eqref{eq:DRLQG} can be computed using the Kalman filtering techniques in Appendix~\S\ref{sec: appx: optimal-solution-classic-LQG} applied to a classical LQG model with distribution $\PP\opt$. Specifically, the optimal control policy is $u_t\opt = K_t\hat{x}_t$ for every $t \in [T-1]$, where $K_t$ is the optimal state-feedback gain given by~\eqref{eq:feedback-gain-Kt} and $\hat{x}_t$ is obtained using the Kalman filter recursions~\eqref{eq:mmse} corresponding to distribution $\PP\opt$.
\end{corollary}
This follows from the theorems in \S\ref{sec:Nash}, but it is worth emphasizing that \emph{all} those structural results are needed. Specifically, \Cref{thm:worst-case-mean-zero} implies that the optimal linear policy~$u\opt$ can be found by solving a \emph{classical} LQG problem corresponding to the (unknown, optimal) zero-mean Gaussian distribution~$\PP\opt$. However, to solve this problem with the Kalman filter recursions requires the covariance matrices of all noise terms $v_t$ under $\PP\opt$ to be positive \emph{definite}. Assumption~\ref{ass:Gaussian} only requires this for the nominal distribution, $\hat{\Sigma}_{v_t} \succ 0$, but it is \Cref{thm:optimal-covs-are-higher} that ensures that $\Sigma\opt_{v_t} \succeq \hat{\Sigma}_{v_t} \succ 0$.

We design an iterative algorithm to compute~$\PP\opt$. $\PP\opt$ is uniquely determined by the covariance matrices $\Sigma_w\opt \in \mathcal{M}_{\Sigma_w}$ and $\Sigma_v\opt \in \mathcal{M}_{\Sigma_v}$, chosen to maximize the objective in~\eqref{eq:dr-affine-max-min}. Moreover, we can replace $\mathcal{M}_{\Sigma_v}$ with $\mathcal{M}_{\Sigma_v}^+ = \{ \Sigma_{v_t} \in \mathcal M_{\Sigma_{v_t}} : \Sigma_{v_t} \succeq \lambda I \}$\footnote{For instance, $\lambda$ can be set as the minimum eigenvalue of $\hat{\Sigma}_v$.} for a small $\lambda > 0$, which is without loss of optimality due to~\Cref{thm:optimal-covs-are-higher}. 

We first reformulate~\eqref{eq:dr-affine-max-min} as
\begin{equation} 
    \max\limits_{\Sigma_w \in \mathcal{M}_{\Sigma_w}, \Sigma_v \in \mathcal{M}_{\Sigma_v}^+} f(\Sigma_w, \Sigma_v),
    \label{eq:dr-sdp-W-V}
\end{equation}
where~$f(\Sigma_w,\Sigma_v)$ denotes the optimal value of the classical LQG problem corresponding to the Gaussian distribution~$\PP$ with covariance matrices~$\Sigma_w$ and~$\Sigma_v$. We propose a Frank-Wolfe algorithm for solving problem~\eqref{eq:dr-sdp-W-V} (see \citealp{ref:frank1956algorithm, ref:levitin1966constrained}).  Because this requires certain smoothness properties for $f$~\citep{jaggi2013revisiting}, we first provide a structural result. 
\begin{proposition}
\label{prop:structure-of-f}
Under Assumptions~\ref{ass:Gaussian}-\ref{thm:optimal-covs-are-higher}, $f$ is concave and $\beta$-smooth on $\mathcal{M}_{\Sigma_w}\times \mathcal{M}_{\Sigma_v}^+$.
\end{proposition}
For a proof, see Appendix~\S\ref{sec:app:beta_smoothness}. This allows using a Frank-Wolfe algorithm to solve problem~\eqref{eq:dr-sdp-W-V}. Each iteration of this algorithm solves a direction-finding subproblem, that is, a variant of problem~\eqref{eq:dr-sdp-W-V} that maximizes the first-order Taylor expansion of~$f$ around the current iterates $(\Sigma^{(k)}_w,\Sigma^{(k)}_v)$:
\begin{equation} 
    \max\limits_{\Sigma_{w} \in \mathcal{M}_{\Sigma_w}, \Sigma_{v} \in \mathcal{M}^+_{\Sigma_v}} \bigl\langle \nabla_{\Sigma_w} f(\Sigma^{(k)}_w, \Sigma^{(k)}_v), \Sigma_{w} - \Sigma^{(k)}_w \bigr\rangle + \bigl \langle \nabla_{\Sigma_v} f(\Sigma^{(k)}_w, \Sigma^{(k)}_v), \Sigma_{v}- \Sigma^{(k)}_v \bigr \rangle.
    \label{eq:auxiliary-problem}
\end{equation}
The next iterates are obtained by moving towards a maximizer $(\Sigma\opt_{w}, \Sigma\opt_{v})$ of~\eqref{eq:auxiliary-problem}, i.e., we update
\[
    (\Sigma^{(k+1)}_w,\Sigma^{(k+1)}_v) \leftarrow (\Sigma^{(k)}_w,\Sigma^{(k)}_v)+\alpha \cdot (\Sigma\opt_{w}-\Sigma^{(k)}_w, \Sigma\opt_{v}-\Sigma^{(k)}_v),
\]
where~$\alpha$ is an appropriate step size. The proposed Frank-Wolfe algorithm enjoys a very low per-iteration complexity because problem~\eqref{eq:auxiliary-problem} is separable. To see this, we reformulate~\eqref{eq:auxiliary-problem} as:
\begin{equation} 
    \label{eq:linearization-oracle}
    \begin{array}{c@{}l@{}l}
    \displaystyle \max_{\{\Sigma_{z}\}_{z \in \mathcal{Z}}}~ & \displaystyle \sum_{z \in \mathcal{Z}} \bigl \langle \nabla_{\Sigma_{z}}f\bigl(\Sigma^{(k)}_w, \Sigma^{(k)}_v \bigr), \Sigma_{z} - \Sigma^{(k)}_{z} \bigr \rangle\\
     \st  & \Sigma_{z} \in \mathcal{M}_{\Sigma_z} ~\forall z \in \mathcal{Z} \\
     & \Sigma_z \in \mathcal{M}_{\Sigma_z}^+ ~ \forall z \in \{v_0, \ldots, v_{T-1}\}.
    \end{array}
\end{equation}
Hence, \eqref{eq:auxiliary-problem} decomposes into $|\mathcal{Z}| = 2T+1$ separate subproblems that can be solved in parallel. 

\begin{remark}
    \label{rem:efficient_computation_wasserstein_KL}
    Although the subproblems above can be reformulated as tractable SDPs that are amenable to off-the-shelf solvers, for specific divergences $\mathds{D}$ one may be able to further simplify this computation. For instance, for the Wasserstein and KL ambiguity sets from \S\ref{sec:examples} (corresponding to divergences $\mathds{W}$ and $\mathds{K}$, respectively), the optimization problem in~\eqref{eq:linearization-oracle} for a given $z \in \mathcal{Z}$ can be reduced to solving a univariate algebraic equation, which can be done to any desired accuracy~$\delta>0$ by an efficient bisection algorithm. Appendix~\S\ref{sec:bisection} provides details for this construction and a proof of all relevant results, that leverage existing results in the literature.
\end{remark}

\begin{remark}[Automatic differentiation]
    \label{rem:differentiation}
    Recall that~$f(\Sigma_w,\Sigma_v)$ is the optimal value of the LQG problem corresponding to the Gaussian distribution~$\PP$ with the covariance matrices~$\Sigma_w$ and~$\Sigma_v$. By using the underlying dynamic programming equations, $f(\Sigma_w,\Sigma_v)$ can thus be expressed in closed form as a serial composition of $\mathcal O(T)$ rational functions (see Appendix~\S\ref{sec: appx: optimal-solution-classic-LQG} for details). Hence, $\nabla_{\Sigma_z} f(\Sigma_w,\Sigma_v)$ can be calculated symbolically for any $z\in\mathcal Z$ by repeatedly applying the chain and product rules. However, the resulting formulas are lengthy and cumbersome.
    We thus compute the gradients numerically using backpropagation. The cost of evaluating $\nabla_{\Sigma_z} f(\Sigma_w,\Sigma_v)$ is then of the same order of magnitude as the cost of evaluating~$f(\Sigma_w,\Sigma_v)$.
\end{remark}
A detailed description of the proposed Frank-Wolfe method is given in~\Cref{alg:FW} below.

\begin{algorithm}[!h]
\caption{Frank-Wolfe algorithm for solving~\eqref{eq:dr-sdp-W-V}}
\begin{algorithmic}[1]
\State \textbf{Input:} initial iterates $\Sigma^{(0)}_w$, $\Sigma^{(0)}_v$, nominal covariance matrices~$\hat \Sigma_w$, $\hat \Sigma_v$, oracle precision~$\delta \in (0, 1)$
\State{{set}~initial iteration counter~$k =0$}
\While{stopping criterion is not met}
\State{\textbf{for}~$z\in\mathcal Z$~\textbf{do in parallel}} 
\State{\qquad compute $\nabla_{\Sigma_z} f(\Sigma^{(k)}_w, \Sigma^{(k)}_v)$}
\State{\qquad compute {$\Sigma_z\opt$} that solves~\eqref{eq:linearization-oracle} to precision $\delta$} \label{alg:line:line-search}
\State{\textbf{end}}
\State $(\Sigma^{(k+1)}_w, \Sigma^{(k+1)}_v) \leftarrow (\Sigma^{(k)}_w, \Sigma^{(k)}_v)+ 2/(2+k) \cdot (\Sigma_w\opt-\Sigma_w, \Sigma_v\opt-\Sigma_v)$
\State $k \gets k + 1$
\EndWhile\\
\textbf{Output}: $\Sigma^{(k)}_w$ and $\Sigma^{(k)}_v$
\end{algorithmic}
\label{alg:FW}
\end{algorithm}
By Theorem~1 and Lemma~7 in \citet{jaggi2013revisiting}, which apply in view of \Cref{prop:structure-of-f}, \Cref{alg:FW} attains a suboptimality gap of $\epsilon$ within $\mathcal{O}(1/\epsilon)$ iterations. Its precise computational complexity is critically dependent on the tractability of \Cref{alg:line:line-search} (this is very efficient for Wasserstein and KL ambiguity, by \Cref{rem:efficient_computation_wasserstein_KL}).

\section{Numerical Experiments}
\label{sec:numerics}
We conduct numerical experiments to assess the merits of the DRLQ model and the effectiveness of the algorithms introduced in \S\ref{sec:DR-LQG-algorithm}. We consider a class of dynamical systems with $n = m = p = d$ with $d\in \mathbb N_{+}$. We set ~$A_t=A$ where $A$ has $0.1$ on the main diagonal and the super-diagonal and zeroes elsewhere ($A_{i,j} = 0.1$ if $i=j$ or $i=j-1$ and $A_{i,j}=0$ otherwise), 
and the other matrices to $B_t= C_t= Q_t = R_t = I_d$. The nominal covariance matrices $\hat{\Sigma}_z$ are constructed randomly and with eigenvalues in the interval $[1,2]$ (to ensure they are positive definite).
We construct ambiguity sets based on either Wasserstein or KL divergence.

All experiments were conducted on an Apple M3 Max machine equipped with 64 GB of RAM. All linear SDP problems were formulated in Python (v3.8.6) using CVXPY (v1.6.1)~\citep{ref:agrawal2018rewriting, ref:diamond2016cvxpy} and solved with MOSEK~\citep{ref:mosek} (v11.0.8). 
Additionally, the gradients of $f(\Sigma_w, \Sigma_v)$ were computed using Pymanopt (v2.2.1)~\citep{ref:townsend2016pymanopt} in combination with PyTorch's automatic differentiation module (v2.6.0)~\citep{ref:paszke2017automatic, ref:NEURIPS2019_bdbca288}. 
The code is publicly available in the Github repository~{\url{https://github.com/BaharTaskesen/DRLQG}}.

\subsection{Computational Efficiency of Algorithm \ref{alg:FW}}
We compare two approaches for finding the optimal value of the DRLQ problem~\eqref{eq:DRLQG}: directly solving the SDP reformulation of~\eqref{eq:dr-affine-max-min} with MOSEK and using our custom Frank-Wolfe method discussed in \Cref{alg:FW}. For these experiments, we set $d=10$ and the corresponding radii of the ambiguity sets as $\rho_{x_0} = \rho_{w_t} = \rho_{v_t} = 10^{-1}$. We compare the two approaches in 10 problem instances (generated randomly and independently) and we compare performance as a function of the problem horizon $T$, which we vary. We set a stopping criterion corresponding to an optimality gap below~$10^{-3}$ and we run the Frank-Wolfe method with~$\delta = 0.95$. 

\Cref{fig:res-wass} and \Cref{fig:res-kl}, which correspond to the Wasserstein and KL ambiguity sets, respectively, depict the results. In both figures, the left panel illustrates the execution time for both approaches as a function of the planning horizon $T$ (omitting runs where MOSEK exceeds $100$s) and the right panel visualizes the empirical convergence behavior of the Frank-Wolfe algorithm. The results highlight that the Frank-Wolfe algorithm achieves running times that are uniformly lower than MOSEK across all problem horizons and is able to find highly accurate solutions already after a small number of iterations (50 iterations for problem instances with horizon~$T=10$).

\begin{figure}
    \centering
    \includegraphics[width=0.49\linewidth]{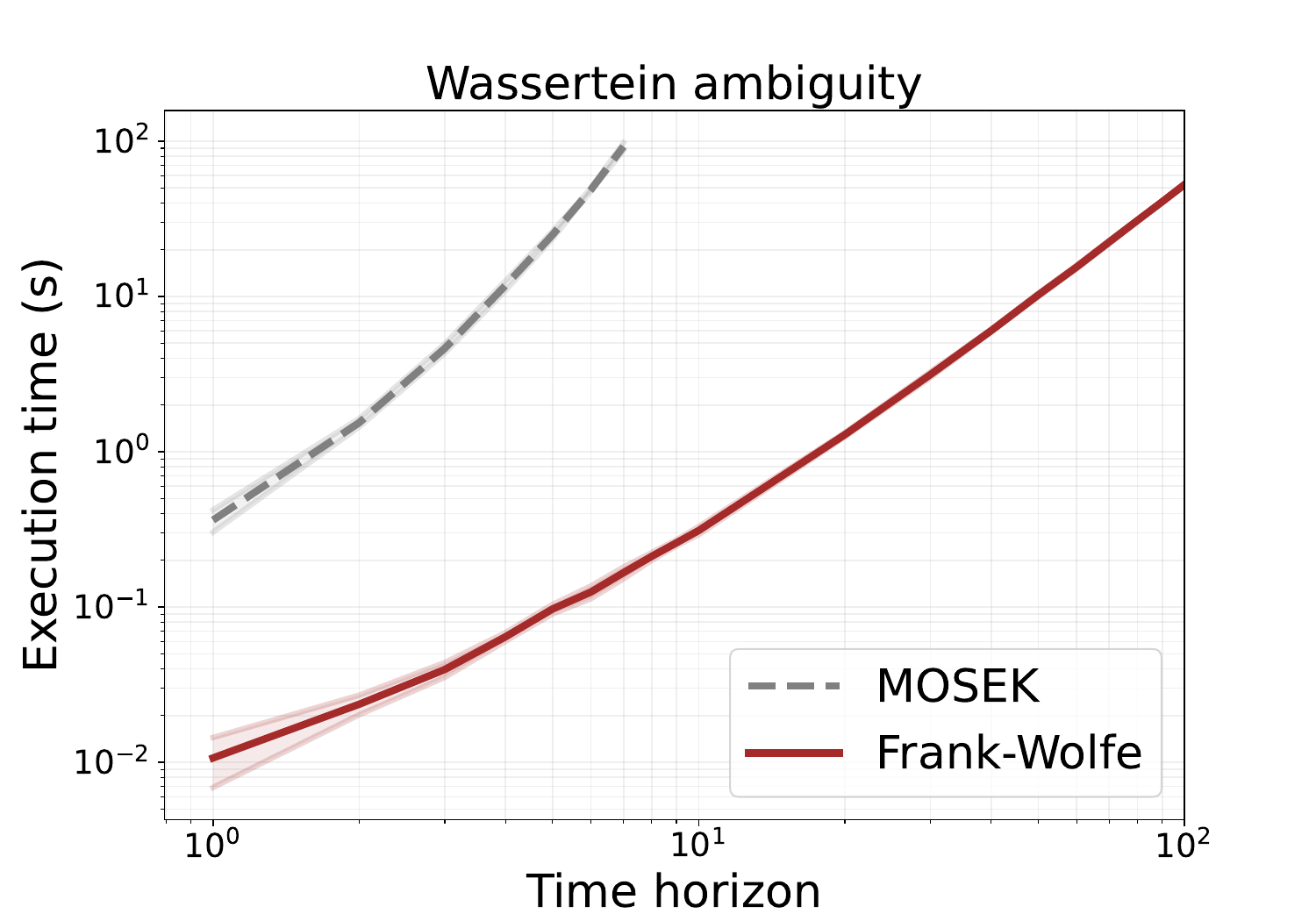} 
    \includegraphics[width=0.49\linewidth]{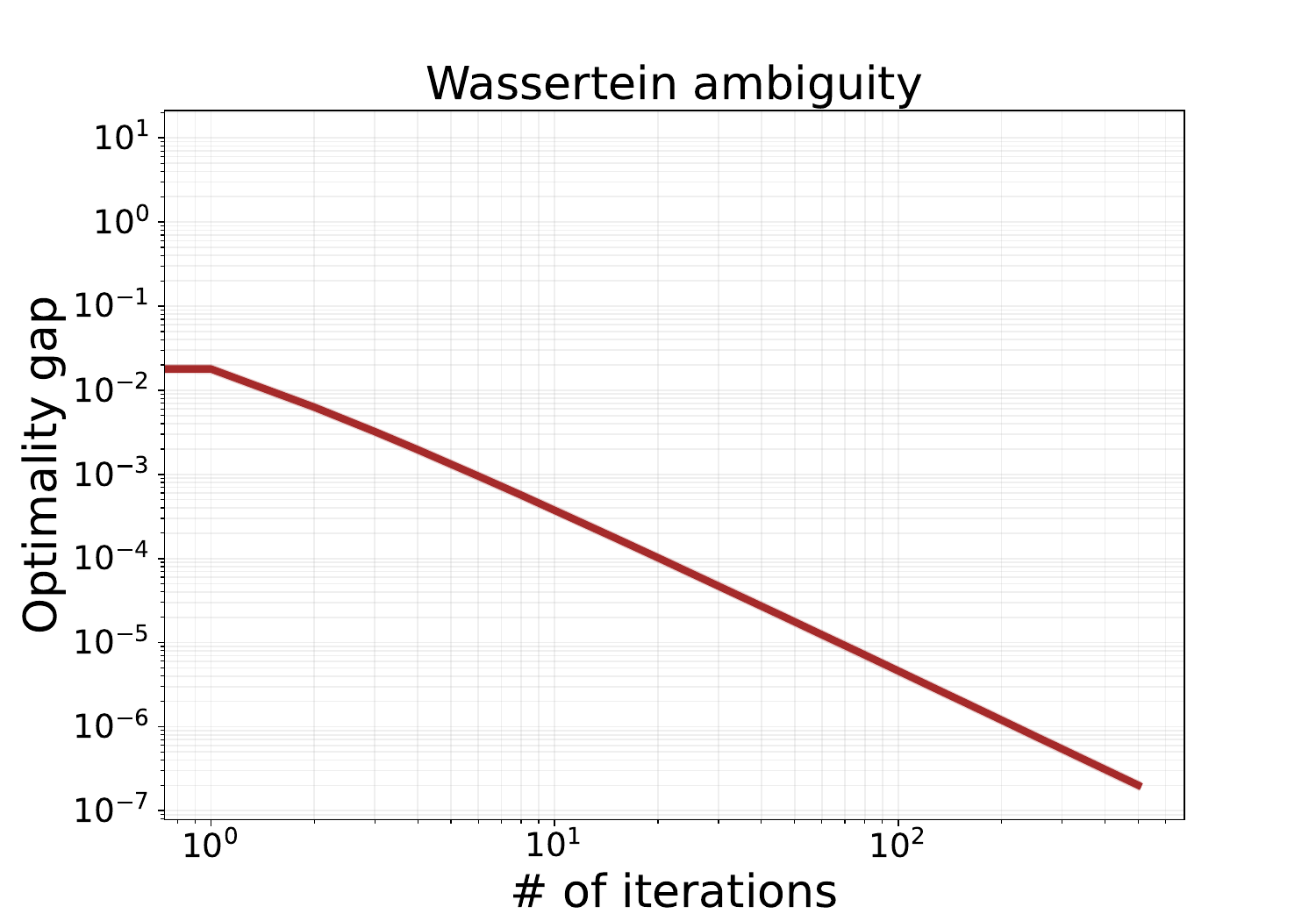}
\caption{Computation time (Left) and optimality gap (Right) of the Frank-Wolfe algorithm for Wasserstein ambiguity.}
\label{fig:res-wass}
\end{figure}

\begin{figure}
    \centering
    \includegraphics[width=0.49\linewidth]{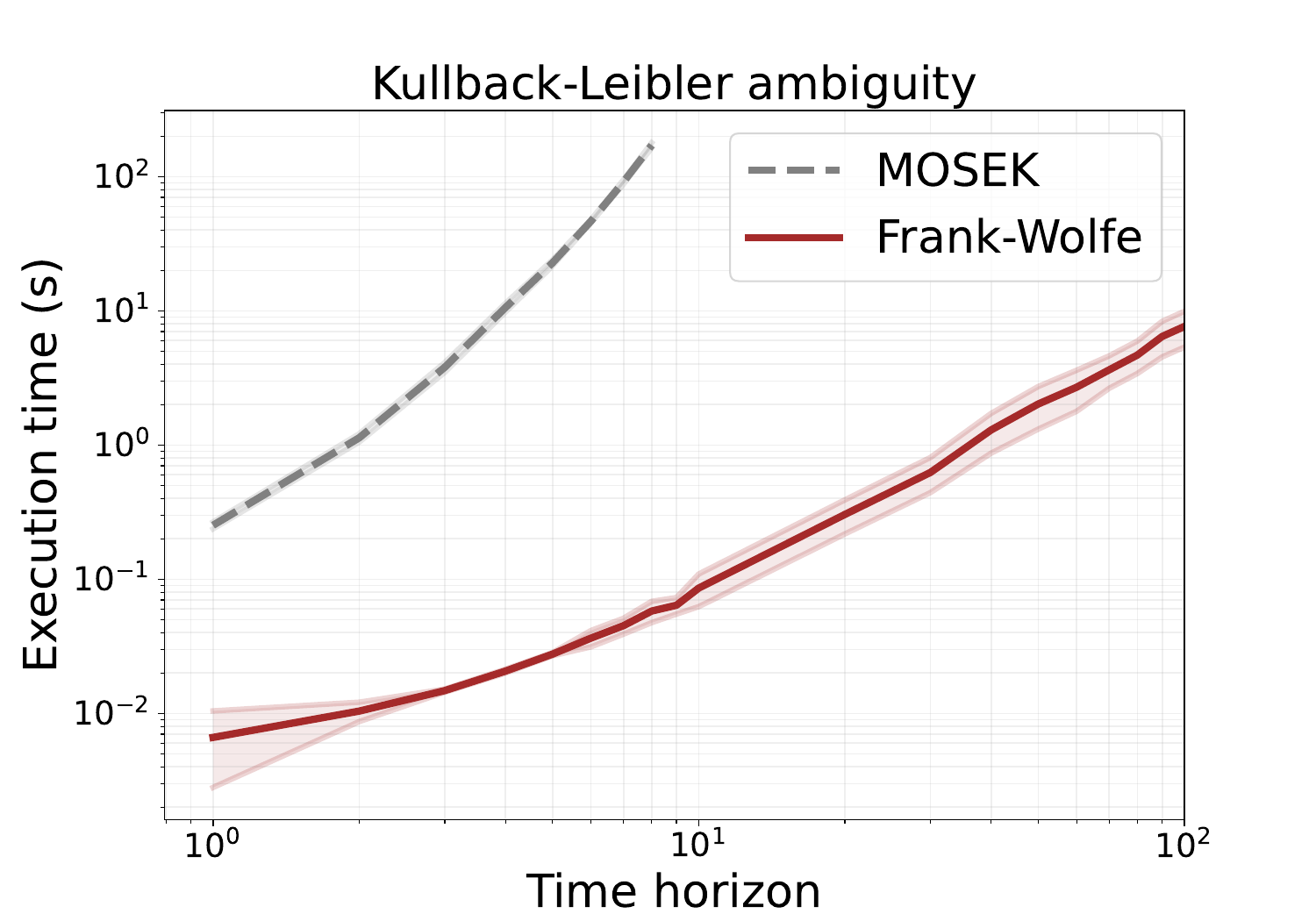}
    \includegraphics[width=0.49\linewidth]{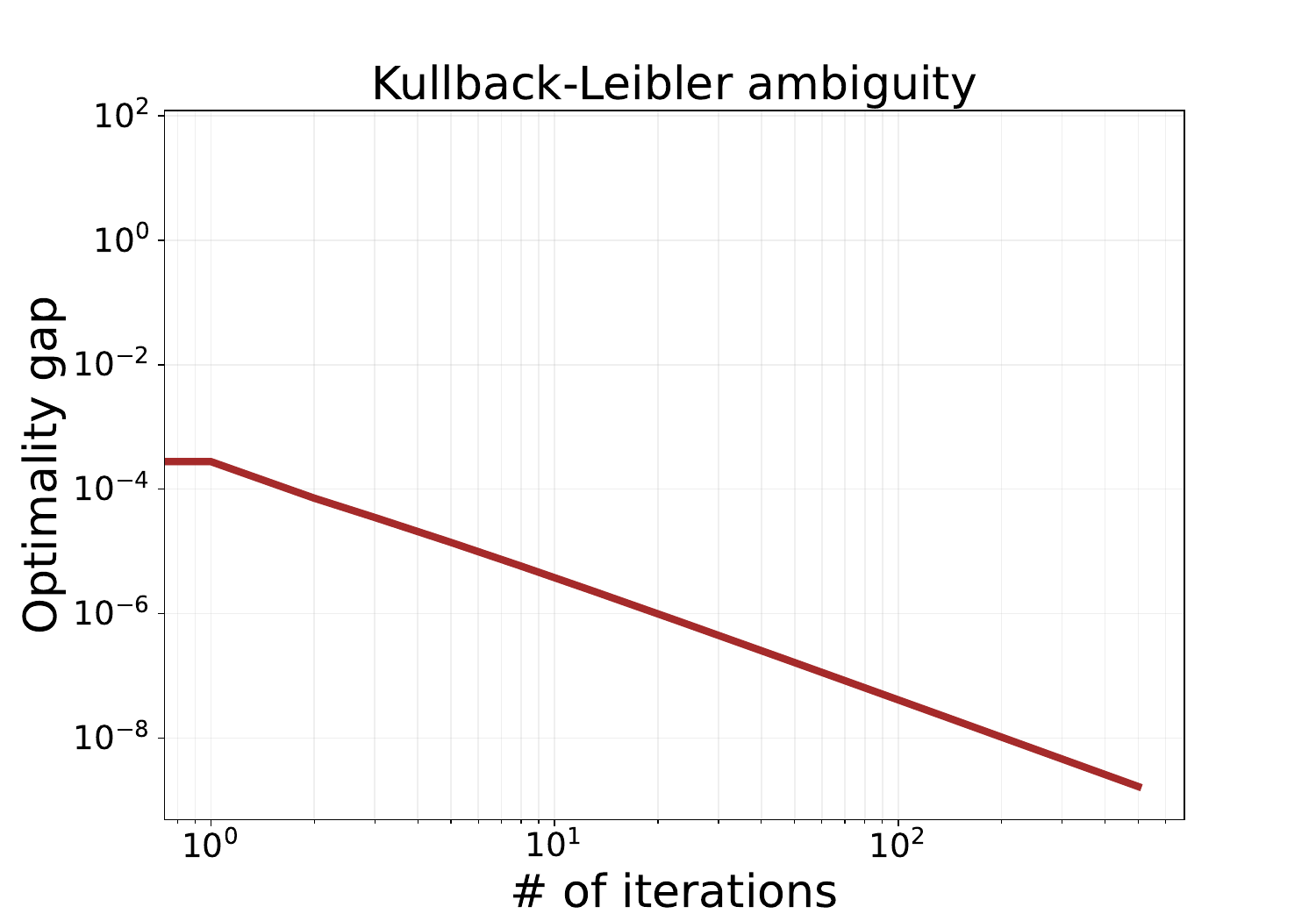}
    \caption{Computation time (Left) and optimality gap (Right) of the Frank-Wolfe algorithm for Kullback-Leibler ambiguity.}
\label{fig:res-kl}
\end{figure}

\subsection{Worst-case Performance}
We next evaluate the benefits of the robust approach. Let $u\opt$ denote the \emph{robustly optimal} policy, i.e., the optimal policy in the DRLQ model, obtained by solving problem~\eqref{eq:DRLQG}, and let $\hat{u}$ denote the \emph{nominally optimal} policy, i.e., the policy that minimizes the expected cost under the nominal distribution, $\EE_{\hat{\PP}}[J(u)]$. To gauge the robustness and conservativeness of the policies $u\opt$ and $\hat{u}$, we compare them under the nominal distribution and under their respective worst-case distributions. Specifically, letting $\PP\opt(u)$ denote the adversary's optimal choice of distribution $\PP$ corresponding to a control policy $u$ and letting $J(u)$ denote the dependency of the cost on the control policy $u$, we calculate the following two performance gaps:
\begin{align*}
    \textup{Worst-case-gap} &= \EE_{\PP\opt(\hat{u})}[J(\hat{u})] - \EE_{\PP\opt(u\opt)}[J(u \opt)], ~~ & 
    \textup{Nominal-gap} &= \EE_{\hat{\PP}}[J(u\opt)] - \EE_{\hat{\PP}}[J(\hat{u})].   
\end{align*}
In our experiments, we set $d=2$, $T=2$, $\rho_{x_0} = \rho_{w_t} = \rho_{v_t} = \rho$, and we vary the common radius $\rho$ from 0 to 10, calculating the ``Worst-case-gap" and ``Nominal-gap" for each value of $\rho$ in 10 independently generated random problem instances. The results are depicted in \Cref{fig:exp-costs}, with the left panel corresponding to the Wasserstein ambiguity set and the right panel corresponding to the KL ambiguity set.  The results indicate that using the optimal policy from the DRLQ model, $u\opt$, leads to dramatically lower worst-case costs, particularly as the ambiguity radius $\rho$ increases. Surprisingly, this improvement does not impact performance in the nominal scenario, where using $u\opt$ does not substantially increase costs relative to using the nominally optimal policy $\hat{u}$.
\begin{figure}
    \centering
    \includegraphics[width=0.45\linewidth]{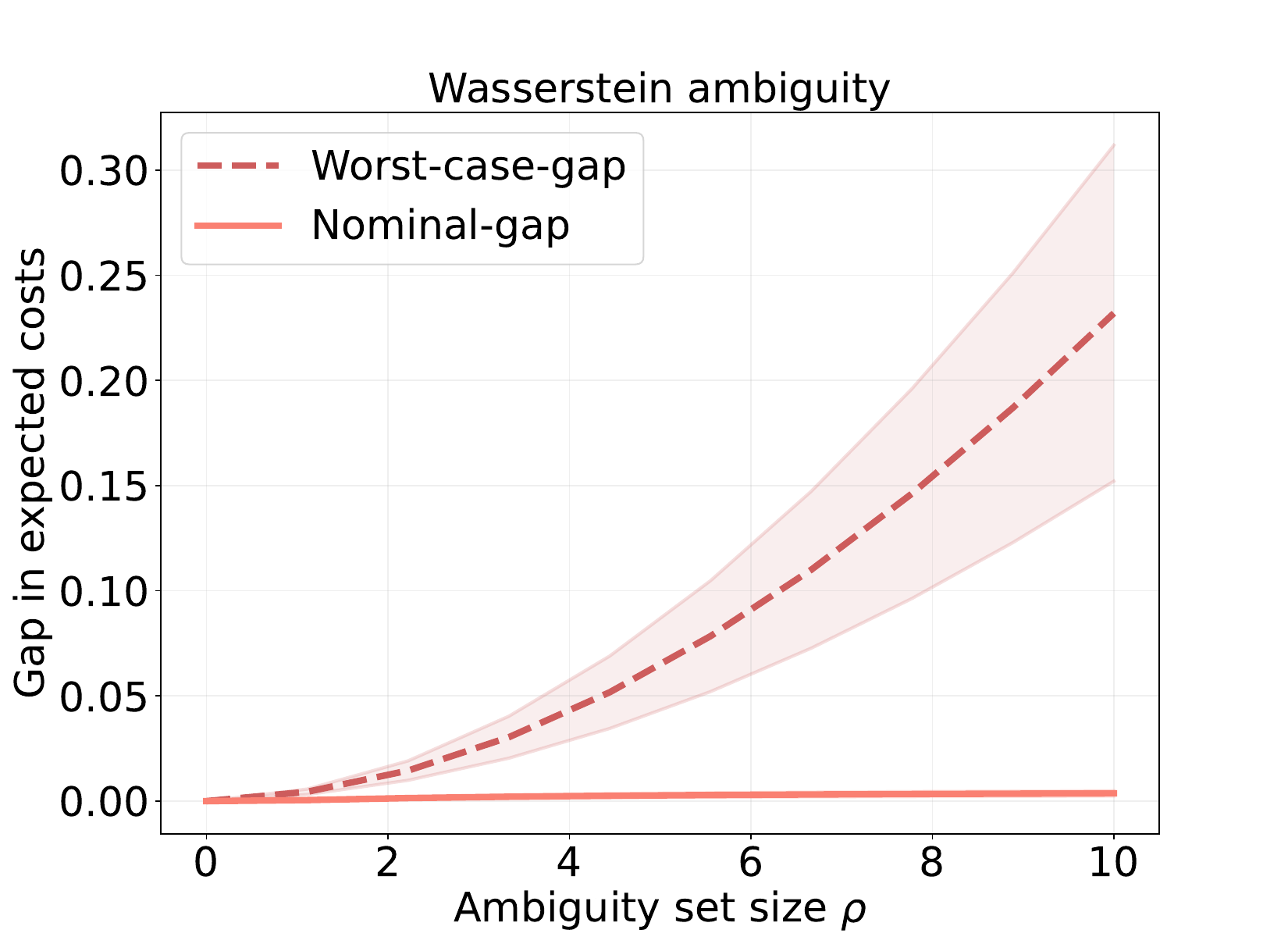}
    \includegraphics[width=0.45\linewidth]{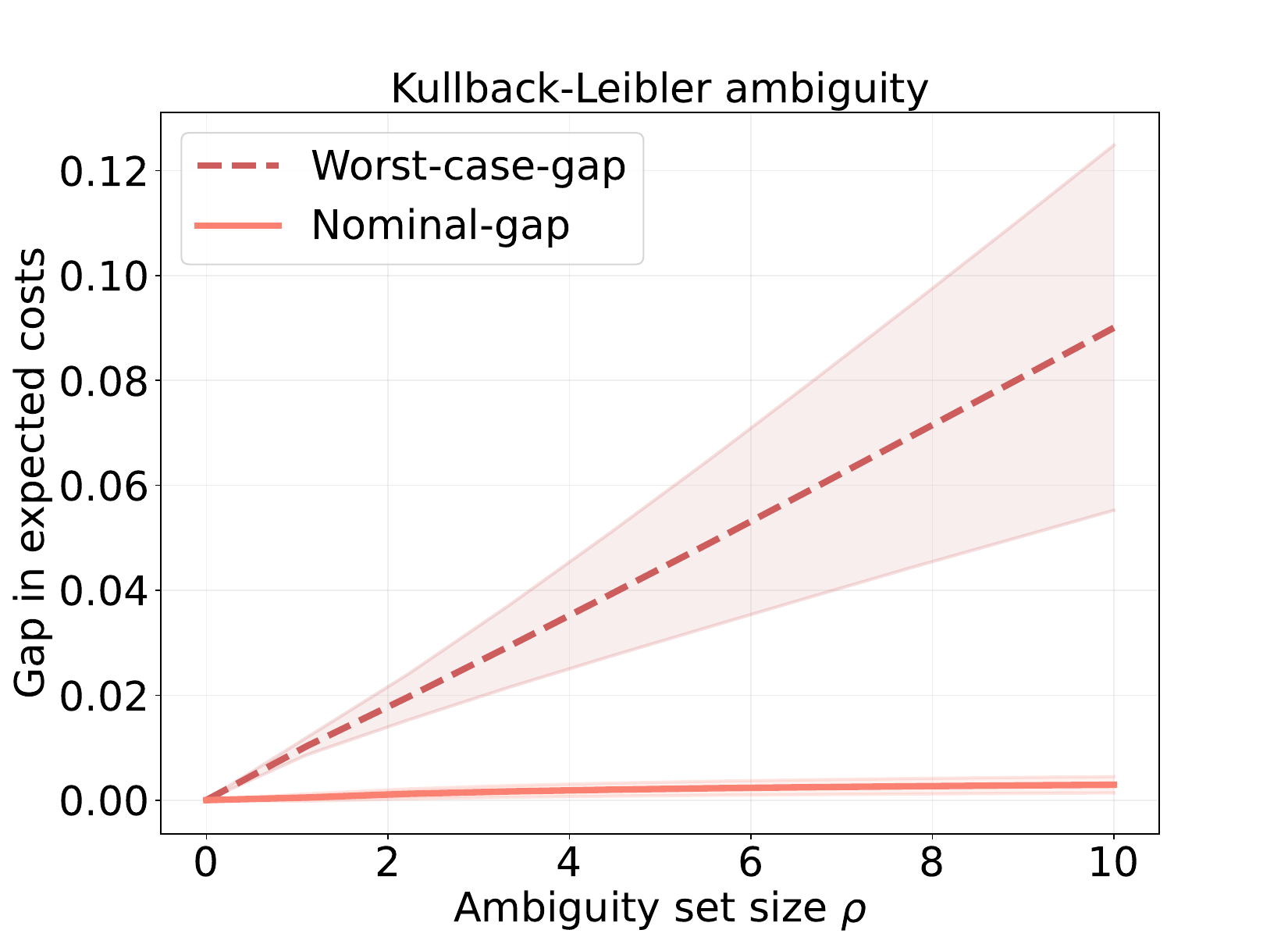}
    \caption{The difference in worst-case expected costs when using the nominally optimal policy $\hat{u}$ instead of the robustly optimal policy $u\opt$ (``Worst-case-gap") and the difference in expected costs under the nominal distribution $\hat{\PP}$ when using the robustly optimal policy $u\opt$ instead of the nominally optimal policy $\hat{u}$ (``Nominal-gap"), as a function of the radius of the ambiguity sets $\rho$, for Wasserstein ambiguity (Left) and KL ambiguity (Right). The solid line is the mean and the confidence bands correspond to one standard deviation, estimated from 10 independent runs.}
    \label{fig:exp-costs}
\end{figure}
\section{Infinite-Horizon DRLQ Problems}
\label{sec:infinite-horizon}
We now extend the results of \S\ref{sec:Nash} to infinite-horizon control problems with an average cost criterion. Throughout this section, we restrict attention to linear \emph{time-invariant} systems of the form \eqref{eq:dynamics} and \eqref{eq:observation} where~$T=\infty$ and $A_t=A_0$, $B_t=B_0$ and $C_t=C_0$ for all $t\in\mathbb N$. All random variables emerging in our model are functions of the initial state $x_0$ and the noise terms $\{w_t\}_{t=0}^\infty$ and $\{v_t\}_{t=0}^\infty$. Therefore, we set the sample space to $\Omega =\R^n\times (\times_{t=0}^\infty (\R^n\times \R^p))$ and equip it with its product Borel $\sigma$-algebra~$\mathcal F$. In analogy to the finite-horizon theory, we define $x=(x_t)_{t=0}^\infty$, $u=(u_t)_{t=0}^\infty$, $y=(y_t)_{t=0}^\infty$, $w=(x_0, (w_t)_{t=0}^\infty)$ and $v=(v_t)_{t=0}^\infty$ as well as infinite-dimensional block matrices $H$, $G$ and~$C$ (whose definitions mirror those for the finite horizon and are omitted for brevity). With these conventions, the input, state and output processes are subject to the usual system equations $x=Hu+Gw$ and $y=Cx+v$. As before, we denote by $\mathcal{U}_y$ the set of control inputs $u$ such that $u_t = \varphi_t(y_0,\ldots,y_t)$ for every~$t \in \mathbb N$, where~$\varphi_t:\R^{p(t+1)}\rightarrow\R^m$ is a measurable control policy. To describe the distribution of the exogenous uncertainties, it is again convenient to set $\mathcal Z=\{x_0,w_0, v_0, w_1,v_1,  \ldots \}$. We then define~$\mathcal{B}^\infty$ exactly as the ambiguity set~$\mathcal{B}$ from \S\ref{sec:problem-definition}.

For costs, we take $Q_t=Q_0$ and $R_t=R_0$ for all $t\in\mathbb N$, where $Q_0 \in \mathbb S_+^n$ and $R_0 \in \mathbb S_{++}^m$. We study an infinite-horizon DRLQ problem that minimizes the worst-case, long-run average cost: 
\[
    J_\PP(x,u) = \limsup\limits_{T\to \infty} \frac{1}{T} \sum\limits_{t=0}^{T-1} \EE_{\PP}\left[ x_t^\top Q_0 x_t + u_t^\top R_0 u_t \right],
\]
where the worst case is taken over all~$\PP\in \mathcal{B}^\infty$. 

%
%
\subsection{Assumptions}
\label{sec:assumptions_inf_horizon}
To ensure that the problem is well posed under an infinite-horizon setting and average-cost criterion, we must slightly strengthen the assumptions in~\S\ref{subsec:assumptions} and make a few additional mild assumptions. 

First, we mirror the classical infinite-horizon LQG setting by stating an assumption on system and cost matrices. Recall that we already focus on time-invariant systems and costs. To state our additional requirement, we recall the following definition of a Schur-stable matrix.
%
%
\begin{definition}
A square matrix $A \in \R^{n\times n}$ is said to be \emph{Schur} stable if its eigenvalues are strictly within the unit circle in $\mathbb C$. 
\end{definition}
If $A$ is a Schur stable matrix, then $A^t\to 0 $ as $t\to \infty$ and the system defined through the dynamics $z_{t+1} = A z_t$ for $t\in \mathbb N_{+}$ is stable, that is, $z_t \to 0$ as $t \to \infty$ \cite[Page 135]{Bertsekas_2017}. 

Subsequently, we impose the following assumption on the system and cost matrices.
%
%
\begin{assumption} 
\label{ass:stabilizability-detectibility}
The matrices $A_0, B_0, C_0$ and $Q_0$ satisfy the following properties.
    \begin{enumerate}[label=(\roman*)]
        \item $(A_0,B_0)$ is stabilizable, i.e., there exists $K \in \R^{m\times n}$ such that $A_0 + B_0 K$ is Schur stable.
        \item $(A_0,C_0)$ is detectable, i.e., there exists $L \in \R^{n\times p}$ such that $A_0 - LC_0$ is Schur stable. 
        \item $Q_0 \succ 0$.
    \end{enumerate}
\end{assumption}
These assumptions are standard in the context of infinite-horizon control of linear dynamical systems (see \citealp{ref:lancaster1995algebraic,Bertsekas_2017}, and our brief review in \S\ref{sec:classical-lqg-inf-horizon}). \emph{Stabilizability} guarantees the existence of a stationary state-feedback control policy that makes the states of a noise-free system converge and allows deriving an optimal state-feedback control policy of the form $u_t\opt = K \hat{x}_t$, where $\hat{x}_t$ denotes the MMSE state estimator (see \S\ref{sec:classical-lqg-inf-horizon}).  However, even if~$u_t$ stabilizes the noise-free system, the states under linear (purified) output feedback could diverge for large~$t$ without the \emph{detectability} assumption and without a strictly positive definite cost matrix $Q_0$.\footnote{This suggests that the stability properties of a system are more difficult to analyze when the control policy is parametrized in terms of the purified outputs. In contrast, the convexity properties of the system are more difficult to analyze when policies are parametrized in terms of the original outputs.}

We also strengthen slightly our assumptions concerning the nominal distribution $\hat{\PP}$ by requiring this to be zero-mean Gaussian and also time-invariant.
\begin{definition}
    We call a probability distribution $\PP\in\mathcal{B}^\infty$ {\em time-invariant} if there exist $\Sigma_w\in\mathbb S_+^{n}$ and $\Sigma_v\in\mathbb S_+^{p}$ such that $\EE_\PP[x_0 x_0^\top] =\Sigma_w$ and $\EE_\PP[w_t w_t^\top] =\Sigma_w$ and $\EE_\PP[v_t v_t^\top] =\Sigma_v$ for all $t\in\mathbb N$. 
\end{definition}
Subsequently, we use $\mathcal{B}^\infty_\mathcal{N}$ to denote the subset of all time-invariant Gaussian distributions in~$\mathcal{B}^\infty$. Throughout this section, we therefore replace Assumption~\ref{ass:Gaussian} with the following assumption.
\begin{assumption}
\label{ass:nominal_inf_horizon}
The nominal distribution $\hat{\PP}$ and the ambiguity set $\mathcal{B}^\infty$ satisfy the requirements:\\
(i) $\hat{\PP}$ is a time-invariant Gaussian distribution, $\hat{\PP} \in \mathcal{B}^\infty_\mathcal{N}$, with $\hat{\mu}_z = 0$ and $\hat{\Sigma}_{z} \succ 0$ for all $z \in \mathcal{Z}$;\\
(ii) $\rho_z = \rho_w$ for all $z \in \{x_0\} \cup \{w_t\}_{t=0}^\infty$ and $\rho_z = \rho_v$ for all $z \in \{v_s\}_{t=0}^\infty$;\\
(iii) all distributions in $\mathcal{B}^\infty$ have zero mean, $\EE_{\PP_z}[z]=0$ for every $z\in\mathcal Z$ and every $\PP \in \mathcal{B}^\infty$.
\end{assumption}
The requirements in~(i) concerning the nominal distribution are standard in infinite-horizon LQG control problems \citep{ref:lancaster1995algebraic,Bertsekas_2017}. Under a time-invariant, Gaussian noise distribution $\hat{\PP}$, these requirements are only slightly stronger than those in Assumption~\ref{ass:Gaussian}, by asking that the covariance matrices for the state noise be positive definite, $\hat{\Sigma}_{x_0} \succ 0$ and $\hat{\Sigma}_{w_t} \succ 0$ for all $t \in \mathbb{N}$. Requirement~(ii) is also aligned with time-invariance by asking that the corresponding ambiguity sets have the same radius. Lastly, the restriction to zero-mean distributions in~(iii) simplifies exposition and is driven by our results in \S\ref{sec:sec3_optimality_linear_zero_mean}; this requirement can be relaxed and arguments mirroring those in \S\ref{sec:sec3_optimality_linear_zero_mean} can be used to prove that nature will optimally choose zero-mean distributions, but we omit details for brevity and instead simply state this as a requirement. 

Throughout this section, Assumption~\ref{ass:general_assumption_about_M_function} remains unchanged. Note that in view of~Assumption~\ref{ass:nominal_inf_horizon}, the sets of moments defined in Assumption~\ref{ass:general_assumption_about_M_function} only involve the covariance matrices $\Sigma_z$ (as in \S\ref{sec:sec3_worst_case_covariance} and \S\ref{sec:DR-LQG-algorithm}) and are also time-invariant, so we define the following simpler notation:
\begin{align}   
    \begin{aligned}
        \mathcal{M}_{\Sigma_w} &= \{\Sigma \in \mathbb{S}_+^{n} : \mathds D(\mathcal N(0, \Sigma), \hat{\PP}_z) \leq  \rho_w \}, &&\forall \, z \in \{x_0\} \cup \{w_t\}_{t=0}^\infty \\
        \mathcal{M}_{\Sigma_v} &= \{\Sigma \in \mathbb{S}_+^{p} : \mathds D(\mathcal N(0, \Sigma), \hat{\PP}_z) \leq  \rho_v \}, && \forall \, z \in \{v_s\}_{t=0}^\infty.
    \end{aligned}
    \label{eq:definitions_Mw_Mv_inf_horizon}
\end{align}
By Assumption~\ref{ass:general_assumption_about_M_function}, the sets  $\mathcal{M}_{\Sigma_{w}}$ and $\mathcal{M}_{\Sigma_{v}}$ are convex and compact.

Lastly, we preserve Assumption~\ref{ass:M-gradients} suitably generalized to our infinite-horizon setting, so that
\begin{align}
    \mathcal{M}_{\Sigma_z} = \{ \Sigma \in \mathbb{S}_+^{d_z} : \gz(\Sigma_z) \leq 0 \}, ~ \forall z \in \mathcal{Z}.
\end{align}
where $\gz$ is a convex, differentiable function.

%
%
\subsection{Construction of Primal and Dual and Their Bounds}
With these preliminaries, the DRLQ problem can be formulated as:
\begin{equation*}
     p\opt = \left\{
    \begin{array}{c@{~~}l}
         \inf\limits_{x,u,y} & \sup\limits_{\PP\in \mathcal{B}^\infty} ~J_\PP(x,u) 
         \\
         \st & u \in \mathcal U_y,~\, x = Hu + Gw, ~\,y = Cx + v.
    \end{array}\right.
\end{equation*}
In analogy to \S\ref{sec:Nash}, we define $\eta=(\eta_t)_{t=0}^\infty$ as the trajectory of all purified observations that satisfies $\eta=Dw+v$, where $D=CG$, and~$\mathcal U_\eta$ as the set of all control inputs $u$ so that $u_t = \phi_t(\eta_0, \dots,   \eta_{t})$ for some measurable function~$\phi_t : \R^{p(t+1)}\rightarrow\R^m$ for every~$t\in \mathbb N_+$. 
By \cite[Proposition~II.1]{ref:hadjiyiannis2011affine}, we can rewrite the infinite-horizon DRLQ problem equivalently as
\begin{equation}
    \label{eq:inf-primal-1}
     p\opt = \left\{
    \begin{array}{c@{~~}l}
         \inf\limits_{x,u} & \sup\limits_{\PP\in \mathcal{B}^\infty} ~ J_\PP(x,u) \\
         \st & u \in \mathcal U_\eta,~\, x = Hu + Gw.
    \end{array}\right.
\end{equation}
Our results also rely on the dual of \eqref{eq:inf-primal-1} defined as
\begin{equation}
    \label{eq:inf-dual-1}
    d\opt = \left\{
    \begin{array}{cclllll}
         &\sup\limits_{\PP\in \mathcal{B}^\infty} &\inf\limits_{x,u} & J_\PP(x,u) \\
         &&\st & u \in \mathcal U_\eta,~\, x = Hu + Gw.
    \end{array}\right.
\end{equation}
In the remainder of this section, we demonstrate that, under mild assumptions, the primal DRLQ problem is solved by a \emph{stationary, linear} control policy, the dual problem is solved by a \emph{time-invariant, Gaussian} distribution, and strong duality holds. To establish these results, we proceed as in \S\ref{sec:Nash}: we first construct an upper bound for the primal problem~\eqref{eq:inf-primal-1} followed by a lower bound for the dual problem~\eqref{eq:inf-dual-1}, and then show that the bounds coincide, which proves all three claims.

\subsection*{Upper Bound for Primal}
Mirroring \S\ref{sec:Nash}, we obtain an upper bound on $p\opt$ by inflating the ambiguity set $\mathcal{B}^\infty$ and restricting the control policies. We define the inflated ambiguity set~$\overline{\mathcal{B}}{}^\infty$ exactly as the ambiguity set~$\overline{\mathcal{B}}$ from \Cref{sec:Nash}, but we also add the additional condition that~$\EE_{\PP_z}[z]=0$ for every~$z\in\mathcal Z$ to mirror the new Assumption~\ref{ass:nominal_inf_horizon}-(ii) on $\mathcal{B}^\infty$. For any divergence~$\mathds{D}$ satisfying Assumption~\ref{ass:general_assumption_about_M_function}-(i), we can readily verify that~$\mathcal{B}^\infty$ is a subset of~$\overline{\mathcal{B}}{}^\infty$. Restricting the control policies $u$ requires more care in the infinite-horizon case, to ensure that long-run-average costs are finite and to enable our subsequent duality proof. To that end, we restrict attention (without loss of optimality) to a \emph{subset} of the set of all \emph{stationary} purified output control policies $u \in \mathcal{U}_\eta$ under which the covariance matrices of controls $u_t$ and states $x_t$ converge, which also guarantees that the long-run average costs converge to a finite value. 
To formalize these, we first define the set of block lower triangular Toeplitz matrices.
\begin{definition}[Toeplitz Matrices]
An infinitely long block matrix~$M$ with blocks of size $k\times l$ is called a block lower triangular Toeplitz matrix if there exist $M_t\in\R^{k\times l}$ for all $t\in\mathbb N$ so that
\[
    M = \begin{bmatrix}
    M_0 & & & \\
    M_1 & M_0 &  &  \\
    M_2 & M_1 & M_0 & \\
    \vdots & &  & \ddots
    \end{bmatrix}.
\]
All blocks of the $t$-th subdiagonal of $M$ are given by the same matrix~$M_{t}\in \R^{k\times l}$, so $M$ can be uniquely identified by its blocks $\{M_{t}\}_{t=0}^\infty$. The family of all block lower triangular Toeplitz matrices with blocks of size $k\times l$ is denoted by $\mathcal T^{k\times l}$. We also define the norm of an infinitely long Toeplitz matrix $M\in\mathcal T^{k\times l}$ as $\|M\|_{\mathcal T} = (\sum_{t=0}^\infty \|M_t\|_{\rm F}^2)^{1/2}$, where $\|M_t\|_{\rm F}$ stands for the Frobenius norm of $M_t$.
\end{definition}
The following lemma summarizes useful structural properties of Toeplitz matrices that will enable us to write compact expressions and study the properties of our control policies.
\begin{lemma}
\label{lem:low-tri-toep-set} The following properties hold for block lower triangular Toeplitz matrices:
\begin{enumerate}
    \item[i)] If $M\in \mathcal T^{k\times l}$ and $N\in\mathcal T^{k\times l}$, then $O=M+N\in \mathcal T^{k\times l}$ with $O_t=M_t+N_t\in\R^{k\times l}$ for all $t\in\mathbb N$.
    \item[ii)] If $M\in \mathcal T^{k\times l}$ and $N\in\mathcal T^{l\times m}$, then $O=MN\in \mathcal T^{k\times m}$ with $O_t= \sum_{s=0}^{t} M_s N_{t-s} \in\R^{k\times m}$ for all $t\in\mathbb N$.
       \item[iii)] If $M\in \mathcal T^{k\times k}$ with $M_0$ invertible, then $N=M^{-1}\in\mathcal T^{k\times k}$, and its blocks obey the recursion
    \begin{align}
        \label{eq:toeplitz-recursion}
        N_0=M_0^{-1} \quad\text{and} \quad N_t=-M_0^{-1}\sum_{s=1}^{t} M_s N_{t-s} \quad \forall t\in\mathbb N.
    \end{align}
\end{enumerate}
\end{lemma}
Subsequently, for a Toeplitz matrix $M$ obtained by adding or multiplying Toeplitz matrices, we use the $(M)_t$ to denote the block matrix on the $t$-th subdiagonal of $M$.

With this preparation, we can formally define the class of \emph{stationary} control policies.
\begin{definition}
    We call a linear output feedback policy $u\in\mathcal U_y$ {\em stationary} if there exists $U'\in\mathcal T^{m\times p}$ such that $u=U'y$. Similarly, we call a linear {\em purified} output feedback policy $u\in\mathcal U_\eta$ {\em stationary} if there exists $U\in\mathcal T^{m\times p}$ such that $u=U\eta$. 
\end{definition}
Subsequently, we focus on \emph{stationary, linear, purified output} feedback policies, i.e., $u\in \mathcal{U}_\eta$ such that $u = U\eta$ for some $U\in\mathcal T^{m\times p}$. (Through a straightforward extension of results in \Cref{lemma:linear-rel-u-eta}, one can verify that these are equivalent to stationary linear output feedback policies.) Under such policies, the following result -- which leverages \Cref{lem:low-tri-toep-set} -- yields compact expressions for several important quantities.
\begin{lemma}
\label{lem:x-t-u-t-in-w-v} 
Consider $u \in \mathcal{U}_\eta$ such that $u = U\eta$ for $U\in\mathcal T^{m\times p}$. For any $t \in \mathbb N$, we have:
\begin{subequations}   
    \begin{align}
        & u_t = \sum_{s=0}^t \left[ (U D)_{t-s} w_s + (U) _{t-s} v_s \right] \quad \text{and} \quad x_t = \sum_{s=0}^t \left[ (G + H U D) _{t-s} w_s + (H U)_{t-s} v_s \right] \label{eq:controls_states_inf_horizon}\\
        & \Sigma_{u_t} = \EE_\PP\left[u_tu_t^\top\right] = \sum_{s=0}^t \left( (U D)_{t-s} \Sigma_{w_s} (U D)_{t-s}^\top + (U) _{t-s} \Sigma_{v_s} (U) _{t-s}^\top\right) \label{eq:cov_controls_inf_horizon} \\
        & \Sigma_{x_t} = \EE_\PP \left[x_tx_t^\top \right] = \sum_{s=0}^t \left( (G + H U D) _{t-s} \Sigma_{w_s} (G + H U D) _{t-s}^\top + (H U)_{t-s} \Sigma_{v_s} (H U)_{t-s}^\top \right) \label{eq:cov_states_inf_horizon} \\
        & \frac{1}{T} \sum_{t=0}^{T-1} \EE_{\PP}\left[ x_t^\top Q_0 x_t + u_t^\top R_0 u_t \right] = \frac{1}{T} \sum_{s=0}^{T-1} \Tr\left(\Sigma_{w_s} \left(\sum_{t=0}^{T-1-s} M_{t}\right) + \Sigma_{v_s} \left(\sum_{t=0}^{T-1-s} N_{t} \right) \right), \label{eq:exp_avg_run_cost_inf_horizon}
    \end{align}
\end{subequations}
where
\[
    M_t = (UD)_t^\top R_0 (UD)_t + (G+HUD)_t^\top Q_0 (G+HUD)_t \quad \text{and} \quad N_t = (U)_t^\top R_0 (U)_t + (HU)_t^\top Q_0 (HU)_t.
\]
\end{lemma}

The expressions in \Cref{lem:x-t-u-t-in-w-v} for the covariances $\Sigma_{u_t},\Sigma_{x_t}$ and the average cost reveal that additional requirements are needed on the stationary control policy $u$ to ensure convergence and finite costs. To that end, in deriving the upper bound on $p\opt$, we further restrict the control policies in the feasible set of problem~\eqref{eq:inf-primal-1} to stationary linear policies $u=U\eta$ where~$U$ belongs to the set:
\begin{align*}
    \mathcal{U}_\infty = \left\{ U \in \mathcal T^{m\times p} : \|U\|_{\mathcal T}<\infty,  \|U D\|_{\mathcal T}< \infty,  \| G +  HU D\|_{\mathcal T} <\infty,  \| HU\|_{\mathcal T}< \infty \right\}.
\end{align*}
Note that $\mathcal U_\infty$ is a convex set because $\mathcal T^{m\times p}$ is a linear space and the norm $\|\cdot\|_\mathcal T$ is convex.

With this, we now formalize the upper bound on $p\opt$ by inflating the ambiguity set to $\overline{\mathcal{B}}{}^\infty$ and using $\mathcal U_\infty$ to construct the stationary policies. We obtain the following upper bound:
\begin{equation}
\overline p^\star=\left\{
    \begin{array}{cclll}
     &\min\limits_{U,x,u} &\max\limits_{\mathbb{P} \in \overline{\mathcal{B}}{}^\infty} ~ J_\PP(x,u) 
     \\
    &\st & U \in \mathcal U_\infty,  u = U(Dw + v),  x = H u + G w.
    \end{array}\right.
    \label{eq:upper-bound-infty}
\end{equation}

The following result provides a refinement for the construction by showing that we can restrict attention to time-invariant, Gaussian distributions in $\overline{\mathcal{B}}{}^\infty$ without any optimality loss.
\begin{proposition}
\label{prop:inf-optimal-distributions}
Under Assumptions~\ref{ass:general_assumption_about_M_function},~\ref{ass:stabilizability-detectibility}-\ref{ass:nominal_inf_horizon}, with $u = U(Dw + v)$ for $U\in \mathcal U_\infty$ and $x = H u + G w$, the inner maximization problem in~\eqref{eq:upper-bound-infty} is solved by a time-invariant Gaussian distribution $\PP^\star\in\mathcal{B}^\infty_{\mathcal{N}}$.
\end{proposition}
\Cref{prop:inf-optimal-distributions} implies that the feasible set of the inner maximization problem in~\eqref{eq:upper-bound-infty} can be restricted to $\mathcal{B}^\infty_{\mathcal{N}}$ without any optimality loss. Because any distribution~$\PP\in \mathcal{B}^\infty_{\mathcal{N}}$ is uniquely determined by the covariance matrices $\Sigma_w$ and $\Sigma_v$, the upper bounding problem~\eqref{eq:upper-bound-infty} is equivalent to the simplified minimax problem: 
\begin{equation}
\overline p^\star=
    \min\limits_{U\in\mathcal U_\infty} \max\limits_{\Sigma_w\in \mathcal{M}_{\Sigma_w}, \Sigma_v \in \mathcal{M}_{\Sigma_v}} ~ J(U;\Sigma_w,\Sigma_v)  
    \label{eq:upper-bound-infty-reformulation}
\end{equation}
with objective function given by
\[
    J(U;\Sigma_w,\Sigma_v) = J_\PP(HU\eta+Gw, U\eta),
\]
and $\PP=\PP_{x_0} \otimes ( \otimes_{t=0}^\infty (\PP_{w_t} \otimes \PP_{v_t}))$, with $\PP_{x_0}=\PP_{w_t} =\mathcal N(0,\Sigma_w)$ and  $\PP_{v_t} =\mathcal N(0,\Sigma_v)$ for all $t\in\mathbb N$.

Moreover, mirroring the developments in \S\ref{sec:sec3_worst_case_covariance}, we can further restrict the feasible sets in the inner maximization in~\eqref{eq:upper-bound-infty-reformulation} to only contain covariance matrices that dominate the nominal covariance matrices in Loewner order. More formally, if we define the sets 
\[ \mathcal{M}_{\Sigma_w}^+ = \{\Sigma_w \in \mathcal{M}_{\Sigma_w} : \Sigma_w \succeq \hat\Sigma_w\} ~ \mbox{and}~  \mathcal{M}_{\Sigma_v}^+ = \{\Sigma_v  \in  \mathcal{M}_{\Sigma_v} : \Sigma_v\succeq \hat\Sigma_v\},\] 
then a straightforward adaptation of \Cref{thm:optimal-covs-are-higher} -- which apply here because Assumption~\ref{ass:M-gradients} holds -- can be used to argue that the optimal $\Sigma_w\opt$ and $\Sigma_v\opt$ in the inner maximization problem in \eqref{eq:upper-bound-infty-reformulation} satisfy $\Sigma_w\opt \succeq \hat\Sigma_w$ and $\Sigma_v\opt \succeq \hat\Sigma_v$. Hence, we reformulate $\bar p\opt$ as 
 \begin{equation}
     \bar p\opt = \min\limits_{U \in \mathcal U_\infty} \max\limits_{\Sigma_w \in \mathcal{M}^+_{\Sigma_w}, \Sigma_v\in \mathcal{M}^+_{\Sigma_v}} J(U; \Sigma_w, \Sigma_v). 
 \end{equation}

\subsection*{Lower Bound for Dual}
To derive a lower bound on $d\opt$, we restrict nature's choices to the set $\underline{\mathcal{B}}^\infty_{\mathcal{N}}$ of all time-invariant Gaussian distributions in the ambiguity set $\mathcal{B}^\infty$ constrained so their covariances belong to the sets $\mathcal{M}^+_{\Sigma_w}$ and $\mathcal{M}^+_{\Sigma_v}$, respectively; that is, \( \underline{\mathcal{B}}^\infty_{\mathcal{N}} = \{ \PP \in \mathcal{B}^\infty_{\mathcal{N}} : \EE_{\PP}[ z z^\top ] \succeq \EE_{\hat{\PP}}[ z z ^\top] ~\mbox{for all}~  z \in \mathcal{Z} \}.\) Therefore, we propose the following lower bound on $d\opt$:
\begin{equation}
    \label{eq:inf-lower-bound-dual}
    \underline d\opt = \left\{
    \begin{array}{c@{~~}c@{~~}l}
         \sup\limits_{\PP \in \underline{\mathcal{B}}^\infty_{\mathcal{N}}} &\inf\limits_{x,u} & J_\PP(x,u) \\
         &\st & u \in \mathcal U_\eta,~\, x = Hu + Gw.
    \end{array}\right.
\end{equation}
As~\eqref{eq:inf-lower-bound-dual} is obtained by restricting the feasible set of the dual DRLQ problem~\eqref{eq:inf-dual-1}, we have $\underline{d}\opt \leq d\opt$. 

The following proposition, which leverages the stabilizability and detectability Assumption~\ref{ass:stabilizability-detectibility} and classical results in infinite-horizon LQG control,  will allow simplifying the formulation in~\eqref{eq:inf-lower-bound-dual} by only considering stationary, linear policies.
\begin{proposition}
\label{prop:inf-optimal-controllers}
Under Assumptions~\ref{ass:general_assumption_about_M_function},~\ref{ass:stabilizability-detectibility}-\ref{ass:nominal_inf_horizon} and for any $\PP\in \underline{\mathcal{B}}^\infty_{\mathcal{N}}$, the inner minimization problem in~\eqref{eq:inf-lower-bound-dual} is solved by a policy $u=U\eta$ for some $U \in \mathcal U_\infty$ and a state process $x = H u + G w$.
\end{proposition}
\Cref{prop:inf-optimal-controllers} implies that the feasible set of the inner minimization problem in~\eqref{eq:inf-lower-bound-dual} can be restricted to linear stationary policies of the form $u=U\eta$ for some $U\in\mathcal U_\infty$ without optimality loss. Thus, the lower bounding problem~\eqref{eq:upper-bound-infty} is equivalent to the simplified maximin problem 
\begin{equation}
\underline d^\star=
    \max\limits_{\Sigma_w\in \mathcal{M}_{\Sigma_w}^+, \Sigma_v \in \mathcal{M}_{\Sigma_v}^+} \min\limits_{U\in\mathcal U_\infty} ~ J(U;\Sigma_w,\Sigma_v),
    \label{eq:lower-bound-infty-reformulation}
\end{equation}
where the objective function $J(U;\Sigma_w,\Sigma_v)$ is defined as before.

Finally, we prove that $J(U;\Sigma_w, \Sigma_v)$ is a convex-concave saddle function. We first prove that the limit superior in the definition of the average cost $J_\PP(x,u)$ reduces to a normal limit whenever $u=U\eta$ is a stationary policy induced by some $U\in\mathcal U_\infty$, from which the result on $J$ will follow.

\begin{lemma}
    \label{lem:limsup-converges}
    Under Assumptions~\ref{ass:general_assumption_about_M_function},~\ref{ass:stabilizability-detectibility}-\ref{ass:nominal_inf_horizon} and for any $\PP\in \underline{\mathcal{B}}^\infty_{\mathcal{N}}$, if $u=U\eta$ and $x=H u+Gw$ for some $U\in\mathcal U_\infty$, then:
    \begin{equation}
        J_\PP(x,u) =  \lim_{T\to\infty} \frac{1}{T} \sum\limits_{t=0}^{T-1} \EE_{\PP}\left[ x_t^\top Q_0 x_t + u_t^\top R_0 u_t \right].
        \label{eq:J-limit-formula}
    \end{equation}
\end{lemma}
With this, the following result proves that $J(U; \Sigma_w, \Sigma_v)$ is convex-concave.

\begin{proposition}
\label{prop:J-saddle}
   Under Assumptions~\ref{ass:general_assumption_about_M_function},\ref{ass:M-gradients},~\ref{ass:stabilizability-detectibility}-\ref{ass:nominal_inf_horizon} and for any $\PP\in \underline{\mathcal{B}}^\infty_{\mathcal{N}}$, the restriction of $J(U; \Sigma_w, \Sigma_v)$ to $\mathcal U_\infty\times (\mathcal{M}^+_{\Sigma_w}\times\mathcal{M}^+_{\Sigma_v})$ is convex in~$U$ and linear in~$(\Sigma_w, \Sigma_v)$.  
\end{proposition}

The insights of this section culminate in the following main result, which shows that all structural results shown in the finite horizon extend to the infinite horizon formulation.

\begin{theorem}[Nash strategies and strong duality]
    Under Assumptions~\ref{ass:general_assumption_about_M_function},\ref{ass:M-gradients},~\ref{ass:stabilizability-detectibility}-\ref{ass:nominal_inf_horizon}, we have:
    \begin{itemize}
        \item[(i)] 
        The optimal value in problem \eqref{eq:inf-primal-1} 
        equals the optimal values in problems~\eqref{eq:inf-dual-1}, \eqref{eq:upper-bound-infty}, and \eqref{eq:inf-lower-bound-dual}, i.e., $p^\star = d^\star = \bar{p}^\star = \underline{d}^\star$, and all optimal values are attained.
        \item[(ii)] The primal DRLQ problem~\eqref{eq:inf-primal-1} admits an optimal stationary linear control policy, $u\opt=U \eta$ for $U \in U_\infty$.
        \item[(iii)] The dual DRLQ problem~\eqref{eq:inf-dual-1} admits an optimal time-invariant Gaussian distribution, $\PP\opt \in \mathcal{B}^\infty_{\mathcal{N}}$.
        \item[(iv)] The optimal covariance matrices under $\PP\opt$ satisfy $\Sigma_w\opt \succeq \hat{\Sigma}_w$ and $\Sigma_v\opt \succeq \hat{\Sigma}_v$.
    \end{itemize}
    \label{thm:strong-duality-infinite}
\end{theorem}
\begin{proof}{Proof.}
The proof of~\Cref{thm:strong-duality-infinite} relies on the minimax theorem due to \citet{ref:fan1953minimax} (also see Theorem 4.2 in \citet{ref:sion1958minimax}), which states that if $\mathcal{M}$ is a non-empty convex subset of a vector space, $\mathcal N$ is any non-empty compact and convex subset of a  topological vector space, and $f: \mathcal{M} \times \mathcal N \to \R$ is convex in its first argument and concave and upper semicontinuous in the second argument, then $\inf_{\nu \in \mathcal N } \max_{\mu \in \mathcal{M}} f(\mu, \nu) = \max_{\mu \in \mathcal{M}}\inf_{\nu \in \mathcal N} f (\mu, \nu)$.

By the construction of the upper and lower bounding problems~\eqref{eq:upper-bound-infty} and~\eqref{eq:inf-lower-bound-dual}, respectively,
\begin{equation}
    \label{eq:primal-dual-bounds-infinite-horizon}
    \underline d\opt \leq d\opt \leq p^\star \leq \overline p^\star.
\end{equation}
From \Cref{prop:inf-optimal-controllers} and the discussion immediately following it, we also have that~$\underline d\opt$ matches the optimal value of the simplified maximin problem~\eqref{eq:lower-bound-infty-reformulation}. Similarly, from \Cref{prop:inf-optimal-distributions} and the discussion following it, we have that~$\overline p^\star$ matches the optimal value of the simplified minimax problem~\eqref{eq:upper-bound-infty-reformulation}. Note that~\eqref{eq:upper-bound-infty-reformulation} and~\eqref{eq:lower-bound-infty-reformulation} differ only with respect to the order of minimization and maximization. By construction, the feasible set~$\mathcal U_\infty$ from which policies are chosen in~\eqref{eq:upper-bound-infty-reformulation} and~\eqref{eq:lower-bound-infty-reformulation} is a non-empty subset of the infinite-dimensional linear space~$\mathcal T^{m\times p}$ and is convex because the norm $\|\cdot\|_\mathcal T$ for Toeplitz matrices is convex. In addition, 
the feasible set $\mathcal{M}^+_{\Sigma_w}\times \mathcal{M}^+_{\Sigma_v}$ from which nature chooses covariance matrices in ~\eqref{eq:upper-bound-infty-reformulation} and~\eqref{eq:lower-bound-infty-reformulation} is a non-empty convex and compact subset of the finite-dimensional linear space~$\mathbb S^n\times\mathbb S^p$. Finally, \Cref{prop:J-saddle} ensures that  the restriction of $J(U; \Sigma_w, \Sigma_v)$ to $\mathcal U_\infty\times (\mathcal{M}_{\Sigma_w}\times\mathcal{M}_{\Sigma_v})$ is convex in~$U$ and linear (and thus trivially upper semicontinuous) in~$(\Sigma_w, \Sigma_v)$. Therefore, the Fan Minimax Theorem is applicable to problems ~\eqref{eq:upper-bound-infty-reformulation} and~\eqref{eq:lower-bound-infty-reformulation} and this implies that $\underline d\opt=\overline p\opt$, which in turn implies that all three inequalities in~\eqref{eq:primal-dual-bounds-infinite-horizon} collapse to equalities. Each of these equalities implies one of the three assertions in~(i)-(iii). 

Lastly, assertion~(iv) follows immediately from the construction of the sets $\mathcal{M}^+_{\Sigma_w}\times \mathcal{M}^+_{\Sigma_v}$, which implies that $\Sigma_w\opt \succeq \hat{\Sigma}_w$ and $\Sigma_v\opt \succeq \hat{\Sigma}_v$.

\end{proof}

\section{Conclusions, Limitations, and Future Directions}
\label{sec:conclusions_future_directions}
This work formulated a distributionally robust version of the classical LQG problem by replacing the fixed disturbance model with a divergence ball around a nominal distribution. Under \revision{zero-mean} Gaussian nominal noise, \revision{an orthogonality requirement on the second moments of the distributions (equivalent to uncorelatedness under zero means),} and suitable structural requirements on the divergence, we proved that affine output-feedback policies and a Gaussian distribution form a Nash equilibrium in our mini-max game. We also showed that it is optimal for the adversary to set the mean of the distribution to zero\revision{, in which case} the \dm{}'s policies become linear and the adversary optimally “inflates’’ the nominal covariance matrix. The results generalize and rationalize many results from the (robust) LQG literature, provide an intuitive rule of thumb to address distributional misspecification in practice, and enable a very efficient Frank-Wolfe algorithm whose iterations are standard LQG subproblems. All results extend to an infinite-horizon, average-cost setting -- yielding stationary linear policies and a time-invariant Gaussian worst-case model -- and to entropy-regularized optimal transport, Fisher divergence, or an  elliptical nominal distribution with 2-Wasserstein distance.

Future work could be aimed at extending this framework or addressing some of its limitations. For instance, one could leverage our results to compute control policies with performance guarantees when the disturbance noise distribution is known but otherwise \emph{general}. More specifically, consider a distribution $\QQ$ such that $\QQ \in \mathcal{B}$ for some ambiguity set $\mathcal{B}$ that is compatible with our assumptions. Then, the optimal solution in the DRLQ problem $\inf_{u \in \mathcal{U}_y} \sup_{\PP \in \mathcal{B}} \EE_{\PP}[J(u)]$ will be feasible in the (intractable) linear quadratic control problem $\inf_{u \in \mathcal{U}_y} \EE_{\QQ}[J(u)]$, and its optimality gap will be upper bounded by the difference between the optimal value of the DRLQ model and the optimal value in the distributionally robust ``optimistic" problem $\inf_{\PP \in \mathcal{B}} \inf_{u \in \mathcal{U}_y} \EE_{\PP} [J(u)]$. For this procedure to be effective, one must be able to solve the latter problem (which may be non-convex) and also test whether a given distribution $\QQ$ belongs to $\mathcal{B}$, and possibly adjust $\mathcal{B}$ to guarantee this, without significantly expanding the ambiguity set too much. These tasks are likely challenging for general distributions $\QQ$, so future work could be devoted to identifying tractable cases and designing algorithms for membership testing, ambiguity set calibration, and solving the optimistic problem.

One could also consider distributionally robust formulations for other stochastic control problems that extend the LQR or LQG frameworks, such as the linear-exponential-quadratic Gaussian model due to \citet{Whittle_risk_sensitive} or the one recently considered in~\cite{deZegherIancuPlambeck_2019}, or the model with affine dynamics and extended quadratic costs from \citet{Barratt_Boyd_2022}. Alternatively, one could consider formulations that involve state or control constraints and derive policies with provable performance guarantees, or consider control problems while learning the ambiguity set, as in \citet{iancu_trichakis_yoon_2021} and related literature.

\section*{Acknowledgements}
This work was supported as a part of the NCCR Automation, a National Centre of Competence in Research, funded by the Swiss National Science Foundation (grant number 51NF40\_225155).
Bahar Taşkesen gratefully acknowledges financial support from the University of Chicago Booth School of Business.
Dan Iancu gratefully acknowledges partial support from INSEAD during his leave from Stanford University, which contributed to the completion of this research.
\newpage

%
%

\section*{Appendix}
\setcounter{section}{0}
\renewcommand\thesection{\Alph{section}}
\renewcommand{\theequation}{A.\arabic{equation}}
\renewcommand{\thefigure}{A.\arabic{figure}}
\renewcommand{\thetable}{A.\arabic{table}}
\renewcommand{\thealgorithm}{A.\arabic{algorithm}}
\renewcommand\thetheorem{\Alph{theorem}}
\renewcommand\thelemma{\Alph{lemma}}

\maketitle

\section{Results on the LQG Problem with Known Distributions} 
\label{sec: appx: optimal-solution-classic-LQG}
\subsection{Finite Horizon}
The finite horizon classical LQG problem assumes that all exogenous noise terms follow known Gaussian distributions with the covariance matrices for the output noise being positive definite, i.e., $\Sigma_{v_t} \succ 0$ for every $t \in [T-1]$ \citep{Bertsekas_2017}. In this section, we restate these classical results and adapt them to the case where the exogenous noise has non-zero mean. Throughout this section, we let $\mu_{x_0},\mu_{v_t}, \mu_{w_t}$ denote the means and $X_0 = \Sigma_{x_0}$, $V_t = \Sigma_{v_t}$, and $W_t = \Sigma_{w_t}$ to denote the covariance matrices characterizing the exogenous noise.

This problem can be solved efficiently via dynamic programming~\citep{Bertsekas_2017}. The unique optimal control inputs satisfy $u\opt_t=K_t \hat x_t$ for every $t\in[T-1]$, where~$K_t\in\R^{n\times n}$ is the optimal feedback gain matrix, and~$\hat x_t=\mathbb E_{\PP}[x_t|y_0,\ldots,y_t]$ is the minimum mean-squared-error estimator of~$x_t$ given the observation history up to time~$t$. Thanks to the celebrated separation principle, $K_t$ can be computed by pretending that the system is deterministic and allows for perfect state observations, and~$\hat x_t$ can be computed while ignoring the control problem.

To compute~$K_t$, one first solves the deterministic LQR problem corresponding to the LQG problem. Its value function~$x_t^\top P_tx_t$ at time~$t$ is quadratic in~$x_t$, and~$P_t$ obeys the backward recursion
\begin{subequations}
\label{eq:LQR-solution}
\begin{equation}
    P_t = A_t^\top P_{t+1} A_t + Q_t - A_t^\top P_{t+1} B_t(R_t + B_t^\top P_{t+1} B_t)^{-1} B_t^\top P_{t+1}A_t \quad \forall t\in[T-1]
\label{eq:control-gain-Pt}
\end{equation}
initialized by~$P_T=Q_T$. The optimal feedback gain matrix~$K_t$ can then be computed from~$P_{t+1}$ as
\begin{equation}
    K_t = -(R_t + B_t^\top P_{t+1}B_t)^{-1} B_t^\top P_{t+1} A_t\quad \forall t\in[T-1].
    \label{eq:feedback-gain-Kt}
\end{equation}
\end{subequations}
Importantly, note that $K_t$ only depends on the system matrices $\{A_\tau,B_\tau\}_{\tau \geq t}$ and on the cost matrices $\{R_\tau,Q_\tau\}_{\tau \geq t}$, but does \emph{not depend} on the distribution of the exogenous noise terms.

Because~$x_t$ and~$(y_0,\ldots,y_t)$ follow a multivariate Gaussian distribution, the minimum mean-squared-error estimator~$\hat x_t$ can be calculated directly using the formula for the mean of a conditional Gaussian distribution. Alternatively, one can use the Kalman filter to compute~$\hat x_t$ recursively, which is more insightful and more efficient. The Kalman filter also recursively computes the covariance matrix~$\Sigma_t$ of~$x_t$ conditional on~$y_0,\ldots,y_t$ and the covariance matrix~$\Sigma_{t+1 | t}$ of~$x_{t+1}$ conditional on~$y_0,\ldots,y_{t}$ evaluated under~$\PP$.
Specifically, these covariance matrices obey the forward recursion
\begin{equation}
    \left.\begin{aligned}
    &\Sigma_t = \Sigma_{t | t-1} - \Sigma_{t | t-1} C_t^\top (C_t \Sigma_{t | t-1} C_t^\top + V_t)^{-1} C_t \Sigma_{t | t-1} \\ & \Sigma_{t+1 | t} = A_{t} \Sigma_t A_{t}^\top + W_t 
    \end{aligned} \right\} ~\forall t\in[T-1]
\label{eq:kalman-cov-updates}
\end{equation}
initialized by~$\Sigma_{0 | -1} = X_0$. Using~$\Sigma_{t | t-1}$, we then define the Kalman filter gain as
\begin{equation*}
    L_t = \Sigma_{t} C_{t}^\top  V_{t}^{-1} \quad \forall t \in [T-1]
\end{equation*}
which allows us to compute the minimum mean-squared-error estimator via the forward recursion
\begin{equation}
    \label{eq:mmse}
    \hat x_{t+1} = A_t \hat x_t + B_t K_t \hat{x}_t + \hat \mu_{w_t}+ L_{t+1} \left(y_{t+1} - C_{t+1}(A_t \hat x_t + B_t K_t \hat{x}_t + \hat \mu_{w_t})  - \mu_{v_{t+1}}\right)  \quad \forall t\in[T-1]
\end{equation}
initialized by~$\hat x_0 = L_0 y_0- \hat \mu_{v_0}$.

Note that~\eqref{eq:feedback-gain-Kt} and~\eqref{eq:mmse} readily show that the optimal control policy $u_t\opt$ depends \emph{affinely} on the outputs $y_0, y_1,\dots, y_t$ and provide the explicit recursive procedure needed to compute all the relevant coefficients.

Moreover, one can verify from these expressions that the optimal value of the LQG problem is
\begin{equation}
    \sum\limits_{t=0}^{T-1} \Tr((Q_t-P_t) \Sigma_t) + \sum\limits_{t=1}^T \Tr( P_t (A_{t-1} \Sigma_{t-1} A_{t -1}^\top + W_{t-1})) + \Tr(P_0 X_0) .
    \label{eq:lqg-cost}
\end{equation}

%
%
\subsection{Definitions of Stacked System Matrices for Finite Horizon}
\label{sec: appx: stacked-matrices}
The stacked system matrices appearing in problem~\eqref{eq:DRLQG} are defined as follows. First, the stacked state and input cost matrices $Q\in \mathbb S^{n(T+1)}$ and $R\in \mathbb S^{mT}$ are set to
\[
    \quad Q = \begin{bmatrix}
    Q_0  \\
    &Q_1   \\
    & & \ddots \\
    & & &Q_{T}
    \end{bmatrix}
    \quad \text{and} \quad R = \begin{bmatrix}
    R_0  \\
    &R_1   \\
    & & \ddots \\
    & & &R_{T-1}
    \end{bmatrix},
\]
respectively. Similarly, the stacked matrices appearing in the linear dynamics and the measurement equations $C\in \R^{pT\times n(T+1)}$, $G\in\R^{n (T+1) \times n(T+1)}$ and $H\in \R^{n(T+1) \times m T}$ are defined as
\begin{equation*}
\begin{aligned}
    C = \begin{bmatrix}
    C_0 &0 \\
    &C_1  & 0\\
    & & \ddots &\ddots\\
    & & &C_{T-1} &0
    \end{bmatrix} ,\quad G =
    \begin{bmatrix}
    A_0^0 \\
    A_0^1 &A^1_1 \\
    \vdots & &\ddots \\
    A_0^T &A_1^T &\dots &A^T_T
    \end{bmatrix}
    \end{aligned}
\end{equation*}
and
\[
    H = 
    \begin{bmatrix}
    0 \\
    A^1_1 B_0 & 0 \\
    A^2_1 B_0 &A^2_2 B_1 &0 \\
    \vdots & & &\ddots \\
    \vdots & & & &0\\
    A_1^T B_0 &A_2^T B_1 &\dots &\dots &A^T_T B_{T-1}
    \end{bmatrix},
\] 
respectively, where $A^t_s = \prod_{k = s}^{t-1} A_{k}$ for every $s < t$ and $A^t_s = I$ for $s= t$.

%
%
\subsection{Results on Purified Output Feedback Policies}
\label{subsec:purified_outputs_classical_LQG}
Using the stacked system matrices, we can now express the purified observation process~$\eta$ as a linear function of the exogenous uncertainties~$w$ and~$v$ that is {\em not} impacted by $u$. The following result summarizes this -- see also \cite{Ben-Tal_Boyd_Nemirovski_2005, ref:skaf2010design}.
\begin{lemma}
    We have $\eta = Dw + v$, where $D = CG$.
    \label{lemma:eta-rep-w-v}
\end{lemma}
\begin{proof}{Proof.} 
The purified observation process is defined as~$\eta = y -\hat y$. Recall now that the observations of the original system satisfy $y=Cx+v$. Similarly, one readily verifies that the observations of the fictitious noise-free system satisfy $\hat y= C \hat x$. Thus, we have $\eta= C(x-\hat x)+v$. Next, recall that the state of the original system satisfies $x=Hu+Gw$, and note that the state of the fictitious noise-free system satisfies $\hat x=Hu$. Combining all of these linear equations finally shows that $u$ cancels out and that $\eta = CGw+v = D w + v$. 
\end{proof}

It is well known that every causal control policy that is linear in the original observations~$y$ can be reformulated as a causal policy that is linear in the purified observations~$\eta$ and vice versa \citep{Ben-Tal_Boyd_Nemirovski_2005, ref:skaf2010design}. Perhaps surprisingly, however, the one-to-one transformation between the respective coefficients of~$y$ and~$\eta$ is {\em not} linear. To keep this paper self-contained, we review these insights in the next lemma. 
\begin{lemma}
If $u=U\eta+q$ for some~$U \in \mathcal U$ and $q \in \R^{pT}$, then $u=U'y+q'$ for $U'= (I + UCH)^{-1} U $ and $q'=(I + UCH)^{-1} q$. Conversely, if $u=U'y+q'$ for some~$U' \in \mathcal U$ and $q' \in \R^{pT}$, then $u=U\eta+q$ for $U=(I - U'CH)^{-1}{ U'}$ and $q=(I - U'CH)^{-1} q'$.
\label{lemma:linear-rel-u-eta}
\end{lemma}
\begin{proof}{Proof.} 
    If $u = U \eta +q$ for some~$U \in \mathcal U$ and $q \in \R^{pT}$, then we have
    \[
        u=U\eta+q= U(y-\hat y)+q=Uy-UC\hat x+q=Uy-UCHu+q,
    \]
    where the second equality follows from the definition of~$\eta$, the third equality holds because $y = Cx+v$, and the last equality exploits our earlier insight that $\hat y = C\hat x$. The last expression depends only on~$y$ and~$u$. Solving for $u$ yields $u= U'y+q'$, where $U'= (I+ UCH)^{-1} U$ and $q'=(I + UCH)^{-1}q$. Note that $(I+ UCH)$ is indeed invertible because $ I+ UCH$ is a lower triangular matrix with all diagonal entries equal to one, ensuring a determinant of one. 
    
    Similarly, if $u = U'y + q'$ for some~$U' \in \mathcal U$ and $q' \in \R^{pT}$, then we have
    \[
        u = U'y + q' = U'(\eta + \hat y)+q' = U'\eta + U'C\hat x +q' = U'\eta + U'CHu +q'.
    \]
    Solving for~$u$ yields $u=U\eta+q$, where $U = (I - U' CH)^{-1} U' $ and $ q = (I - U' CH)^{-1} q'$. Note again that $(I- U'CH)$ is indeed invertible because $(I - U' CH)$ is a lower triangular matrix with all diagonal entries equal to one. 
 \end{proof}

%
%
\subsection{Infinite-Horizon Results}
\label{sec:classical-lqg-inf-horizon}
The infinite-horizon classical LQG problem with average cost criterion assumes that all exogenous noise terms follow known time-invariant Gaussian distributions with the covariance matrices for the  noise being positive definite, \textit{i.e.}, $\Sigma_{v_t} = \Sigma_{v} \succ 0$ and $\Sigma_{w_t} = \Sigma_{w} \succ 0$ for all $t \in \mathbb N$.
For simplicity, we consider the case where the exogenous noise has zero mean. Furthermore, consistent with \Cref{sec:infinite-horizon}, we assume that the standard assumptions outlined in~Assumption~\ref{ass:stabilizability-detectibility} hold.

The infinite-horizon classical LQG problem is solved by optimal control inputs that satisfy $u_t\opt = K \hat x_t$ for every $t \in \mathbb N$, where $K \in \R^{n \times n}$ is the optimal steady-state feedback gain matrix, and $\hat x_t$ is again the minimum mean-squared-error estimator of~$x_t$ given the observation history up to time~$t$.

To compute $K$, one first finds the $P \in \mathbb S_+^{n}$ that solves the discrete-time algebraic Riccati equation~(DARE)
\begin{equation}P = A_0^\top P A_0 + Q_0 - A_0^\top P B_0 (R_0 + B_0^\top P B_0)^{-1} B_0^\top P A_0,
\label{eq:dare}
\end{equation}
which is guaranteed to exist and to be unique under Assumption~\ref{ass:stabilizability-detectibility} by \cite[Lemma 16.6.1]{ref:lancaster1995algebraic}.
The matrix~$K$ can then be computed from~$P$ as $$K = - (R_0 + B_0^\top P B_0)^{-1} B_0^\top P A_0.$$

As in the finite-horizon case, the Kalman filter can be used to compute $\hat x_t$. The steady-state covariance matrix $\tilde \Sigma \in \mathbb S_+^{n}$ solves
\begin{equation}\tilde \Sigma = A_0 \tilde \Sigma A_0^\top + \Sigma_w - A_0\tilde \Sigma C_0^\top(C_0 \tilde \Sigma C_0^\top + \Sigma_v)^{-1} C_0 \tilde \Sigma A_0^\top,
\label{eq:kalman-dare}
\end{equation}
and is guaranteed to exist and to be unique under Assumption~\ref{ass:stabilizability-detectibility} by \cite[Theorem 17.5.3]{ref:lancaster1995algebraic}. We define the steady-state Kalman filter gain as
$$L = \tilde \Sigma C_0^\top (\Sigma_v + C_0 \tilde \Sigma C_0 )^{-1}.$$
This allows state estimation via the recursion
\begin{equation}\hat x_{t+1} = A_0 \hat x_t + B_0 u_t\opt + L (y_{t+1} - C_0(A_0 \hat x_t + B_0 u^\star_t)),
\label{eq:state-estimate-dyn}
\end{equation}
initialized by $\hat x_0 = L y_0$.
If $\hat\Sigma_w \succ 0$, then $A(I - LC_0)$ is Schur stable by \citep[Theorem 17.5.3]{ref:lancaster1995algebraic}. 
Additionally, if $Q_0 \succ 0$, then $(A_0 +B_0 K)$ is Schur stable by \citep[Theorem 16.6.4]{ref:lancaster1995algebraic}. These observations will be useful for establishing that both the state $x_t$ and its estimate $\hat{x}_t$ admit stationary covariance matrices.

\begin{lemma}
Suppose that Assumption~\ref{ass:stabilizability-detectibility} holds, $\hat\Sigma_w \succ 0$ and $\hat\Sigma_v\succ 0$.
   If $u_t\opt = K \hat x_t$, and $\hat x_t$ obeys the forward recursion 
   \[\hat x_{t+1} = A_0 \hat x_t + B_0 u_t\opt + L(y_{t+1} - C_0 (A_0 \hat x_t + B_0 u_t\opt)),\]
   initialized by $\hat x_0 = L y_0$, then $x_t$ and $\hat x_t$ admit stationary covariance matrices. 
    \label{lem:stationary-covs}
\end{lemma}
\begin{proof}{Proof.}
Denote by $e_t = x_t - \hat x_t$ the estimation error and by $z_{t} = [x_t^\top, e_t^\top]^\top$ the joint vector of state and estimation error under the optimal control inputs $u^\star_t = K \hat{x}_t$ for every $t \in \mathbb N$.
The state dynamics under the optimal control inputs are then given by
\begin{equation}
    \begin{aligned}
        x_{t+1} &= A_0 x_t + B_0 K \hat x_t + w_t = A_0(e_t + \hat x_t) + B_0 K \hat x_t + w_t\\
        &= A_0 e_t + (A_0+B_0 K)\hat x_t + w_t = (A_0+B_0 K) x_t - B_0 K e_t + w_t.
    \end{aligned}
    \label{eq:x-t-dyn}
\end{equation}
    For the error dynamics, we have
    \begin{equation}
    \begin{aligned}
        e_{t+1}
        &= x_{t+1} - \hat x_{t+1} \\
        &= A_0 x_t + B_0 K \hat{x}_t + w_t -  A_0 \hat x_t - B_0 K \hat{x}_t - L (C_0 x_{t+1} + v_{t+1}) + L C_0(A_0 \hat x_t + B_0 K \hat x_t)\\
        &= (A_0 - L C_0 A_0) (x_t - \hat x_t) - (I - L C_0) w_t - L v_{t+1}\\
        &= (A_0 - L C_0 A_0) e_t + (I - L C_0) w_t - L v_{t+1},
    \end{aligned}
    \label{eq:e-t-dynamics}
    \end{equation}
    where the second equality follows from \eqref{eq:x-t-dyn} and $y_{t+1} = C_0 x_{t+1} + v_{t+1}$, the third equality follows from $x_{t+1} = A_{0} x_t + B_0 K \hat{x}_t + w_t$ and rearranging terms, and the last equality follows from the definition of $e_t$. 
    Combining the dynamics in \eqref{eq:x-t-dyn} and \eqref{eq:e-t-dynamics}, we have
    \begin{equation*}
        z_{t+1} =Fz_t + \Xi \, \xi_t,~\text{where}~ F = \begin{bmatrix}
            A_0 + B_0 K  &  -B_0 K\\
            0            &  A_0 - L C_0 A_0
        \end{bmatrix},~\xi_t = \begin{bmatrix}
        w_{t}\\
        v_{t+1}
    \end{bmatrix}~ \text{and}~\Xi = \begin{bmatrix}
            I   & 0\\
            I - LC_0    & -L
        \end{bmatrix}.
    \end{equation*}
   Note that $(\xi_t)_{t\in\mathbb N}$ is an i.i.d. sequence with zero mean and finite covariance in the form of 
    \begin{equation}
        \Sigma_{\xi_t} = \EE[\xi_t \xi_t^\top] = \begin{bmatrix}
            \Sigma_w \quad &     0\\
            0        & \Sigma_v
        \end{bmatrix}:= \Sigma_\xi.
    \end{equation}
    As $\xi_t $ is independent of $z_t$, the linear recursion of $z_{t+1}$ implies that $\Sigma_{z+1} = \EE[z_{t+1} z_{t+1}^\top]$ follows the discrete-time Lyapunov recursion 
    \begin{equation}\Sigma_{z_{t+1}} = F \Sigma_{z_t} F^\top +  \Xi \Sigma_{\xi}\Xi^\top. 
    \label{eq:lyapunov-z-covs}
    \end{equation}
     Note that $F$ is an upper block triangular matrix, and thus its spectrum is given by $\sigma(F) = \sigma(A_0 + B_0 K) \cup \sigma(A_0 - L C_0 A_0)$. As both $A_0 + B_0 K$ and $(I -L C_0 ) A_0$ are Schur stable, $\sigma(F) < 1$. Therefore, $F$ is a Schur stable matrix. By \cite[Theorem 3.4]{ref:kumar2015stochastic}, as $F$ is Schur stable and $G \Sigma_{\xi} G^\top \succeq 0$, the matrix $\Sigma_{z_t}$ converges to the unique positive semidefinite solution of 
     \begin{equation}\Sigma_{z} = F \Sigma_z F^\top + G \Sigma_{\xi} G^\top.
     \label{eq:sigma-z-lyapunov-eq}
     \end{equation}
        By construction of $z_t$, $\Sigma_{z_t}$ admits the following block from 
     \begin{equation}
         \Sigma_{z_t} =\begin{bmatrix}
             \EE[x_t x_t^\top] \quad \,  & \EE[x_t e_t^\top]\\
            \EE[e_t x_t^\top]   \quad\, &  \EE[e_t e_t^\top]
         \end{bmatrix}
     \end{equation}
     As $\Sigma_{z_t}$ converges to a stationary matrix $\Sigma_z$ solving \eqref{eq:sigma-z-lyapunov-eq}, each of its blocks must converge as well, and thus this observation completes the first assertion of the statement that $x_t$ admits a stationary covariance matrix. 
     
          Next, recall that the estimator satisfies $\hat x_t = x_t - e_t$, and thus we have 
     \begin{equation*}\EE[\hat x_t \hat x_t^\top] =\EE[(x_t - e_t)(x_t - e_t)^\top] = \EE[x_t x_t^\top] + \EE[e_t e_t^\top] - \EE[x_t e_t^\top] - \EE[e_t x_t^\top]. 
     \label{eq:estimator-covs}
     \end{equation*}
 Because each of the components of $\Sigma_{z_t}$ converges, the matrices on the right-hand-side of the expression above also converge. Hence, $\EE[\hat x_t \hat x_t^\top]$ converges to a stationary matrix.  
\end{proof}

%
%
\section{Proofs for \Cref{sec:problem-definition}}
\label{app:proofs_for_section2}

\proof[Proof of~Theorem~\ref{theorem: complexity result}.]
Under the assumptions of the theorem we have $x_{1} = w_0$ and $y_0 = v_0 $. Hence, $\EE_\PP[w_0 | v_0] $ is equivalent to $\EE_{\PP_{w_0}}[{w_0}]$ because $v_0$ and $w_0$ are independent. As $\PP_{w_0}$ is the uniform distribution over polytope~$\mathcal H$, computing $\EE_{\PP_{w_0}}[w_0]$ is equivalent to computing the centroid of~$\mathcal H$, which is known to be $\#$P-hard~\cite[Theorem~1]{ref:rademacher2007approximating}.  \endproof


%
%
\medskip
\subsection{Verifying Assumption~\ref{ass:general_assumption_about_M_function} for Examples in \S\ref{sec:examples}}

\smallskip
\subsubsection{Assumption~\ref{ass:general_assumption_about_M_function} for Wasserstein Ambiguity Sets.}
\label{appendix:assumption2_wasserstein}
Our construction and proof rely on the Gelbrich distance, which we formally define next.
\begin{definition}[Gelbrich distance]
\label{def:gelbrich-distance}
For any $d \in \mathbb N$, the Gelbrich distance between two pairs of mean vectors and covariance matrices~$(\mu_z, \Sigma_z), (\hat\mu_z, \hat \Sigma_z) \in \setmoments{d_z}$ is given by
\begin{equation*}
    \mathds G \bigl((\mu_z, \Sigma_z), (\hat \mu_z, \hat\Sigma_z) \bigr) = \sqrt{\|\mu_z - \hat \mu_z\|^2+ \operatorname{Tr}\left(\Sigma_z + \hat\Sigma_z - 2\left( \hat \Sigma_z^{1/2} \Sigma_z \hat \Sigma_z^{1/2}\right)^{1/2}\right)}.
\end{equation*}
\end{definition}
The Gelbrich distance is closely related to the 2-Wasserstein distance. Indeed, it is known that the 2-Wasserstein distance between two distributions is bounded below by the Gelbrich distance between the mean-covariance pairs of the two distributions, and when the two distributions are Gaussian, this bound tight. These results are summarized in the next proposition.
\begin{proposition}[Gelbrich bound {\cite[Theorem 2.1]{gelbrich1990formula}}]\label{prop:WassersteinSubsetGelbrich}
For any two distributions $\PP_z, \hat \PP_z \in \setalldist{d_z}$ with mean-covariance pairs $(\mu_z, \Sigma_z), ( \hat \mu_z, \hat \Sigma_z) \in \setmoments{d_z}$, respectively, we have: 
\begin{enumerate}
    \item[i)]$\mathds W (\PP_z, \hat \PP_z) \geq  \mathds G ((\mu_z, \Sigma_z),(\hat \mu_z, \hat \Sigma_z))$, 
    \item[ii)] $\mathds W (\PP_z, \hat \PP_z) =  \mathds G ((\mu_z, \Sigma_z),(\hat \mu_z, \hat \Sigma_z))$ if $\PP_z$ and $\hat \PP_z$ are Gaussian.\footnote{\citet{gelbrich1990formula} proves this equality more generally, for any two elliptical distributions with the same generator (we discuss this in Appendix~\S\ref{sec: elliptical} -- see \Cref{prop: elliptical variant of Assumption 2 holds}). The equality for the Gaussian case is a special instance of that result.}
\end{enumerate}
\end{proposition}
\Cref{prop:WassersteinSubsetGelbrich} implies that Assumption~\ref{ass:general_assumption_about_M_function} (i) is satisfied. To see that Assumption~\ref{ass:general_assumption_about_M_function} (ii) is also satisfied,  
consider the set $\mathcal{M}_{(\mu_z, M_z)}$ for $\mathds{D}=\mathds{W}$, which we can rewrite as:
\begin{equation*}
    \begin{aligned}
        \mathcal{M}_{(\mu_1, M_z)}^{\mathds W} 
        &= \left\{ (\mu_z, M_z) \in \setmoments{d_z} : \mathds W\left(\mathcal N(\mu_z,M_z), \hat{\PP}_z\right) \leq \rho_z \right\}\\
        &= \left\{ (\mu_z, M_z) \in \setmoments{d_z}, \, \left( \mathds{G}((\mu_z,M_z - \mu_z\mu_z^\top),(\hat \mu_z, \hat \Sigma_z)) \right)^2 \leq \rho_z^2 \right\},
    \end{aligned}
\end{equation*}
where the second equality follows from~\Cref{prop:WassersteinSubsetGelbrich}-(ii). The latter set is known to be convex and compact~\cite[Proposition 3.17]{ref:nguyen2019adversarial}, so we conclude that the 2-Wasserstein distance satisfies also~Assumption~\ref{ass:general_assumption_about_M_function}-(ii).

%
%
\smallskip
\subsubsection{Assumption~\ref{ass:general_assumption_about_M_function} for Kullback-Leibler Ambiguity Sets.}
\label{appendix:assumption2_KL}
Mirroring the Wasserstein case, we first formalize a divergence between mean-covariance pairs. 
\begin{definition}[KL Divergence Between Moments]
 \label{def:kl-covs}
 The KL-type divergence from $(\mu_z, \Sigma_z )\in \R^{d_z} \times \mathbb S^{d_z}_{+}$ to $(\hat\mu_z, \hat\Sigma_z) \in\R^{d_z} \times \mathbb S^{d_z}_{++}$ is given by
 \[\mathds T((\mu_z,\Sigma_z), (\hat\mu_z, \hat \Sigma_z)) =  \frac{1}{2}\left( (\mu_z - \hat \mu_z)^\top \hat \Sigma_z^{-1} (\mu_z -\hat \mu_z) + \operatorname{Tr}\left(\Sigma_z \hat \Sigma_z^{-1} \right) - \log\det\left(\Sigma_z \hat \Sigma_z^{-1} \right) - d_z \right).\]
\end{definition}
In our case, the mean-covariance pair $(\hat \mu_z, \hat \Sigma_z)$ corresponds to the nominal distribution $\hat{\PP}$. To ensure that the inverse of $\hat \Sigma_z$ is well-defined, we must therefore require that the nominal distribution $\hat{\PP}_z$ of every noise term $z \in \mathcal{Z}$ be non-degenerate, that is, the nominal covariance matrix is positive definite, $\hat{\Sigma}_z \succ 0$. This slightly strengthens Assumption~\ref{ass:Gaussian}, but is without substantial practical loss.

The following proposition shows that the KL divergence from any distribution to a non-degenerate, Gaussian distribution is bounded below by the KL-type divergence between their respective mean-covariance pairs, and the bound is tight if the former distribution is Gaussian.
\begin{proposition}[KL bound] 
\label{prop: KL bound}
For any distribution $\PP_z$ on $\setalldist{d_z}$ with mean-covariance pair $(\mu_z, 
\Sigma_z) \in \setmoments{d_z}$ and Gaussian distribution $\hat \PP_z$ with mean-covariance pair $(\hat \mu_z, \hat \Sigma_z) \in \R^{d_z} \times \mathbb S_{++}^{d_z}$, we have: 
\begin{enumerate}
    \item[i)] $\mathds{K} (\PP_z, \hat \PP_z) \geq  \mathds T((\mu_z,\Sigma_z),(\hat \mu_z, \hat \Sigma_z))$
    \item[ii)] $\mathds{K} (\PP_z, \hat \PP_z) =  \mathds T((\mu_z,\Sigma_z),(\hat \mu_z, \hat \Sigma_z))$ if $\PPz{1}$ is Gaussian.
\end{enumerate}
\end{proposition}
\begin{proof}{Proof.}
Let $\hat f_z$ denote the density of the Gaussian distribution $\hat \PP_z$. Because the KL divergence from $\PP_z$ to $\hat \PP_z$ is finite only if $\PP_z$ is absolutely continuous with respect to $\hat \PP_z$, $\PP_z$ must admit a density on $\R^{d_z}$, which we denote by $f_z$. The KL divergence can then be written as:
\begin{align*}
\mathds{K}(\PP_z, \hat \PP_z) 
 &= \int_{\R^{d_z}} f_z(z)\,\ln \left(\frac{f_z(z)}{\hat f_z(z)}\right)\,\diff z \\
 & = h(\PP_z) - \int_{\R^{d_z}} f_z(z)\,\ln \hat f_z(z)\diff z,
\end{align*}
where $h(\PP_z) = - \int_{\R^{d_z}} f_z(z)\,\ln f_z(z)\diff z$ denotes the \emph{differential entropy} of $\PP_z$. Because $\hat \PP_z$ is Gaussian with mean $\hat \mu_z$ and covariance $\hat \Sigma_z$, the second term above can be written as:
\begin{align*}
\int_{\R^{d_z}} f_z(z)\,\ln \hat f_z(z)\,\diff z
&= -\tfrac12\, \mathbb{E}_{\PP_z}\Bigl[(z-\hat \mu_z)^\top (\hat \Sigma_z)^{-1} (z-\hat \mu_z)\Bigr]
  -\tfrac12\,\ln\bigl((2\pi)^d\,\det(\hat \Sigma_z)\bigr) \\
&= -\tfrac12\, \mathbb{E}_{\PP_z}\Bigl[\Tr \Bigl( (\hat \Sigma_z)^{-1} (z-\hat \mu_z) (z-\hat \mu_z)^\top \Bigr)\Bigr]
  -\tfrac12\,\ln\bigl((2\pi)^d\,\det(\hat \Sigma_z)\bigr) \\
&= -\tfrac12\, \Tr \Bigl((\hat \Sigma_z)^{-1}\,\bigl(\Sigma_z 
  + (\mu_z-\hat \mu_z)(\mu_z-\hat \mu_z)^\top\bigr)\Bigr) -\tfrac12\,\ln\bigl((2\pi)^d\,\det(\hat \Sigma_z)\bigr),
\end{align*}
where in the last step we used that fact that the mean is $\mu_z$ and the covariance is $\Sigma_z$ under $\PP_z$. Because $\mu_z$, $\Sigma_z$, $\hat \mu_z$, and $\hat \Sigma_z$ are fixed, the expression above is fixed (for any distribution $\PP_z$ with mean $\mu_z$ and covariance $\Sigma_z$), and the problem of minimizing the KL divergence reduces to the problem of \emph{maximizing} the differential entropy $h(\PP_z)$. It is a standard result that among all distributions on $\mathbb{R}^{d_z}$ with given mean and covariance matrix, the Gaussian distribution maximizes the differential entropy \citep[][Theorem~9.6.5]{CoverThomas2006}. This completes the proof. 
\end{proof}

\Cref{prop: KL bound} implies that~Assumption~\ref{ass:general_assumption_about_M_function}-(i) is satisfied. To see that~Assumption~\ref{ass:general_assumption_about_M_function}-(ii) is also satisfied, consider the set $\mathcal{M}_{(\mu_z, M_z)}$ for $\mathds{D}=\mathds{K}$, which can be rewritten as:
\begin{equation*}
    \begin{aligned}
        \mathcal{M}_{(\mu_z, M_z)}^{\mathds K} 
        &= \bigl\{(\mu_z, M_z) \in \setmoments{d_z} : \mathds K\bigl(\mathcal N(\mu_z, M_z), \hat{\PP}_z\bigr) \leq \rho_z \bigr\} \\
        &= \left\{(\mu_z, M_z) \in \setmoments{d_z} : M_z - \mu_z \mu_z^\top \in \mathbb S_{++}^{d_z},~\mathds T((\mu_z,\Sigma_z),(\hat \mu_z,\hat  \Sigma_z)) \leq \rho_z \right\}.
    \end{aligned}
\end{equation*}
The second equality follows from \Cref{prop: KL bound}-(ii) and because any distribution $\PP_z$ in the set $\mathcal{M}_{(\mu_z, M_z)}^{\mathds{K}}$ must yield a positive definite covariance matrix for $z$ because we consider 
$\rho_z$ finite.\footnote{This follows because $\mathds{T}\bigl((\mu_z,M_z - \mu_z \mu_z^\top),(\hat \mu_z,\hat  \Sigma_z) \bigr)$ is finite only if $\Sigma_z = M_z - \mu_z \mu_z^\top \succ 0$, due to the $- \log \det \Sigma_z$ term.} The set $\mathcal{M}_{(\mu_z, M_z)}^{\mathds K}$ is known to be convex and compact ~\cite[Lemma~A.3]{taskesen2021sequential}. 
We thus conclude that the KL divergence satisfies all premises of Assumption~\ref{ass:general_assumption_about_M_function}.

%
%
\smallskip
\subsubsection{Assumption~\ref{ass:general_assumption_about_M_function} for Moment Ambiguity Sets.}
\label{sec:assumption 2 moment ambiguity sets}
We finally consider moment ambiguity sets in which the divergence $\mathds D$ between two probability distributions relies only on the first two moments of the respective distributions. Specifically, for  distributions $\PP_z, \hat \PP_z \in \setalldist{d_z}$ with mean-second moment matrix pairs $(\mu_z, M_z)$ and $(\hat \mu_z, \hat M_z)$, respectively, we
take 
\[\mathds D(\PP_z, \hat \PP_z) = \mathds{M}\bigl((\mu_z, M_z),(\hat \mu_z, \hat M_z)\bigr),\] 
where $\mathds{M}:  \setmoments{d_z} \times \setmoments{d_z} \to [0, +\infty]$ is any divergence between mean-second moment matrix pairs satisfying $\mathds{M}(m_z,m_z)=0$ for all $m_z = (\mu_z,M_z) \in \setmoments{d_z}$. The ambiguity set $\mathcal{B}_z$ for the random variable $z$ with nominal distribution $\hat \PP_z$ can therefore be expressed~as:
\begin{equation*}
\begin{aligned}
    \mathcal{B}_z = \{\PP_z \in \setalldist{d}: \EE_{\PP_z}[z] = \mu_z,\,\EE_{\PP_z}[z z^\top ] = M_z,\, \mathds{M}\left((\mu_z, M_z), (\hat{\mu}_z, \hat{M}_{z})\right) \leq \rho_{z}\}.
\end{aligned}
\end{equation*}
In this case, Assumption~\ref{ass:general_assumption_about_M_function}-(i) is readily satisfied for any $\mathds{M}$ because every distribution -- including a Gaussian distribution -- with a given mean and second moment matrix would yield the same divergence from $\hat{\PP}_z$ and would therefore minimize the divergence from the nominal $\hat{\PP}_z$. Moreover,  Assumption~\ref{ass:general_assumption_about_M_function}-(ii) is satisfied if the set 
\[ \mathcal{M}_{(\mu_z, M_z)} = \{(\mu_z, M_z) \in \setmoments{d_z}:  \mathds{M}\left((\mu_z, M_z), (\hat{\mu}_z, \hat{M}_{z})\right) \leq \rho_{z}\} \]
is convex and compact for the given $(\hat{\mu}_z, \hat{M}_{z})$. This is readily satisfied if the restriction of $\mathds{M}$ to its first argument $(\mu_z, \Sigma_z)$ is a quasiconvex and coercive function.

%
%
\section{Proofs for \Cref{sec:Nash}}
\label{app:proofs_for_section3}

\subsection{Proofs for \S\ref{sec:sec3_upper_bound_p} (Upper Bound for Primal)}
%
%
%
\proof[Proof of \Cref{prop: min max SDP upper bound problem}.]
In problem~\eqref{upper bound problem 1}, $u$ and $x$ are given by $u = q + UDw + Uv$ and ${x} = H u + G w = Hq + (G + HUD) w + HUv$, respectively. By substituting these expressions into the objective function of problem \eqref{upper bound problem 1}, we obtain the following equivalent reformulation:
\begin{equation*}
\begin{aligned}
\min\limits_{\substack{q \in \R^{pT}\\ U \in \mathcal U}} \max_{\mathbb P \in \overline{\mathcal{B}}} &\mathbb E_{\mathbb P} \Big[ (U (Dw + v) + q)^\top \bar R (U (Dw + v) + q) + w^\top (G^\top QHUD + G^\top Q G) w \Big]\\[-2ex]
&\hspace{1cm}+ \EE_{\PP} [2 w^\top G^\top QHq].
\end{aligned}
\end{equation*}
Because the objective function in the problem above is quadratic in $w$ and $v$, we can express the expectation with respect to any $\mathbb{P} \in \overline{\mathcal{B}}$ in terms of the first two moments of $w$ and $v$, $(\mu_w, M_w)$ and $(\mu_v, M_v)$ respectively, where: 
\begin{equation*}
    \begin{aligned}
    &\mu_w = \EE_{\PP}[w] = (\mu_{x_0}, \mu_{w_0}, \ldots, \mu_{w_{T-1}}),~&&M_w = \mathbb{E}_{\mathbb P}[w w^\top] = \operatorname{diag}(M_{x_0}, M_{w_0}, \dots, M_{w_{T-1}})\\
    &\mu_v = \EE_{\PP}[v] = (\mu_{v_0}, \ldots, \mu_{v_{T-1}}),~&&M_v = \mathbb{E}_{\mathbb P}[v v^\top] = \operatorname{diag}(M_{v_0}, \dots, M_{v_{T-1}}).
    \end{aligned}
\end{equation*}
Thus, the problem becomes
\begin{equation}\label{upper bound problem 1 reformulation}
\begin{array}{ccclll}
&\min\limits_{\substack{q \in \R^{pT}\\ U \in \mathcal U}} &\max\limits_{\substack{\mu_w, M_w,\\
\mu_v, M_v,
\mathbb P}} &\Tr\Big( \bigl( (UD)^\top R UD + (G+HUD)^\top Q (G+HUD) \bigr) M_w  +  U^\top \bar R U M_v\Big ) \\[-2ex]
&&&+2 q^\top(\bar R UD + G^\top QH) \mu_w + 2 q^\top \bar R U \mu_v + q^\top \bar R q\\
&&\st &\mathbb{P} \in \overline{\mathcal{B}},~ \mu_w = \EE_\PP[w], M_w = \mathbb{E}_{\mathbb P}[w w^\top],\mu_v = \EE_\PP[v],~M_v = \mathbb{E}_{\mathbb P}[  v v^\top].
\end{array}
\end{equation}

We first prove that the objective in problem~\eqref{upper bound problem 1 reformulation} is at least as large as the objective in problem~\eqref{eq: distributionally robust control problem -- affine simplified min max}. It suffices to prove that any $(\mu_w,M_w),(\mu_v,M_v)$ feasible in the inner maximization problem in~\eqref{upper bound problem 1 reformulation} are also feasible in the inner maximization problem in~\eqref{eq: distributionally robust control problem -- affine simplified min max}. This holds because any $\mathbb{P} \in \overline{\mathcal{B}}$ and $(\mu_w,M_w),(\mu_v,M_v)$ feasible in the inner maximization in~\eqref{upper bound problem 1 reformulation} must satisfy $(\mu_w,M_w) \in \mathcal{M}_{(\mu_w, M_w)}$  and $(\mu_v,M_v) \in \mathcal{M}_{(\mu_v, M_v)}$ by the definition of $\overline{\mathcal{B}}$, and therefore $(\mu_w,M_w),(\mu_v,M_v)$ are feasible in the inner maximization problem in~\eqref{eq: distributionally robust control problem -- affine simplified min max}. 

To conclude our proof, we show that the inner maximization problems in~\eqref{upper bound problem 1 reformulation} and~\eqref{eq: distributionally robust control problem -- affine simplified min max} have the same optimal value. Note that for any $(\mu_w, M_w)$ and $(\mu_v, M_v)$
feasible in the inner maximization problem in~\eqref{eq: distributionally robust control problem -- affine simplified min max}, the distribution obtained by taking independent couplings of Gaussian distributions with the same first two moments -- i.e., $\mathbb P = \PP_{x_0} \otimes (\otimes_{t=0}^{T-1}\PP_{w_t})\otimes (\otimes_{t=0}^{T}\PP_{v_t})$ where $\PP_{z} = \mathcal N(\mu_{z}, M_{z})$ for every $z \in \mathcal{Z}$ and $\PP_1 \otimes \PP_2$ denotes the independent coupling of distributions $\PP_1$ and $\PP_2$ -- would be feasible in the inner maximization problem in \eqref{upper bound problem 1 reformulation} and would result in the same objective value. 

Therefore, the relaxation is exact and the optimal values of~\eqref{upper bound problem 1},~\eqref{eq: distributionally robust control problem -- affine simplified min max}, and~\eqref{upper bound problem 1 reformulation} coincide. 
\endproof

\subsection{Proofs for \S\ref{sec:sec3_lower_bound_d} (Lower Bound for Dual)}
%
%
\begin{proof}[Proof of \Cref{prop:dual-sdp}.]
We can replace the feasible set~$\mathcal U_\eta$ of the inner minimization  with~$\mathcal U_y$ because the space~$\mathcal U_y$ of all causal output feedback policies coincides with the space~$\mathcal U_\eta$ of all causal {\em purified} output feedback policies. 

With this change, the inner minimization problem in~\eqref{eq: dual distributionally robust control problem restriction} becomes an LQG problem where the uncertainties have a known Gaussian distribution $\PP \in {\mathcal{B}}_{\mathcal{N}}$, for which classical results apply. By standard LQG theory, an \emph{affine} output feedback policy $u=U'y+q'$ for some~$U'\in\mathcal U$ and~$q'\in\R^{pT}$ is optimal (see Appendix~\S\ref{sec: appx: optimal-solution-classic-LQG} and \citealp{Bertsekas_2017}). And because any such policy can be equivalently expressed as an affine {\em purified}-output feedback policy~$u=U\eta+q$ for some~$U\in\mathcal U$ and~$q\in\R^{pT}$ (see \Cref{lemma:linear-rel-u-eta}), the feasible set of the inner minimization problem in~\eqref{eq: dual distributionally robust control problem restriction} can be taken as the set of all affine purified-output feedback policies without sacrificing optimality. 

Thus, problem~\eqref{eq: dual distributionally robust control problem restriction} has the same optimal value as problem:
\begin{equation*}
    \begin{array}{ccl} \max\limits_{\mathbb P \in \mathcal{B}_{\mathcal N}} &\min\limits_{q, U, x, u} &\mathbb E_{\mathbb P} \left[   u^\top R   u +   x^\top Q  x \right]\\
    &\st &  U \in \mathcal U, u = q + U \eta, {x} = H u + G w.
    \end{array}
\end{equation*}

Using a similar reasoning as in the proof of \Cref{prop: min max SDP upper bound problem}, we can now substitute the linear representations of $ u$ and $  x$ into the objective function and reformulate the above problem as
\begin{equation}
\begin{array}{cccll}
&\max\limits_{\substack{\mu_w, {M_w}, \\ \mu_v, {M_v},\mathbb P}} &\min\limits_{\substack{q \in \R^{pT}\\ U \in \mathcal U}}  &\Tr\Big( 
\bigl( (UD)^\top R UD + (G+HUD)^\top Q (G+HUD) \bigr) M_w 
+  U^\top \bar R U M_v\Big ) \\[-2ex]
&&&+2 q^\top(\bar R U D + G^\top QH) \mu_w + 2 q^\top \bar R U \mu_v + q^\top \bar R q,\\
&&\st &\mathbb P \in \mathcal{B}_{\mathcal N},\mu_w = \EE_{\PP}[w], ~M_w= \mathbb{E}_{\mathbb P}[w w^\top],
~\mu_v = \EE_\PP[v],~M_v = \mathbb{E}_{\mathbb P}[  v v^\top].
\end{array}
\label{eq:simplified_rewriting_dual_bnd}
\end{equation}

Lastly, we use a similar reasoning as in the proof of~\Cref{prop: min max SDP upper bound problem} to prove that problem~\eqref{eq:simplified_rewriting_dual_bnd} has the same optimal objective as problem~\eqref{eq:dr-affine-max-min}. Consider any 
$(\mu_w, {M_w}), (\mu_v, {M_v}), \mathbb P$ feasible in the outer maximization problem in~\eqref{eq:simplified_rewriting_dual_bnd}. Because $\PP\in \mathcal {\mathcal{B}}_{\mathcal N}$ is a {\em Gaussian} distribution and marginal distributions of Gaussians are also Gaussian, we can write the requirement $\mathds{D}(\PPz{1}, \hat{\mathbb{P}}_{z}) \leq \rho_{z}$ in the definition of $\mathcal{B}_{\mathcal N}$ equivalently as $\mathds{D}(\mathcal{ N}(\mu_{z},M_{z}), \hat{\mathbb P}_{z}) \leq \rho_{z}$, for any $z \in {\cal Z}$. This implies that $(\mu_w, {M_w}), (\mu_v, {M_v})$ is feasible in the outer maximization problem in~\eqref{eq:dr-affine-max-min}, proving that the optimal value in~\eqref{eq:dr-affine-max-min} is at least as large as that in~\eqref{eq:simplified_rewriting_dual_bnd}. However, for any~$(\mu_w, {M_w}), (\mu_v, {M_v})$ feasible in the outer maximization in~\eqref{eq:dr-affine-max-min}, the Gaussian distribution obtained from independent couplings $\mathbb P = \PP_{x_0} \otimes (\otimes_{t=0}^{T-1}\PP_{w_t})\otimes (\otimes_{t=0}^{T}\PP_{v_t})$ where $\PP_{z} = \mathcal N(\mu_{z}, M_{z})$ for every $z \in \mathcal{Z}$, is feasible in~\eqref{eq:simplified_rewriting_dual_bnd}, proving that the two problems have the same optimal value. 
\end{proof}

%
%
\subsection{Proofs for \S\ref{sec:sec3_optimality_linear_zero_mean} (Optimality of Linear Policies and Zero-Mean Distributions)}
\begin{proof}[Proof of \Cref{prop:assumption3_for_W_KL_M}.]
We prove the result separately, for each divergence of interest. Recall that $\hat{\mu}_z=0$, which implies that $\hat{M}_z = \hat{\Sigma}_z$, so the nominal distribution of interest is $\hat{\PP}_z = \mathcal{N}(0,\hat{\Sigma}_z)$. Because the proof is done separately for each $z\in \mathcal{Z}$, we simplify notation by dropping the subscript $z$ from quantities of interest such as $\PP_z,\hat{\PP}_z,\hat{\mu}_z, M_z, \Sigma_z$.

Consider any two Gaussian distributions $\PP_{}, \PP_{}' \in \mathcal{B}_{}$ such that $\EE_{\PP_{}}[zz^\top] = \EE_{\PP_{}'}[zz^\top] = M_{}$ and $\mu'_{} = \hat{\mu}_{} = 0$, which implies that $\Sigma'_{} =\EE_{\PP_{}'}[(z-\mu_{}')(z-\mu'_{})^\top]= M_{}$. 
For the subsequent arguments, it helps to note that
\begin{align}
    \| \mu_{} - \EE_{\hat{\PP}_{}}[z] \|^2 + \Tr ( \Sigma_{} ) - \| \mu_{}' -  \EE_{\hat{\PP}_{}}[z] \|^2 - \Tr ( \Sigma_{}' ) &= \Tr (M_{}) - \Tr (M_{}) = 0.
    \label{eq:equality_for_normals_implied_by_assump_4}
\end{align}
Also, it is useful to recall that the maps $X \mapsto \hat{\Sigma}_{}^{1/2} X \hat{\Sigma}_{}^{1/2}$ and $X \mapsto X^{\frac{1}{2}}$ are operator monotone on the cone of positive semidefinite matrices $\mathbb S^{d_{}}_{+}$, that is, are functions $f : \mathbb{S}_{+}^{d_{}} \rightarrow \mathbb{S}_{+}^{d_{}}$ that satisfy $f(X) \succeq f(Y)$  for any $X,Y \in \mathbb S_{+}^{d_{}}$ with $X \succeq Y$. (For a proof of these facts, see Theorem~1.5.9 and Theorem~4.2.3 in~\citet{ref:bhatia2009positive}.)

\medskip \noindent \textbf{The 2-Wasserstein distance $\mathds{W}$.} Recall from \S\ref{appendix:assumption2_wasserstein} and \Cref{prop:WassersteinSubsetGelbrich} that in this case, we have:
\begin{align*}
    \mathcal{M}_{(\mu_{}, M_{})} = \Bigl\{(\mu_{}, M_{}) \in \setmoments{d_{}}: ~ \bigl( \mathds{G}\bigl((\mu_{}, M_{}), (0, \hat{M}_{}) \bigr) \bigr)^2 \leq \rho_{}^2 \Bigr\}.
\end{align*}
To prove that $(\mu,M) \in  \mathcal{M}_{(\mu_{}, M_{})}$ implies that $(0,M) \in  \mathcal{M}_{(\mu_{}, M_{})}$, it therefore suffices to show that $( \mathds{G}((0, M_{}), (0, \hat{M}_{})) )^2 \leq ( \mathds{G}((\mu_{}, M_{}), (0, \hat{M}_{}) ) )^2$. To that end, using the expression of $\mathds{G}$ from \Cref{def:gelbrich-distance} and applying~\eqref{eq:equality_for_normals_implied_by_assump_4} implies that:
\begin{align}
    & \bigl( \mathds{G}\bigl((0, M_{}), (0, \hat{M}_{}) \bigr) \bigr)^2 - \bigl(\mathds{G}\bigl((\mu_{}, M_{}), (0, \hat{M}_{}) \bigr) \bigr)^2 \notag \\ & \qquad \qquad =-2 \left( \Tr \left( \hat{\Sigma}_{}^{1/2} M_{} \hat{\Sigma}_{}^{1/2}\right)^{1/2} - \Tr \left( \hat{\Sigma}_{}^{1/2} (M_{} -\mu_{}\mu_{}^\top) \hat{\Sigma}_{}^{1/2}\right)^{1/2} \right) \leq 0,
    \label{eq:first_condition_assumpt_3_W}
\end{align}
where the inequality follows because $M_{} \succeq M_{} - \mu_{}\mu_{}^\top$ and the mappings $X \mapsto \hat{\Sigma}_{}^{1/2} X \hat{\Sigma}_{}^{1/2}$ and $X\mapsto X^{1/2}$ are operator monotone and $\Tr(X) \geq \Tr(Y)$ if $X \succeq Y$. This completes the argument.

\medskip
\noindent \textbf{The KL Divergence $\mathds{K}$.} Recall from \Cref{appendix:assumption2_KL} and \Cref{prop: KL bound} that in this case, we have:
\begin{align*}
    \mathcal{M}_{(\mu_{}, M_{})} = \Bigl\{(\mu_{}, M_{}) \in \setmoments{d_{}}: \, \mathds{T}\bigl((\mu_{}, M_{}), (0, \hat{M}_{}) \bigr) \leq \rho_{} \Bigr\}.
\end{align*}
With the same proof strategy as above, using the expression of $\mathds{T}$ from \Cref{def:kl-covs} and applying~\eqref{eq:equality_for_normals_implied_by_assump_4}, we obtain:
\begin{align*}
     \mathds{T}\bigl((0, M_{}), (0, \hat{M}_{}) \bigr) - \mathds{T}\bigl((\mu_{}, M_{}), (0, \hat{M}_{}) \bigr) 
     &= - \frac{1}{2} \mu_{}^\top \hat{\Sigma}_{}^{-1} \mu_{} + \frac{1}{2} \Tr\left( (M_{} - M_{} + \mu_{} \mu_{} ^\top ) \hat{\Sigma}_{}^{-1} \right) \\
     & \qquad \qquad -\frac{1}{2} \left( \log\det\left(M_{} \hat{\Sigma}_{}^{-1} \right) - \log\det\left( (M_{} - \mu_{} \mu_{}) \hat{\Sigma}_{}^{-1} \right) \right) \\
     &\leq 0,
\end{align*}
where the inequality follows because the first two terms cancel out and the term in the last bracket is positive because $\log \det (X) \geq \log \det (Y)$ if $X \succeq Y$ and $M_{} \succeq M_{} - \mu_{}\mu_{}^\top$.

\medskip \noindent \textbf{The Moment-Based Ambiguity Set  $\mathds{M}$.} Recall from \Cref{sec:assumption 2 moment ambiguity sets} that we have:
\[ \mathcal{M}_{(\mu_{}, M_{})} = \{(\mu_{}, M_{}) \in \setmoments{d_{}}:  \mathds{M}\left((\mu_{}, M_{}), (\hat{\mu}_{}, \hat{M}_{})\right) \leq \rho_{}\}. \]
Therefore, the condition stated in \Cref{prop:assumption3_for_W_KL_M} exactly ensures that Assumption~\ref{ass:setting_mean_zero_feasible} is satisfied.  
\end{proof}

%
%
\bigskip
\begin{proof}[Proof of~Theorem~\ref{thm:worst-case-mean-zero}.]
Because problem~\eqref{DRCPdual} has the same optimal value as problem~\eqref{eq: dual distributionally robust control problem restriction} by \Cref{corollary:normal-distributions-are-optimal}, we prove the results for the optimal solution $\PP\opt \in\mathcal{B}_{\mathcal N}$ to the outer maximization in~\eqref{eq: dual distributionally robust control problem restriction} and the corresponding optimal policy $u\opt$ to the inner minimization in~\eqref{eq: dual distributionally robust control problem restriction}. (All these optimal solutions exist in view of \Cref{theorem:lower-equal-upper}.)

We first recall a few important facts from~\Cref{prop:dual-sdp} and its proof. By \Cref{prop:dual-sdp}, the optimal value in problem~\eqref{eq: dual distributionally robust control problem restriction} is the same as the optimal value in problem~\eqref{eq:dr-affine-max-min}. The proof of \Cref{prop:dual-sdp} also shows that for any $\PP \in \mathcal{B}_{\mathcal{N}}$ feasible in~\eqref{eq: dual distributionally robust control problem restriction}, the means and second moments of $w$ and $v$ evaluated under $\PP$ would be feasible in~\eqref{eq:dr-affine-max-min}, i.e., $\PP \in \mathcal{B}_{\mathcal{N}}$ implies that $(\mu_w, M_w) \in \mathcal{M}_{(\mu_w, M_w)}$ and $(\mu_v, M_v) \in \mathcal{M}_{(\mu_v, M_v)}$. Moreover, a policy of the form $u = U \eta + q$ for $U \in \mathcal U$ and $q \in \R^{pT}$ is optimal in the inner minimization problems in~\eqref{eq:dr-affine-max-min} and in~\eqref{eq: dual distributionally robust control problem restriction}. Lastly, recall that the objective in~\eqref{eq:dr-affine-max-min} evaluated for the moment pairs $(\mu_w, M_w),(\mu_v, M_v)$ and policy $u = U \eta + q$ is given by:
\begin{align*}
    f\bigl((q, U); (\mu_w, M_w), (\mu_v, M_v) \bigr) &= \Tr\left( (UD)^\top R UD + (G+HUD)^\top Q (G+HUD) \bigr) M_w +  U^\top \bar R U M_v \right ) \\ 
    & \qquad \qquad +2 q^\top(\bar R UD + G^\top QH) \mu_w + 2 q^\top \bar R U \mu_v + q^\top \bar R q.
\end{align*}
Let us fix any $U \in \mathcal{U}$ and consider the inner minimization problem in~\eqref{eq:dr-affine-max-min}, which corresponds to finding $q \in \mathbb{R}^{pT}$ to minimize the function $f$ above. Because this is an unconstrained, convex optimization problem, the solution is given by the first-order optimality condition,
\begin{equation}
q\opt(U, \mu_w, \mu_v) =-\bar{R}^{-1} \bar{K}(\mu_w, \mu_v),
\label{eq:opt_q_given_P_and_U}
\end{equation}
where $\bar{K}(\mu_w, \mu_v)= (\bar{R} U D + G^\top Q H ) \mu_w + \bar{R} U \mu_v$. In particular, the optimal choice $q\opt$ only depends on $U$ and on the means $\mu_w,\mu_v$, and moreover, $q\opt=0$ if $\mu_w=0$ and $\mu_v=0$.

In this context, consider the distribution $\PP\opt \in \mathcal{B}_\mathcal{N}$ that is optimal in~\eqref{eq: dual distributionally robust control problem restriction} and let $\mu_z\opt = \EE_{\PP \opt}[z]$ and $M_z\opt = \EE_{\PP \opt}[z z^\top]$ denote the mean and second moment matrix of $z\in \mathcal Z$ under $\PP\opt$, respectively.

We prove our desired result by contradiction. Suppose that $\PP\opt$ is such that at least one of $\mu_w\opt$ and $\mu_v\opt$ is non-zero and consider a Gaussian distribution $\tilde{\PP}$ under which the first and second moment pairs for the random vectors $w$ and $v$ are respectively given by $(0, M_w\opt)$ and $(0, M_v\opt)$. That is, $\tilde{\PP}$ has zero mean and the same second moments as $\PP \opt$. 

$\tilde{\PP}$ is feasible in problem~\eqref{eq: dual distributionally robust control problem restriction} by  Assumption~\ref{ass:setting_mean_zero_feasible}. We claim that the value of the objective in~\eqref{eq:dr-affine-max-min} corresponding to $\tilde{\PP}$ is at least as large as the objective in~\eqref{eq:dr-affine-max-min} corresponding to $\PP\opt$. Recalling the definition of $f$, we can evaluate the difference in these objectives for an arbitrary fixed $U \in \mathcal{U}$ and the optimal $q\opt$ from~\eqref{eq:opt_q_given_P_and_U} as:
\begin{align*}
    & = f\left(\left(-\bar{R}^{-1} \bar{K}(\mu_w, \mu_v), U\right); (\mu_w, M_w), (\mu_v, M_v)\right) - f\left((0, U); (0, M_w), (0, M_v)\right) \\
    & = 2 \left(q\opt(U, \mu_w, \mu_v)\right)^\top \bar{K}(\mu_w, \mu_v) + \left(q\opt(U, \mu_w, \mu_v)\right)^\top \bar{R} q\opt(U, \mu_w, \mu_v)) \\
    & = - \left(\bar{K}(\mu_w, \mu_v)\right)^\top \bar{R}^{-1} \bar{K}(\mu_w, \mu_v).
\end{align*}
The last expression is non-positive because $\bar{R} \succ 0$. Because this holds for an arbitrary $U\in \mathcal{U}$, we conclude that $\tilde{\PP}$ yields an objective in~\eqref{eq:dr-affine-max-min} at least as large as that achieved by $\PP\opt$. Therefore, it is always optimal for nature to choose a distribution $\PP\opt$ that is zero-mean.

For $\PP\opt$ with zero mean,  $q\opt=0$ by~\eqref{eq:opt_q_given_P_and_U}, so the optimal policy can be taken as $u = U\opt \eta$ for some $U \in \mathcal{U}_\eta$.
\end{proof}

%
%
\medskip
\subsection{Proofs for \S\ref{sec:sec3_worst_case_covariance} (Worst-Case Covariance)}
\subsubsection{Verifying Assumption~\ref{ass:M-gradients}.}
\label{sec: appx: verification-assumption-m-gradients}
We show that all ambiguity sets introduced in \S\ref{sec:examples} also satisfy Assumption~\ref{ass:M-gradients} (under a mild condition for the moment-based ambiguity set). This is summarized in the following result.
\begin{proposition}
    \label{prop:assumption4_for_base_examples}
    The ambiguity sets based on the divergences $\mathds{W}$ and $\mathds{K}$ introduced in \S\ref{sec:examples} satisfy Assumption~\ref{ass:M-gradients} if $\rho_z > 0$ for every $z \in \mathcal{Z}$. The moment-based ambiguity set also satisfies Assumption~\ref{ass:M-gradients} if $\rho_z > 0$ and $\mathds{M}((0,\Sigma_z),(0,\hat{\Sigma}_z))$ is differentiable in $\Sigma_z$ on $\mathbb{S}_{++}^{d_z}$ and satisfies:
    \begin{align} 
    \left.\nabla_{\Sigma_z} \mathds{M}((0,\Sigma_z),(0,\hat{\Sigma}_z))\right|_{\Sigma_z = \Sigma_1} 
    \!\!\!\!\succeq \!\!
    \left.\nabla_{\Sigma_z} \mathds{M}((0,\Sigma_z),(0,\hat{\Sigma}_z))\right|_{\Sigma_z = \Sigma_2}\Rightarrow\Sigma_1 \succeq \Sigma_2,~ \forall\, \Sigma_1, \Sigma_2 \in \mathbb S_+^{d_z}, ~ \forall \, z\in \mathcal{Z}.
    \label{eq:condition_M_assumption4}
    \end{align}
\end{proposition}

\begin{proof}{Proof.}
To simplify notation, we subsequently drop the subscript $z$ from $\Sigma_z, \hat{\Sigma}_z$, $\rho_z$ or $d_z$.

%
%
\smallskip
\noindent {\bf Wasserstein ambiguity sets.} Set $\gz(\Sigma) = (\mathds{G}((0, \Sigma),(0, \hat{\Sigma})))^2 - \rho^2 $ where $\mathds{G}$ is the Gelbrich distance. We claim that all requirements in Assumption~\ref{ass:M-gradients} are satisfied if $\rho > 0$. We have:
\[ \gz(\Sigma) = \Tr\Bigl(\Sigma + \hat{\Sigma} - 2\bigr( \hat{\Sigma}^{1/2} \Sigma \hat{\Sigma}^{1/2}\bigr)^{1/2}\Bigr) - \rho^2.\] 
\Cref{prop:WassersteinSubsetGelbrich} and the discussion that verified Assumption~\ref{ass:general_assumption_about_M_function}-(ii) imply that the set $\mathcal{M}_\Sigma = \bigl\{ \Sigma \in \mathbb{S}^{d}_+ : \gz(\Sigma) \leq 0 \bigr\}$ is convex and compact. Moreover, $\gz$ is differentiable in $\Sigma$ on $\mathbb{S}_{++}^{d}$ because $\mathds{G}$ is. Requirements~(i) and~(ii) are readily satisfied because 
$\gz$ and the Gelbrich distance $\mathds{G}$ are minimized at $\Sigma = \hat{\Sigma}$ and we have $\gz(\hat{\Sigma}) = -\rho^2 < 0$ when $\rho > 0$. To check~(iii), note that the gradient $\gradgz$ is:
\[\gradgz(\Sigma) = I - \hat\Sigma^\frac{1}{2} (\hat\Sigma^\frac{1}{2} \Sigma \hat\Sigma^\frac{1}{2})^{-\frac{1}{2}}\hat\Sigma^\frac{1}{2}. \] 
Fix $\Sigma_1, \Sigma_2 \in \mathbb S_+^{d_z}$ and suppose that $\gradgz(\Sigma_1) \succeq \gradgz(\Sigma_2)$. This implies
\[\hat\Sigma^\frac{1}{2} (\hat\Sigma^\frac{1}{2} \Sigma_1 \hat\Sigma^\frac{1}{2})^{-\frac{1}{2}}\hat\Sigma^\frac{1}{2} \preceq \hat\Sigma^\frac{1}{2} (\hat\Sigma^\frac{1}{2} \Sigma_2 \hat\Sigma^\frac{1}{2})^{-\frac{1}{2}}\hat\Sigma^\frac{1}{2} ~\Rightarrow~ (\hat\Sigma^\frac{1}{2} \Sigma_1 \hat\Sigma^\frac{1}{2})^{-\frac{1}{2}} \preceq (\hat\Sigma^\frac{1}{2} \Sigma_2 \hat\Sigma^\frac{1}{2})^{-\frac{1}{2}} ~\Rightarrow~ \Sigma_1 \succeq \Sigma_2, \]
where the first implication follows because pre- and post-multiplying by $\hat{\Sigma}^{-\frac{1}{2}}$ preserves ordering, and the second implication follows because the operator $X \mapsto X^{-1/2}$ reverses ordering on $\mathbb{S}_+^d$ (i.e., $X_1 \succeq X_2$ implies $(X_1)^{-1/2} \preceq (X_2)^{-1/2}$) and pre- and post-multiplying by $\hat{\Sigma}^{1/2}$ preserves ordering.

%
%
\smallskip
\noindent \textbf{Kullback-Leibler ambiguity sets.} Set $\gz(\Sigma) = \mathds{T}((0, \Sigma),(0, \hat{\Sigma})) - \rho$ where $\mathds{T}$ is the KL-type divergence between moments introduced in \Cref{def:kl-covs}. We claim that all requirements in Assumption~\ref{ass:M-gradients} are satisfied if $\rho > 0$. We have:
\[ \gz(\Sigma) = \frac{1}{2}\Bigl( \Tr\bigl(\Sigma \hat{\Sigma}^{-1}\bigr) - \log\det\bigl(\Sigma \hat{\Sigma}^{-1} \bigr) - d \Bigr) - \rho.\] 
\Cref{prop: KL bound} and the discussion that verified Assumption~\ref{ass:general_assumption_about_M_function}-(ii) imply that the set $\mathcal{M}_\Sigma = \bigl\{ \Sigma \in \mathbb{S}^{d}_{++} : \gz(\Sigma) \leq 0 \bigr\}$ is convex and compact. Moreover, $\gz$ is differentiable in $\Sigma$ on $\mathbb{S}_{++}^{d}$ because the $\Tr(\cdot)$ and $\log \det (\cdot)$ operators are differentiable on $\mathbb{S}_{++}^{d}$. Requirements~(i) and~(ii) are readily satisfied because 
$\gz$ and $\mathds{T}$ are both minimized at $\Sigma = \hat{\Sigma}$ and we have $\gz(\hat{\Sigma}) = -\rho < 0$ when $\rho > 0$. To check~(iii), note that the gradient $\gradgz$ can be written using matrix-calculus rules as:
\[ \gradgz(\Sigma) = \frac{1}{2} (\hat{\Sigma}^{-1} - \Sigma^{-1}). \]
That $\gradgz(\Sigma_1) \succeq \gradgz(\Sigma_2)$ implies $\Sigma_1 \succeq \Sigma_2$ follows because $X \mapsto X^{-1}$ is order-reversing on $\mathbb{S}^{d}_{++}$.

%
%
\smallskip
\noindent \textbf{Moment-Based ambiguity sets.} Define 
$\gz(\Sigma) = \mathds{M}((0,\Sigma),(0,\hat{\Sigma})) - \rho$. The conditions follow immediately by recognizing that $\mathds{M}$ being differentiable and satisfying $\mathds{M}(m,m)=0$ for all $m \in \setmoments{d}$ (as required in its definition in \S\ref{sec:examples}) imply that $\gradgz(\hat{\Sigma})=0$ and $\gz(\hat{\Sigma}) < 0$ for $\rho > 0$. Condition~(iii) is exactly equivalent to the stated condition~\eqref{eq:condition_M_assumption4} on $\mathds{M}$. 
\end{proof}

%
%
\bigskip
\begin{proof}[Proof of \Cref{thm:optimal-covs-are-higher}.]
Problem~\eqref{DRCPdual} has the same optimal value as problem~\eqref{eq:dr-affine-max-min}, so we consider the latter formulation. Recall that $\hat{\mu}_z=0$, so \Cref{thm:worst-case-mean-zero} allows us to restrict attention to zero-mean, Gaussian distributions $\PP_z$ for nature and linear control policies $u = U \eta$. 

We prove a slightly stronger result: that for any linear control policy, nature's optimal choice satisfies $\Sigma\opt_{z}\succeq \hat\Sigma_{z}$, for every $z \in \mathcal{Z}$. Consider any $U\in \mathcal U$ and define auxiliary matrices
\[\Upsilon_w = (UD)^\top R UD + (G+HUD)^\top Q (G+HUD)\quad \text{and} \quad \Upsilon_v = U^\top R U + (HU)^\top R HU.\]
Because $R \succ 0$ and $Q \succeq 0$, we have $\Upsilon_w \succeq 0$ and $\Upsilon_v \succeq 0$. Substituting into the objective of~\eqref{eq:dr-affine-max-min} yields:
\begin{equation}
    \Tr(\Upsilon_w \Sigma_w) + \Tr(\Upsilon_v \Sigma_v). 
    \label{eq:objective_with_upsilon_w_v}
\end{equation}
Function~\eqref{eq:objective_with_upsilon_w_v} is separable in $\Sigma_w$ and $\Sigma_v$. In fact, because $\Sigma_w$ and $\Sigma_v$ are block-diagonal matrices (because the random noise terms $z \neq z' \in \mathcal{Z}$ are uncorrelated), this objective further decomposes based on each component $z \in \mathcal{Z}$. For instance, with $\mathcal Z_1 = \{x_0, w_0, \ldots, w_{T-1}\}$, we have: 
\[\Tr(\Upsilon_w \Sigma_w) = \sum\limits_{z\in\mathcal Z_1} \Upsilon_w^{(z)}\Sigma_{z},\]
where each $\Upsilon_w^{(z)}$ is the diagonal block of $\Upsilon_w$ corresponding to block $\Sigma_z$ in $\Sigma_w\opt$. A similar decomposition applies to the objective over $\Sigma_v$, with $\Upsilon_v^{(z)}$ denoting the blocks in $\Upsilon_w$. Because the constraints $(\Sigma_w, \Sigma_v) \in \mathcal{M}_{\Sigma_w} \times \mathcal{M}_{\Sigma_v}$ are also separable by $z\in \mathcal{Z}$, problem~\eqref{eq:dr-affine-max-min} is separable by $z\in \mathcal{Z}$:
\begin{equation}
\max\limits_{(\Sigma_w, \Sigma_v) \in \mathcal{M}_{\Sigma_w} \times \mathcal{M}_{\Sigma_v}} \Tr(\Upsilon_w \Sigma_w + \Upsilon_v  \Sigma_v) = \sum_{z \in \mathcal{Z}_1} \max\limits_{\Sigma_z \in \mathcal{M}_{\Sigma_z}} \Tr(\Upsilon_w^{(z)} \Sigma_z) + \sum_{z \in \mathcal{Z} \setminus \mathcal{Z}_1} \max\limits_{\Sigma_z \in \mathcal{M}_{\Sigma_z}} \Tr(\Upsilon_v^{(z)} \Sigma_z).
\label{eq:max-problem-decomposes}
\end{equation}
We prove the result for a given $z\in \mathcal{Z}$. Without loss of generality, consider the problem:
\begin{equation*}
\begin{array}{cclll}
\textup{(SDP)}_1 \qquad
    &\max\limits_{\Sigma_z} &\Tr(\Upsilon^{(z)}_w \Sigma_z) \\
    &\st & \gz(\Sigma_z) \leq 0 \\
    &&\Sigma_z \succeq 0.
\end{array}
\end{equation*}
Instead of solving this problem, we consider the following relaxation:
\begin{equation*}
\begin{array}{cclll}
    \textup{(SDP)}_2 \qquad  &\max\limits_{\Sigma_z} &\Tr(\Upsilon^{(z)}_w \Sigma_z) \\
    &\st & \gz(\Sigma_z) \leq 0 \\
    &&\Sigma_z \succeq \epsilon I,
\end{array}
\end{equation*}
where $\epsilon$ satisfies $0 < \epsilon < \lambda_{\min}(\hat{\Sigma}_z)$. Subsequently, we will prove that an optimal solution to (SDP)$_2$ exists satisfying $\Sigma_z\opt \succeq \hat{\Sigma}_z$, which (because $\hat{\Sigma}_z \succ \epsilon I$) will imply that the last constraint in (SDP)$_2$ will not be binding and $\Sigma_z\opt$ will also be optimal in (SDP)$_1$.

Consider problem (SDP)$_2$. Because $\Upsilon_w \succeq 0$, its principal sub-matrices are positive semidefinite, so $\Upsilon_w^{(z)} \succeq 0$. Let $\Sigma_z\opt$ be a maximizer in the optimization problem above. (The Weierstrass Theorem, which applies because the objective is linear and Assumption~\ref{ass:general_assumption_about_M_function} guarantees that the feasible set $\mathcal{M}_{\Sigma_z}$ is compact, implies that a maximizer exists.) We prove that a maximizer exists so that $\Sigma_{w}\opt \succeq \hat\Sigma_w$.

\textbf{(a)} If $\Upsilon_w^{(z)} =0$, we can replace $\Sigma_z\opt $ with $\hat\Sigma_z$, which satisfies $\Sigma_z\opt \succeq \hat\Sigma_z$ and is optimal.

\textbf{(b)} If $\Upsilon_w^{(z)} \neq 0$, with $\lambda_z$ and $\Lambda_z$ as the dual variables for the constraints above, the KKT conditions for this maximization problem are:
%
%
\begin{align}
    &-\Upsilon_w^{(z)} + \lambda_z \, \gradgz(\Sigma_{z}\opt) -\Lambda_z = 0 && \tag{Stationarity} \label{eq:st-cond}\\
    &\lambda_z \, \gz (\Sigma_{z}\opt) = 0,~\Tr(\Lambda_z^\top \Sigma_{z}\opt) = 0 &&\tag{Complementary Slackness}\label{eq:cs-cond}\\
    &\lambda_z \geq 0, \Lambda_z
    \in \mathbb S_+^n &&\tag{Dual Feasibility}\label{eq:df-cond}
\end{align}
In writing these KKT conditions, the use of the gradient $\gradgz$ was valid because the feasible set of (SDP)$_2$ is contained in $\mathbb{S}_{++}^{d_z}$ and $\gz$ is differentiable on the latter set, by Assumption~\ref{ass:M-gradients}. Moreover, the KKT conditions are necessary for optimality here because we have a convex optimization problem and Slater's condition is satisfied because $\gz(\hat{\Sigma}_z) < 0$ due to Assumption~\ref{ass:M-gradients}-(i). 

The \ref{eq:st-cond} condition implies:
\begin{align*}
    \lambda_z \, \gradgz(\Sigma_{z}\opt) = \Upsilon^{(z)}_w + \Lambda_z. 
\end{align*}

Because $\Upsilon_w^{(z)} \succeq 0$ and $\Lambda_z\succeq 0$, we have $\lambda_z \, \gradgz(\Sigma_z\opt) \succeq 0$. Moreover, if $\lambda_z=0$ above, we must have $\Upsilon^{(z)}_w = \Lambda_z = 0$, which contradicts our standing assumption that $\Upsilon_w^{(z)} \neq 0$.

If $\lambda_z \neq 0$, because $\gradgz(\Sigma_{z}\opt) \succeq 0$ by the argument above, we have: 
\begin{align}
\gradgz(\Sigma_z\opt) \succeq 0 = \gradgz (\hat\Sigma_z),
\label{eq:proof_assum4_gradient_ineq}
\end{align}
where the last equality follows from Assumption~\ref{ass:M-gradients}-(ii), which together with the convexity and differentiability of $\gz$ implies that $\gradgz(\hat{\Sigma}_z)=0$. But then, \eqref{eq:proof_assum4_gradient_ineq} and Assumption~\ref{ass:M-gradients}-(iii) imply $\Sigma_{z}\opt \succeq \hat\Sigma_z$, which completes our proof.  
\end{proof}

%
%
\section{Proofs for \Cref{sec:DR-LQG-algorithm}}
\label{sec:app:proofs_section_4}

\subsection{Concavity and $\beta$-Smoothness for the Objective}
\label{sec:app:beta_smoothness}
Recall the notation from \Cref{sec:DR-LQG-algorithm} whereby $f(\Sigma_w,\Sigma_v)$ is the optimal value of the classical LQG problem for a Gaussian distribution with zero mean and covariances $\Sigma_w, \Sigma_v$. 

%
%
%
\begin{definition}[$\beta$-smoothness]
\label{def:beta-smooths}
For $\mathcal{M}_1,\mathcal{M}_2 \subseteq \mathbb S_+^d$, the function $f: \mathcal{M}_1 \times \mathcal{M}_2 \to \R$ is called $\beta$-smooth for some $\beta>0$ if
\[
    |\nabla f(\Sigma_1,\Sigma_2)-\nabla f(\Sigma_1',\Sigma_2')| \leq \beta\left( \|\Sigma_1-\Sigma_1'\|_{\mathrm{F}}^2+\|\Sigma_2- \Sigma_2'\|_{\mathrm{F}}^2 \right)^\frac{1}{2} \quad \forall\, \Sigma_1, \Sigma_2\in \mathcal{M}_1, \Sigma_2, \Sigma_2'\in \mathcal{M}_2.
\]
\end{definition}
\proof[Proof of Proposition \ref{prop:structure-of-f}.]
The function~$f(\Sigma_w, \Sigma_v)$ is concave because the objective function of the inner minimization problem in~\eqref{eq:dr-affine-max-min} is linear (and hence concave) in~$\Sigma_w$ and~$\Sigma_v$ and because concavity is preserved under minimization. 
By \Cref{corol:worst-case-Kalman-elliptical}, the value of $f(\Sigma_w,\Sigma_v)$ is given by~\eqref{eq:lqg-cost}, where $\Lambda_t$
for $t \in [T-1]$, is a function of~$(\Sigma_w,\Sigma_v)$ defined recursively through the Kalman filter equations
\eqref{eq:kalman-cov-updates}. Note that all inverse matrices in \eqref{eq:kalman-cov-updates}
are well-defined because any $\Sigma_v$ {$\in \mathcal{M}_{\Sigma_v}^+$} is strictly positive definite. Therefore, $\Lambda_t$ constitutes a proper rational function, i.e., a ratio of two polynomials with the polynomial in the denominator being strictly positive, for every~$t\in[T-1]$. Thus, $f(\Sigma_w, \Sigma_v)$ is infinitely often continuously differentiable on a neighborhood of $\mathcal{M}_{\Sigma_w} \times \mathcal{M}_{\Sigma_v}^{+}$.

Because $f(\Sigma_w,\Sigma_v)$ is concave and at least twice continuously differentiable, it is $\beta$-smooth on $\mathcal{M}_{\Sigma_w} \times \mathcal{M}_{\Sigma_v}^+$ if and only if the largest eigenvalue of the negative of the Hessian matrix, 
$\lvert\lambda_{\max}\left( -\nabla^2 f(\Sigma_w,\Sigma_v) \right)\rvert$, is bounded above by~$\beta$ on $\mathcal{M}_{\Sigma_w} \times \mathcal{M}_{\Sigma_v}^+$. The value $\lvert \lambda_{\max}\left( - \nabla^2(f(\Sigma_w,\Sigma_v)) \right)\rvert$ coincides with the spectral norm of~$\nabla^2 f(\Sigma_w,\Sigma_v)$, so we can set:
\begin{equation}\label{eq: max eigenvalue as spectral norm}
    \beta = \sup_{\Sigma_w \in \mathcal{M}_{\Sigma_w}, \Sigma_v \in \mathcal{M}_{\Sigma_v}^+} \Vert \nabla^2 f(\Sigma_w,\Sigma_v) \Vert_2,
\end{equation}
where $\Vert \cdot \Vert_2$ denotes the spectral norm. The supremum in the above maximization problem is finite and attained thanks to Weierstrass' theorem, which applies because $f(\Sigma_w,\Sigma_v)$ is twice continuously differentiable and the spectral norm is continuous, while the sets~$\mathcal{M}_{\Sigma_w}$ and~$\mathcal{M}_{\Sigma_v}^+$ are compact by Assumption~\ref{ass:general_assumption_about_M_function}. This observation completes the proof.

\endproof

\subsection{Bisection Algorithm for the Linearization Oracle}
\label{sec:bisection}
We now show that the direction-finding subproblem~\eqref{eq:linearization-oracle} can be solved efficiently via bisection for the Wasserstein and KL ambiguity sets. 
\subsubsection{Wasserstein Ambiguity.}
\label{sec:bisection:wass} We establish that~\eqref{eq:linearization-oracle} can be reduced to the solution of $2T+1$ univariate algebraic equations. To this end, let $\Gamma_z = \nabla_{\Sigma_z} f(\Sigma^{(k)}_{w}, \Sigma^{(k)}_v)$ denote the gradient at the current iterate with respect to~$\Sigma_z$. Note that a routine calculation (see proof of \Cref{thm:optimal-covs-are-higher}) shows that $\Gamma_z$$\succeq 0$. If $\Gamma_z=0$, then $\Sigma_z\opt=\hat\Sigma_z$ is trivially optimal in~\eqref{rem:efficient_computation_wasserstein_KL}. From now on we thus assume that~$\Gamma_z\neq 0$.

%
%
\begin{proposition}{\cite[Proposition A.4]{nguyen2023bridging}}
If $\Gamma_z \in \mathbb S_+^d$, $\Gamma_z \neq 0$, $\hat \Sigma_z \in \mathbb S_{++}^{d_z}$, and $\rho_z >0$, then the unique optimal solution to the problem
\begin{equation}
\begin{array}{cclll}
     &\max &\langle \Gamma_z, \Sigma_z - \Sigma_z^{(k)} \rangle \\
     &\st & \Sigma_z \in \mathbb S_+^d \\
     && \mathds G(\Sigma_z, \hat \Sigma_z) \leq \rho_z\\
     && \Sigma_z \succeq \lambda_{\min}(\hat \Sigma_z) I
\end{array}
\label{eq:linear-oracle-problem-wass}
\end{equation}
is ~$\Sigma_z\opt = (\gamma\opt)^2 (\gamma\opt I  - \Gamma_z)^{-1}\hat \Sigma_z(\gamma\opt I - \Gamma_z)^{-1}$,
where $\gamma\opt$ is the unique solution of
\begin{equation}
\rho_z^2 - \langle \hat \Sigma_z , (I - \gamma\opt (\gamma\opt I - \Gamma_z)^{-1})^2\rangle = 0
\label{eq:algebraic-gamma-exp}
\end{equation}
satisfying
\begin{equation}\underline\gamma = \lambda_1(1+ (p_1^\top \hat \Sigma_z p_1)^\frac{1}{2} / \rho_z) \leq \gamma\opt \leq \lambda_1(1+ \Tr(\hat \Sigma_z)^\frac{1}{2} /\rho_z) = \overline \gamma,
\label{eq:upper-lower-gamma-wass}
\end{equation}
where $\lambda_1 = \lambda_{\max}(\Gamma_z)$ and $p_1$ is an eigenvector for $\lambda_1$.
\end{proposition}
In practice, we need to solve the algebraic equation~\eqref{eq:algebraic-gamma-exp} numerically.
The numerical error in approximating~$\gamma\opt$ should be contained to ensure that~$\Sigma_z\opt$ approximates the exact maximizer of problem~\eqref{eq:linear-oracle-problem-wass}. 
The next proposition shows that for any tolerance~$\delta \in (0,1)$, a $\delta$-approximate solution of~\eqref{eq:linear-oracle-problem-wass} can be computed with an efficient bisection algorithm.
\begin{proposition}{\cite[Theorem~6.4]{nguyen2023bridging}}
For any fixed $\rho_z >0, \hat \Sigma_z \in \mathbb{S}_{++}^d$ and $\Gamma_z \in \mathbb{S}_{+}^d, \Gamma_z \neq 0$, define 
$\mathcal G_{\Sigma_z} = \{\Sigma_z \in \mathbb S_+^d : \mathds G(\Sigma_z, \hat \Sigma_z) \leq \rho_z, \hat \Sigma_z \succeq \lambda_{\rm{min}}(\hat \Sigma_z)\}$
as the feasible set of problem~\eqref{eq:linear-oracle-problem-wass}, and let $\hat{\Sigma}_z\in \mathcal G_{\Sigma_z}$ be any reference covariance matrix. Then, Algorithm~\ref{alg:bisection} with $\delta \in(0,1)$, $\varphi(\gamma)=\gamma(\rho_z^2+\left\langle\gamma(\gamma I-\Gamma_z\right)^{-1}-I, \hat \Sigma_z  \rangle) - \langle \Gamma_z, \Sigma^{(k)}_z \rangle$, $L(\gamma) = (\gamma)^2  (\gamma I - \Gamma_z)^{-1} \hat \Sigma_z (\gamma I - \Gamma_z)^{-1}$ for any $\gamma>\lambda_{\max }(\Gamma_z)$, $\overline\gamma$ and $\underline \gamma$ as defined in~\eqref{eq:upper-lower-gamma-wass}
returns in finite time a matrix $\Sigma_z^\delta \in \mathbb{S}_{+}^d$ that satisfies:  $\Sigma_z^\delta \in \mathcal G_{\Sigma_z}$ (feasible) and $\langle \Gamma_z, \Sigma_z^\delta - \Sigma_z^{(k)} \rangle \geq  \delta \max _{\Sigma_z \in \mathcal{G}_{\Sigma_z}}\langle \Gamma_z, \Sigma_z - \Sigma_z^{(k)}\rangle$ ($\delta$-Suboptimal).
\label{prop:bisection}
\end{proposition}
{\small
\begin{algorithm}
\hspace*{\algorithmicindent} \textbf{Input:} nominal covariance matrix $\hat \Sigma_z \in \mathbb S_{++}^{d_z}$, radius $\rho_z\in \R_{++}$,  \\
\hspace*{\algorithmicindent} \quad \quad \quad reference covariance matrix $\Sigma_z \in \mathcal{M}_z$,\\
\hspace*{\algorithmicindent} \quad \quad \quad gradient matrix~$\Gamma_z\in \mathbb S_+^d$, $\Gamma_z\neq 0$, precision~$\delta \in (0,1)$,\\
\hspace*{\algorithmicindent} \quad \quad \quad dual objective function $\phi : \R \to \R$ and estimation function $L(\gamma) : \R \to \mathbb S^{d}_+$ \\
\hspace*{\algorithmicindent} \quad \quad \quad upper bound~$\overline \gamma$ and lower bound $\underline \gamma$ 
\caption{Bisection algorithm to compute~$L_{\Sigma_z}^\delta$}
\label{alg:bisection}
\begin{algorithmic}[1]
\State{\textbf{repeat}}
\State{\quad set $\gamma \gets (\overline \gamma + \underline \gamma ) /2$ }
\State{\quad \textbf{if} $\frac{\diff \phi}{\diff \gamma}(\gamma) < 0$~\textbf{then}~set $\underline \gamma \gets \gamma$~\textbf{else}~set $\overline \gamma \gets \gamma$~\textbf{endif}}
\State{\textbf{until} $\frac{\diff \phi}{\diff\gamma}(\gamma) > 0$ and $\langle  L(\gamma )- \Sigma^{(k)}_z, \Gamma_z \rangle \geq  \delta \phi(\gamma)$}
\end{algorithmic}
\hspace*{\algorithmicindent} \textbf{Output}: $L(\gamma)$
\end{algorithm}}
In summary, for any~$z \in \mathcal Z$, Algorithm~\ref{alg:bisection} computes a~$\delta$-approximate solutions to the direction-finding subproblem~\eqref{eq:linearization-oracle} with $\Gamma_z = \nabla_{\Sigma_z} f(\Sigma^{(k)}_{w}, \Sigma^{(k)}_v)$.

\subsubsection{Kullback-Leibler Divergence}
We first establish that~\eqref{eq:linearization-oracle} can be reduced to the solution of a univariate algebraic equation. Let $\Gamma_z = \nabla_{\Sigma_z} f(\Sigma^{(k)}_{w}, \Sigma^{(k)}_v)$ denote the gradient at the current iterate.

\begin{theorem}
\label{thm:algebraic-form-kl}
    If $\hat \Sigma_z \in \mathbb S_{++}^{d_z}$, $\Gamma_z \in \mathbb S_+^d$, $\Gamma_z \neq 0$, $\Sigma_z^{(k)} \in \mathbb S^d_+$,  and $\rho_z \in \R_{++}$, then the problem
    \begin{equation}
    \begin{array}{ccll}
         &\max\limits_{\Sigma_z \in \mathbb S_+^d} &\langle \Gamma_z, \Sigma_z -\Sigma_z^{(k)} \rangle \\
         &\st & \mathds T(\Sigma_z, \hat{\Sigma}_z) \leq \rho_z\\
    \end{array}
    \label{eq:linear-oracle-problem-kl}
    \end{equation}
    is uniquely solved by~$L\opt_{\Sigma_z}=$  
    $\gamma\opt(\gamma\opt \hat \Sigma_z^{-1} -  \Gamma_z )^{-1}$,
    where
    $\gamma \opt$ is the unique solution of the following nonlinear equation
    \begin{equation}2 \rho_z = \log\det(I - \hat \Sigma_z \Gamma_z / \gamma\opt) +  \Tr(( \gamma\opt I - \hat \Sigma_z \Gamma_z)^{-1}\hat \Sigma_z \Gamma_z),
    \label{eq:kl-dual-var-algebraic-equation}
    \end{equation}
satisfying
\begin{equation}\lambda_1 < \gamma \opt\leq  \lambda_1\left (1 + \frac{d}{\rho_z}\right),
\label{eq:upper-lower-gamma-kl}
\end{equation}
where $\lambda_1$ is the largest eigenvalue of $\hat \Sigma^{\frac{1}{2}} \Gamma_z \hat \Sigma^{\frac{1}{2}}$.
\end{theorem}

The proof relies on Lemma~A.1 in \cite{taskesen2021sequential}, restated below for completeness.
\begin{lemma}{\cite[Lemma A.1]{taskesen2021sequential}}
\label{lemma:kl-ball-dual} 
Fix $\hat \Sigma \in \mathbb S_{++}^{d_z}$, then for any~$F  \in \mathbb S^d$ and $\rho> 0$, the optimization problem
\begin{equation}
    \begin{array}{ccll}
         &\sup\limits_{\Sigma \in \mathbb S^d_{++}} & \Tr(F \Sigma)  \\
         & \st & \mathds T(\Sigma, \hat \Sigma) \leq \rho
    \end{array}
    \label{eq:kl-ball-lin-obj-primal}
\end{equation}
admits the strong dual formulation
\begin{equation}
    \begin{array}{ccllll}
&\inf\limits_{\gamma \in \R} & 2 \gamma \rho - \gamma \log\det(I - \hat \Sigma F / \gamma)  \\
         & \st &\gamma \geq  0, ~ \gamma \hat \Sigma^{-1} \succ F.
    \end{array}
\end{equation}
\end{lemma}
\proof[Proof of Theorem~\ref{thm:algebraic-form-kl}.]
 By \Cref{lemma:kl-ball-dual}, 
 for any~$\hat \Sigma_z \in \mathbb S_{++}^{d_z}$, problem~\eqref{eq:linear-oracle-problem-kl} admits the strong dual
\begin{equation}
\begin{array}{cclll}
      & \inf\limits_{\gamma \in \R} &2 \gamma \rho_z - \gamma \log\det(I - \hat\Sigma_z \Gamma_z /\gamma)- \langle \Gamma_z, \Sigma_z^{(k)} \rangle \\
     & \st & \gamma \geq  0,~ \gamma \hat \Sigma_z^{-1} \succ \Gamma_z.
\end{array}    
\label{eq:kl-dual}
\end{equation} 
For ease of exposition, we let ~$h(\gamma) = 2 \gamma \rho_z - \gamma \log\det(I - \hat \Sigma_z \Gamma_z / \gamma)$, which is the objective function of problem~\eqref{eq:kl-dual} without the constant term $\langle \Gamma_z, \Sigma_z^{(k)} \rangle$. The gradient of~$h(\gamma)$ satisfies
\[\nabla h(\gamma) =  2\rho_z- \log\det(I - \hat \Sigma_z \Gamma_z / \gamma) - \Tr((\gamma - \hat \Sigma_z \Gamma_z)^{-1} \hat \Sigma_z\Gamma_z ).
\]
By the above expression of $\nabla h(\gamma)$ and the strict convexity of $h(\gamma)$ for $\gamma \geq  0$ and $\gamma \hat \Sigma_z^{-1} \succ \Gamma_z$ 
, the value $\gamma \opt$ that solves~\eqref{eq:kl-dual-var-algebraic-equation} is the unique minimizer of~\eqref{eq:kl-dual}.

Now, we show that $\Sigma_z\opt$ obtained as $\gamma\opt (\gamma\opt \hat \Sigma_z^{-1} - \Gamma_{z})^{-1}$ is feasible in~\eqref{eq:linear-oracle-problem-kl}, that is, it satisfies
\begin{equation}\Tr(\Sigma_z\opt \hat\Sigma_z^{-1}) - \log \det(\Sigma_z\opt \hat \Sigma_z^{-1}) - d_z \leq 2 \rho_z.
\label{eq:kl-feasible-inequality}
\end{equation}
Through the Woodbury matrix identity we have
\[ \Sigma_z\opt \hat \Sigma_z^{-1} =\gamma\opt(\gamma\opt\hat \Sigma_z^{-1} - \Gamma_z)^{-1}\hat \Sigma_z^{-1}= I + \hat \Sigma_z(\gamma\opt - \Gamma_z \hat \Sigma_z)^{-1} \Gamma_z.\]
By replacing the $\Sigma_z\opt\hat \Sigma_z^{-1}$ term inside the trace operator in~\eqref{eq:kl-feasible-inequality}, we have
\[\Tr(I  + \hat \Sigma_z (\gamma\opt - \Gamma_z \hat \Sigma_z)^{-1} \Gamma_z ) - \log\det(\gamma\opt (\gamma\opt \hat \Sigma_z^{-1} - \Gamma_z)^{-1} \hat \Sigma_z^{-1}) - d_z = 2\rho_z,\]
where the equality follows because $\gamma\opt$ solves~\eqref{eq:kl-dual-var-algebraic-equation}.

Next, we show that $\Sigma_z\opt = \gamma\opt(\gamma\opt \hat \Sigma_z^{-1} - \Gamma_z)^{-1}$ is optimal in~\eqref{eq:linear-oracle-problem-kl}. To this end, note that the objective value of $\Sigma_z\opt$ in~\eqref{eq:linear-oracle-problem-kl} coincides with the optimal value of its strong dual \eqref{eq:kl-dual} because
\begin{align*}
\Tr(\Sigma_z\opt \Gamma_z ) &= \Tr((I - \hat \Sigma_z
\Gamma_z / \gamma\opt)^{-1} \hat \Sigma_z \Gamma_z)\\
&= \gamma\opt \rho_z - \gamma\opt \log\det ( I - \hat \Sigma_z\Gamma_z / \gamma\opt)\\
&=\inf\limits_{\gamma\geq  0, \, \gamma \hat \Sigma_z^{-1} \succ \Gamma_z} \gamma \rho_z - \gamma\log\det (I - \hat \Sigma_z\Gamma_z / \gamma) ,
\end{align*}
where the second equality follows because $\gamma\opt$ solves~\eqref{eq:kl-dual-var-algebraic-equation}.


It remains to show the upper and lower bounds on~$\gamma\opt$. The lower bound on $\gamma\opt$ follows because $\gamma\opt$ is feasible in~\eqref{eq:kl-dual} and thus satisfies $\gamma\opt\hat\Sigma^{-1}_z \succ \Gamma_z$. 
For any~$\gamma\opt$ such that $\gamma\opt\hat \Sigma_z^{-1} \succ \Gamma_z $, the algebraic equation~\eqref{eq:algebraic-gamma-exp} yields
\begin{align}
   0 &=  \rho_z - \log\det (I - \hat \Sigma_z \Gamma_z / \gamma\opt) - \Tr((\gamma\opt I - \hat \Sigma_z \Gamma_z)^{-1} \hat \Sigma_z \Gamma_z) \nonumber\\
   &>\rho_z - \Tr((\gamma\opt I - \hat\Sigma_z \Gamma_z)^{-1} \hat\Sigma_z \Gamma_z)\nonumber\\
   &= \rho_z - \Tr((\gamma\opt I - \hat \Sigma_z^\frac{1}{2} \Gamma_z \hat\Sigma_z^\frac{1}{2} )^{-1}\hat \Sigma_z^\frac{1}{2} \Gamma_z \hat\Sigma_z^\frac{1}{2} ) \nonumber\\
   &=\rho_z - \sum\limits_{i=1}^{d_z} \frac{\lambda_i}{\gamma\opt - \lambda_i},
   \label{eq:ae-upper-bound}
\end{align}
where the inequality holds because the eigenvalues of $I - \hat \Sigma_z \Gamma_z / \gamma\opt$ are in $(0,1)$ under the condition $\gamma\opt\hat \Sigma_z^{-1} \succ \Gamma_z $, and this implies that $\log \det(I - \hat \Sigma_z \Gamma_z / \gamma\opt)$ is negative. In the last expression, $\lambda_1, \ldots, \lambda_{d_z}$ denote the eigenvalues of $\hat \Sigma_z^\frac{1}{2} \Gamma_z \hat\Sigma_z^\frac{1}{2} $ indexed in descending order. We thus have
\begin{equation*}
    \rho_z < \sum\limits_{i=1}^{d_z} \frac{\lambda_i}{\gamma\opt - \lambda_i} \leq \frac{d_z \lambda_1}{\gamma\opt -\lambda_1},
\end{equation*}
where the second inequality holds immediately because $\gamma^\star > \lambda_1$. The above inequality implies the second inequality in \eqref{eq:upper-lower-gamma-kl}, and therefore the correctness of the upper bound. 
This observation concludes the proof.  
\endproof

%
%
\medskip
%
%
\section{Proofs for \S\ref{sec:infinite-horizon}}
\begin{proof}[Proof of \Cref{lem:low-tri-toep-set}.]
Assertions~(i) and~(ii) are easy to verify. Details are omitted for brevity. Assertion~(iii) is a direct consequence of assertion~(ii). Indeed, if $N\in\mathcal T^{k\times k}$ is the unique Toeplitz matrix whose blocks satisfy~\eqref{eq:toeplitz-recursion}, then $O=MN$ is the identity matrix in $\mathcal T^{k\times k}$, that is, we have $O_0=I$ and $O_t=0$ for every $t\in\mathbb N_+$. 
\end{proof}

%
%
\medskip
\begin{proof}[Proof of \Cref{lem:x-t-u-t-in-w-v}.]
To facilitate following the proof, we restate the identities to prove:
\begin{align*}
    & u_t = \sum_{s=0}^t \left[ (U D)_{t-s} w_s + (U) _{t-s} v_s \right] \quad \text{and} \quad x_t = \sum_{s=0}^t \left[ (G + H U D) _{t-s} w_s + (H U)_{t-s} v_s \right] \\
    & \Sigma_{u_t} = \EE_\PP\left[u_tu_t^\top\right] = \sum_{s=0}^t \left( (U D)_{t-s} \Sigma_{w_s} (U D)_{t-s}^\top + (U) _{t-s} \Sigma_{v_s} (U) _{t-s}^\top\right) \\ 
    & \Sigma_{x_t} = \EE_\PP \left[x_tx_t^\top \right] = \sum_{s=0}^t \left( (G + H U D) _{t-s} \Sigma_{w_s} (G + H U D) _{t-s}^\top + (H U)_{t-s} \Sigma_{v_s} (H U)_{t-s}^\top \right) \\ 
    & \frac{1}{T} \sum_{t=0}^{T-1} \EE_{\PP}\left[ x_t^\top Q_0 x_t + u_t^\top R_0 u_t \right] = \frac{1}{T} \sum_{s=0}^{T-1} \Tr\left(\Sigma_{w_s} \left(\sum_{t=0}^{T-1-s} M_{t}\right) + \Sigma_{v_s} \left(\sum_{t=0}^{T-1-s} N_{t} \right) \right). 
\end{align*}
Recall from \Cref{lemma:eta-rep-w-v} that $u=U\eta= U(Dw + v)$. As $U\in\mathcal T^{m\times p}$, $C \in \mathcal T^{p\times n}$ and $G \in \mathcal T^{n \times n}$, \Cref{lem:low-tri-toep-set}\,(ii) further implies that $UD=UCG \in \mathcal T^{m\times n}$. Hence, we find
\[
    u_t = \sum_{s=0}^t \left( (U D)_{t-s} w_s + (U)_{t-s} v_s \right). 
\]
As $x= Hu+ Gw = HU(Dw + v) + Gw$ and as $H \in \mathcal T^{n \times m}$, $U D \in \mathcal T^{m\times n}$ and $G \in \mathcal T^{n \times n}$, the expression for $x_t$ readily follows from a similar argument. 

The expressions for the covariance matrices $\Sigma_{u_t}, \Sigma_{x_t}$ follow readily by recognizing that the random vectors in the set $\mathcal{Z}$ (i.e., $x_0, \{w_s\}_{s=0}^\infty, \{v_s\}_{s=0}^\infty)$ have zero mean and are uncorrelated. 

The average expected cost of $(x,u)$ over the first $T$ periods can then be expressed as
\begin{align*}
    \frac{1}{T} \sum_{t=0}^{T-1} \EE_{\PP}\left[ x_t^\top Q_0 x_t + u_t^\top R_0 u_t \right] & = \frac{1}{T} \sum\limits_{t=0}^{T-1} \Tr\left( \Sigma_{x_t} Q_0 + \Sigma_{u_t} R_0 \right) \\
    & = \frac{1}{T} \sum_{t=0}^{T-1} \sum_{s=0}^t \Tr\left(\Sigma_{w_s} M_{t-s} + \Sigma_{v_s} N_{t-s}\right) \\
    & = \frac{1}{T} \sum_{s=0}^{T-1} \sum_{t=s}^{T-1} \Tr\left(\Sigma_{w_s} M_{t-s} + \Sigma_{v_s} N_{t-s}\right)\\
    & = \frac{1}{T} \sum_{s=0}^{T-1} \Tr\left(\Sigma_{w_s} \left(\sum_{t=0}^{T-1-s} M_{t}\right) + \Sigma_{v_s} \left(\sum_{t=0}^{T-1-s} N_{t} \right) \right).
\end{align*}
Here, the first equality holds because $u_t$, $x_t$ have zero mean; the second equality uses~\eqref{eq:cov_controls_inf_horizon}-\eqref{eq:cov_states_inf_horizon} and the expressions for $M_t$ and $N_t$; the third equality is obtained by interchanging the order of summing over~$t$ and~$s$; and the fourth inequality follows from an index shift $t\leftarrow t-s$.  
\end{proof}

%
%
\begin{proof}[Proof of \Cref{prop:inf-optimal-distributions}.]
Select any $\PP\in \overline{\mathcal{B}}{}^\infty$ and $U\in\mathcal U_\infty$. We claim that the expression of the long-run-average costs in ~\Cref{lem:x-t-u-t-in-w-v} achieves a finite limit as $T \rightarrow \infty$. To that end, recalling the definition of the matrices $M_t$ and $N_t$ from~\Cref{lem:x-t-u-t-in-w-v}, we claim that the infinite matrix sums
\[
    M_\infty= \sum_{t=0}^{\infty} M_{t} \quad \text{and}\quad N_\infty= \sum_{t=0}^{\infty} N_{t}
\]
exist and are finite. Specifically, an infinite sum of matrices such as $\sum_{t=0}^\infty M_t$ exists and is finite if and only if each corresponding sum of matrix entries, $\sum_{t=0}^\infty (M_t)_{ij}$, converges for all $i, j$. Equivalently, this condition holds if and only if the scalar series $\sum_{t=0}^\infty a^\top M_t b$ converges for all vectors $a, b \in \mathbb{R}^n$.

Note first that, for any $a\in\R^n$, the limit $\lim_{T\to\infty}\sum_{t=0}^T a^\top  M_t a$ exists because $a^\top M_t a\geq 0$ for every~$t\in\mathbb N$, so $\sum_{t=0}^T a^\top M_t a$ is non-decreasing in~$T$. The limit is also finite because
\begin{align*}
    \sum_{t=0}^T a^\top M_t a & \leq \sum_{t=0}^T \left( \lambda_{\rm max}(R_0)\, a^\top (UD)_t^\top (UD)_t a + \lambda_{\rm max}(Q_0)\, a^\top(G+HUD)_t^\top (G+HUD)_t a\right) \\
    & \leq \sum_{t=0}^T \left( \lambda_{\rm max}(R_0)\, \|(UD)_t\|_{\rm F}^2 \|a\|_2^2 + \lambda_{\rm max}(Q_0)\, \|(G+HUD)_t\|_{\rm F}^2 \|a\|_2^2\right)\\
    & \leq \lambda_{\rm max}(R_0)\, \|UD\|_{\mathcal T}^2 \|a\|_2^2 + \lambda_{\rm max}(Q_0)\, \|G+HUD\|_{\mathcal T}^2 \|a\|_2^2 < \infty \quad \forall T\in\mathbb N.
\end{align*}
Above, the first two inequalities follow from the definition of $M_t$ and the Cauchy-Schwarz inequality, respectively. The third inequality exploits our definition of the Toeplitz norm $\|\cdot\|_{\mathcal T}$ and holds because $U\in\mathcal U_\infty$. Hence, $\sum_{t=0}^{\infty} a^\top M_{t} a$ exists and is finite. In addition, we have
\begin{align*}
    \lim_{T\to\infty} \sum_{t=0}^T a^\top M_t b =  \lim_{T\to\infty} \frac{1}{4} \sum_{t=0}^T (a+b)^\top M_t(a+b)- \lim_{T\to\infty} \frac{1}{4}\sum_{t=0}^T (a-b)^\top M_t(a-b)
\end{align*}
for all $a,b\in\R^n$. Both limits on the right hand side exist and are finite thanks to the above arguments. Hence, the limit on the left hand side exists and is finite, too. Because the choice of $a,b\in\R^m$ was arbitrary, every element of the matrix $\sum_{t=0}^{T} M_{t}$ has a finite limit as $T\rightarrow \infty$, proving that $M_\infty$ exists and is finite. Similarly, one can show that $N_\infty$ exists and is finite.

Select any $\Sigma_w^\star \in\arg\max_{\Sigma_w\in\mathcal{M}_{\Sigma_w}} \Tr(\Sigma_w M_\infty)$ and $\Sigma_v^\star \in\arg\max_{\Sigma_v\in\mathcal{M}_{\Sigma_v}} \Tr(\Sigma_v N_\infty)$, which exist because $\mathcal{M}_{\Sigma_w}$ and $\mathcal{M}_{\Sigma_v}$ are compact. Next, define $\PP^\star=\PP^\star_{x_0} \otimes ( \otimes_{t=0}^\infty (\PP^\star_{w_t} \otimes \PP^\star_{v_t}))$, where $\PP^\star_{x_0}=\PP^\star_{w_t} =\mathcal N(0,\Sigma^\star_w)$ and $\PP^\star_{v_t} =\mathcal N(0,\Sigma^\star_v)$ for all $t\in\mathbb N$. By construction, $\PP^\star$ is a time-invariant Gaussian distribution. By~\eqref{eq:definitions_Mw_Mv_inf_horizon}, $\mathcal{M}_{\Sigma_{x_0}}=\mathcal{M}_{\Sigma_{w_t}}=\mathcal{M}_{\Sigma_w}$ and $\mathcal{M}_{\Sigma_{v_t}}=\mathcal{M}_{\Sigma_v}$ for all $t\in\mathbb N$, so 
one readily verifies that~$\PP^\star$ belongs to~$\mathcal{B}^\infty_{\mathcal{N}}$.
Then, because $\mathcal{B}^\infty_{\mathcal{N}}\subseteq \overline{\mathcal{B}}{}^\infty$, we have that~$\PP^\star$ is feasible in the inner maximization problem in~\eqref{eq:upper-bound-infty}. In the remainder of the proof we will show that~$\PP^\star$ is also optimal. To this end, note first that for any $\PP\in \overline{\mathcal{B}}{}^\infty$ and $T\in\mathbb N$ we have
\begin{align*}
    \frac{1}{T} \sum_{t=0}^{T-1} \EE_{\PP}\left[ x_t^\top Q_0 x_t + u_t^\top R_0 u_t \right] & \leq \frac{1}{T} \sum_{s=0}^{T-1} \Tr\left(\Sigma_{w_s} M_\infty + \Sigma_{v_s} N_\infty \right) \leq \Tr\left(\Sigma_{w}^\star M_\infty + \Sigma_{v}^\star N_\infty \right).
\end{align*}
Above, the first inequality follows from~\eqref{eq:exp_avg_run_cost_inf_horizon} and the construction of~$M_\infty$ and~$N_\infty$, while the second inequality follows from the choice of~$\Sigma_w^\star$ and~$\Sigma_v^\star$. Because the inequality above holds for all $\PP\in \overline{\mathcal{B}}{}^\infty$ and $T\in\mathbb N$, we may conclude that $\max_{\mathbb{P} \in \overline{\mathcal{B}}{}^\infty} J_\PP(x,u) \leq \Tr\left(\Sigma_{w}^\star M_\infty + \Sigma_{v}^\star N_\infty \right)$. To show that this upper bound is attained by~$\PP^\star$, choose any~$\varepsilon>0$. Then, there exists $T_0\in\mathbb N$ such that
\begin{equation}
    \label{eq:M-N-convergence}
    \left\|M_\infty - \sum_{t=0}^{T} M_{t} \right\|_2\leq \frac{\varepsilon}{2(1+\Tr(\Sigma_{w}^\star))} \quad \text{and} \quad \left\|N_\infty  - \sum_{t=0}^{T} N_{t} \right\|_2\leq \frac{\varepsilon}{2(1+\Tr(\Sigma_{v}^\star))} \quad \forall T\geq T_0.
\end{equation}
Thus, we have
\begin{align*}
    & \frac{1}{T} \sum_{t=0}^{T-1} \EE_{\PP^\star}\left[ x_t^\top Q_0 x_t + u_t^\top R_0 u_t \right]  \geq \frac{1}{T} \sum_{s=0}^{T-1-T_0} \Tr\left(\Sigma_{w}^\star \left(\sum_{t=0}^{T-1-s} M_{t}\right) + \Sigma_{v}^\star \left(\sum_{t=0}^{T-1-s} N_{t} \right) \right) \\
    & \quad \geq \frac{T-T_0}{T} \left(\Tr\left(\Sigma_{w}^\star M_\infty+ \Sigma_{v}^\star N_\infty \right) - \Tr(\Sigma_{w}^\star)\, \frac{\varepsilon}{2(1+\Tr(\Sigma_{w}^\star))} - \Tr(\Sigma_{v}^\star)\, \frac{\varepsilon}{2(1+\Tr(\Sigma_{v}^\star))}\right)\\
    & \quad \geq \frac{T-T_0}{T} \left(\Tr\left(\Sigma_{w}^\star M_\infty+ \Sigma_{v}^\star N_\infty \right) - \varepsilon \right) \quad \forall T>T_0.
\end{align*}
The first inequality in the above expression follows from~\eqref{eq:exp_avg_run_cost_inf_horizon} applied to $\PP=\PP^\star$ and from the observation that all terms in the sum over~$s$ are nonnegative. The second inequality exploits~\eqref{eq:M-N-convergence}. Taking the limit superior over~$T$ on both sides then shows that
\[
    J_{\PP^\star}(x,u) \geq \limsup_{T\to\infty} \frac{T-T_0}{T} \left(\Tr\left(\Sigma_{w}^\star M_\infty+ \Sigma_{v}^\star N_\infty \right) - \varepsilon \right) = \Tr\left(\Sigma_{w}^\star M_\infty+ \Sigma_{v}^\star N_\infty \right) - \varepsilon.
\]
As the choice of~$\varepsilon>0$ was arbitrary, we may conclude that $J_{\PP^\star}(x,u) = \Tr(\Sigma_{w}^\star M_\infty+ \Sigma_{v}^\star N_\infty )$. This shows that~$\PP^\star$ solves indeed the inner maximization problem in~\eqref{eq:upper-bound-infty}.

\end{proof}

%
%
\medskip
\begin{proof}[Proof of \Cref{prop:inf-optimal-controllers}.]
Fix any $\PP\in \underline{\mathcal{B}}^\infty_{\mathcal{N}}$. By Proposition~II.1 in  \citet{ref:hadjiyiannis2011affine}, the inner minimization problem in~\eqref{eq:inf-lower-bound-dual} over \emph{purified} output feedback controls $u \in \mathcal{U}_\eta$ is equivalent to a classical infinite-horizon LQG problem with output feedback controls $u \in \mathcal{U}_y$, of the form
\begin{equation}
    \begin{array}{cll} \min\limits_{u, x, y} & J_\PP(u,x) 
    \\
    \st &  u \in \mathcal U_{y},~{x} = H u + G w, ~y=Cx+v. \end{array}
\label{eq:min-limsup-y}
\end{equation}
Under Assumption~\ref{ass:stabilizability-detectibility}, the results in Appendix\S\ref{sec:classical-lqg-inf-horizon} are directly applicable to the latter problem. Specifically, there exist $K\in\R^{m\times n}$ and $L\in\R^{n\times p}$ such that both $ A_0+ B_0 K$ and $A_0-LC_0$ are Schur stable, and problem~\eqref{eq:min-limsup-y} is solved by $u_t= K \hat x_t$ for every $t\in\mathbb N$, where $\hat{x}_t$ is a state estimator that follows the recursion in \eqref{eq:state-estimate-dyn} (from Appendix\S\ref{sec:classical-lqg-inf-horizon}) and satisfies
\begin{align*}
    \hat x_{t+1} &=(I-LC_0) (A_0 + B_0 K)\hat{x}_t+ L y_{t+1} \quad \forall t \in \mathbb N.
\end{align*}
This recursion implies that $u=U'y$, where $U'\in\mathcal T^{m\times p}$ is a block lower triangular Toeplitz matrix with blocks 
\[
    U'_t = K(I-LC_0)^t (A_0 + B_0K)^tL \quad \forall t\in\mathbb N.
\]
Consequently, problem~\eqref{eq:min-limsup-y} is solved by a {\em stationary} linear output feedback policy.

An immediate generalization of \Cref{lemma:linear-rel-u-eta} to infinite-horizon control problems ensures that the linear output feedback policy~$u = U' y$ can equivalently be represented as a linear purified output feedback policy $u=U\eta$, where $U = (I - U' CH)^{-1} U' $ and~$I$ stands for the identity matrix in $\mathcal T^{m\times m}$. Observe now that $C\in\mathcal T^{p\times n}$ and $H\in\mathcal T^{n\times m}$. As $U'\in\mathcal T^{m\times p}$, assertions~(i) and~(ii) of \Cref{lem:low-tri-toep-set} imply that $I - U' C H \in \mathcal{T}^{m\times m}$. In addition, as~$H$ has zero blocks on the diagonal, one can readily verify that all blocks on the main diagonal of $I - U' CH$ are $m$-dimensional identity matrices. This ensures via \Cref{lem:low-tri-toep-set}\,(iii) that $U= (I - U' CH)^{-1} U'\in\mathcal T^{m\times p}$. Consequently, $u=U\eta$ constitutes a {\em stationary} linear purified output feedback policy. By construction, this policy solves the inner minimization problem in \eqref{eq:inf-lower-bound-dual}, which is equivalent to \eqref{eq:min-limsup-y}

We verify that $U \in \mathcal{U}_\infty$. Recall from \Cref{lem:x-t-u-t-in-w-v}  that with $u=U\eta$ for $U\in\mathcal T^{m\times p}$, the covariance of~$u_t$ is:
\begin{align*}
    \Sigma_{u_t} &= \EE_\PP\left[u_t u_t^\top\right] = \sum_{s=0}^t \left( (U D)_{t-s} \Sigma_{w_s} (U D)_{t-s}^\top + (U) _{t-s} \Sigma_{v_s} (U) _{t-s}^\top\right)
\end{align*}
where $(UD)_s$ is shorthand for the block matrix on the $s$-the subdiagonal of the Toeplitz matrix~$UD$. Moreover, by \Cref{lem:stationary-covs} in Appendix\S\ref{sec:classical-lqg-inf-horizon}, the state estimator $\hat{x}_t$ must admit a stationary covariance matrix in the limit $t\rightarrow \infty$, which implies that the feedback control $u_t = K \hat{x}_t$ also admits a stationary covariance matrix, and thus the limit
\[
    \lim_{t\to\infty}\Sigma_{u_t} = \sum_{t=0}^{\infty} (UD)_t\Sigma_w (UD)_t^\top + \sum_{t=0}^{\infty} U_t\Sigma_v U_t^\top
\]
exists and is finite. Because $\Sigma_w, \Sigma_v\succ 0$ by our standing assumption that $\PP\in \underline{\mathcal{B}}^\infty_{\mathcal{N}}$, this implies that
\[
    \sum_{t=0}^{\infty}\Tr\left( (UD)_t (UD)_t^\top \right) = \left\| UD \right\|^2_{\mathcal T} <\infty \quad \text{and} \quad \sum_{t=0}^{\infty} \Tr\left(U_t U_t^\top\right) = \left\| U \right\|^2_{\mathcal T}<\infty.
\]
Next, recall that $x=H u + G w=(HUD+G) w +HUv$. From \Cref{lem:stationary-covs}, we know that~$x_t$ admits a stationary covariance matrix, and thus the limit
\[
    \lim_{t\to\infty}\Sigma_{x_t} = \sum_{t=0}^{\infty} (HUD+G)_t\Sigma_w (HUD+G)_t^\top + \sum_{t=0}^{\infty} (HU)_t\Sigma_v (HU)_t^\top
\]
exists and is finite. Because $\Sigma_w, \Sigma_v\succ 0$ (because $\PP\in \underline{\mathcal{B}}^\infty_{\mathcal{N}}$), this implies that $\| HUD+G \|^2_{\mathcal T} <\infty$ and $\| HU \|^2_{\mathcal T} <\infty$. In summary, we have shown that $U\in{\mathcal U}_\infty$, and thus the claim follows.

\end{proof}

%
\medskip
\begin{proof}[Proof of \Cref{lem:limsup-converges}.]
Recall that $R_0 \succ 0$ and $Q_0 \succ 0$ by Assumption~\ref{ass:stabilizability-detectibility}-(iii). This ensures that $\Tr(R_0)$ and $\Tr(Q_0)$ are both strictly positive.
By \Cref{lem:stationary-covs}, which applies because $\PP\in \underline{\mathcal{B}}^\infty_{\mathcal{N}}$, the limits $\lim_{t\to\infty}\Sigma_{x_t} = \Sigma_{x_\infty}$ and $\lim_{t\to\infty}\Sigma_{u_t} = \Sigma_{u_\infty}$ exist and are finite. We will show that
\[
    J_\PP(x,u) = \Tr(\Sigma_{x_\infty}Q_0) + \Tr(\Sigma_{u_\infty}R_0).
\]
As $\Sigma_{u_t}$ converges to $\Sigma_{u_\infty}$ and $\Sigma_{x_t}$ converges to $\Sigma_{x_\infty}$, for any $\varepsilon>0$ there exists $t_0 \in \mathbb N$ such that 
\begin{equation}
    \left\|\Sigma_{x_T} - \Sigma_{x_{\infty}} \right\|_2 \leq \frac{\varepsilon}{3 \Tr(Q_0)}\quad \text{and} \quad \left\|\Sigma_{u_T} - \Sigma_{u_{\infty}} \right\|_2 \leq \frac{\varepsilon}{3 \Tr(R_0)} \quad \forall t \geq t_0.
    \label{eq:covariance-norm-bounds}
\end{equation}
In addition, one easily verifies that there exists $T_0 > t_0$ such that
\begin{equation}
    \left| \frac{1}{T} \sum_{t=1}^{t_0} \bigl[ \Tr \bigl((\Sigma_{x_t}  -\Sigma_{x_\infty}) Q_0 \bigr) + \Tr\bigl((\Sigma_{u_t}  -\Sigma_{u_\infty}) R_0 \bigr) \bigr] \right| \leq \frac{\varepsilon}{3} \quad \forall T\geq T_0.
    \label{eq:partial-sum-T-0}
\end{equation}
Then, we have
\begin{align*}
    &\left| \frac{1}{T} \sum_{t=1}^T \mathbb{E}[ x_t^\top Q_0 x_t +  u_t^\top R_0 u_t] - \Tr(\Sigma_{x_\infty} Q_0 ) - \Tr(\Sigma_{u_\infty} R_0 )\right| \\
    =& \left| \frac{1}{T} \sum_{t=1}^T \Bigl[ \Tr \bigl( (\Sigma_{x_t}  -\Sigma_{x_\infty}) Q_0 \bigr) + \Tr\bigl( (\Sigma_{u_t}  -\Sigma_{u_\infty}) R_0 \bigr) \Bigr] \right| \\
    \leq &\left| \frac{1}{T} \sum_{t=1}^{t_0} \Bigl[ \Tr\bigl( (\Sigma_{x_t}  -\Sigma_{x_\infty}) Q_0 \bigr) + \Tr\bigl( (\Sigma_{u_t}  -\Sigma_{u_\infty}) R_0 \bigr) \Bigr] \right| \\
    & ~~ + \left| \frac{1}{T} \sum_{t=t_0+1}^{T}  \Tr((\Sigma_{x_t}  -\Sigma_{x_\infty}) Q_0) \right| + \left| \frac{1}{T} \sum_{t=t_0+1}^{T}  \Tr((\Sigma_{u_t}  -\Sigma_{u_\infty}) R_0) \right|\\
    &\leq \frac{\varepsilon}{3} + \frac{\varepsilon (T-t_0)}{3 T} + \frac{\varepsilon (T-t_0)}{3 T} \leq \varepsilon
\end{align*}
for all $T\geq T_0$. Here, the first inequality exploits the triangle inequality, while the second inequality follows from~\eqref{eq:covariance-norm-bounds} and~\eqref{eq:partial-sum-T-0} and from the matrix H\"older inequality
$\Tr(\Sigma Q) \leq \| \Sigma\|_2 \Tr(Q)$, which holds for any symmetric matrices~$\Sigma$ and~$Q$ with~$Q\succeq 0$. As~$\varepsilon$ was chosen arbitrarily, the above estimate proves that the limit in~\eqref{eq:J-limit-formula} exists and is indeed equal to $\Tr(\Sigma_{x_\infty}Q_0) + \Tr(\Sigma_{u_\infty}R_0)$. 

\end{proof}

%
%
\begin{proof}[Proof of \Cref{prop:J-saddle}.]
Consider any $U\in\mathcal U_\infty$ and any $\PP\in \underline{\mathcal{B}}^\infty_{\mathcal{N}}$ with covariance matrices $\Sigma_w\in\mathcal{M}^+_{\Sigma_w}$ and $\Sigma_v\in\mathcal{M}^+_{\Sigma_v}$. For any~$T \in \mathbb N_+$, define $J_T(U; \Sigma_w, \Sigma_v)$ as shorthand for the average expected cost under the control policy $u=U(Dw+v)$ and associated states $x=Hu+Gw$ over the first~$T$ periods, with an expression given by~\eqref{eq:exp_avg_run_cost_inf_horizon}. Hence, $J_T(U; \Sigma_w, \Sigma_v)$ is convex quadratic in~$U$ and linear in~$(\Sigma_w, \Sigma_v)$ for every~$T\in\mathbb N_+$. By the definition of the long-run average expected cost~$J(U; \Sigma_w, \Sigma_v)$, we have
\[
    J(U; \Sigma_w, \Sigma_v) = \limsup_{T\to\infty} J_T(U; \Sigma_w, \Sigma_v) = \lim_{T'\to\infty} \sup_{T\geq T'} J_T(U; \Sigma_w, \Sigma_v).
\]
Because convexity is preserved under pointwise suprema and limits, this shows that $J(U; \Sigma_w, \Sigma_v)$ is convex in~$U\in\mathcal U_\infty$ for any $(\Sigma_w,\Sigma_v) \in\mathcal{M}^+_{\Sigma_w} \times \mathcal{M}^+_{\Sigma_v}$. Moreover, \Cref{lem:limsup-converges} implies that
\begin{equation}
    \label{eq:J-identity}
    J(U; \Sigma_w, \Sigma_v) = \lim_{T\to\infty} J_T(U; \Sigma_w, \Sigma_v).
\end{equation}
Hence, the claim follows.
\end{proof}

%
%
\medskip
\section{Extension to Entropy-Regularized Optimal Transport}
\label{sec:entropy_reg_opt_transport}
In this section, we discuss a different example of ambiguity set that is compatible with our framework. Consider taking the divergence~$\mathds D$ as the entropy-regularized optimal transport discrepancy with regularization parameter~$\epsilon\geq  0$, denoted with $\mathds W_\epsilon$ and defined as follows.

\begin{definition}[Entropy-regularized Optimal Transport Discrepancy]
\label{def:ent_regularized_OT_divergence}
The entropy-regularized optimal transport discrepancy between two distributions $\PP_z$ and $\hat \PP_{z}$ on $\R^{d_z}$ with regularization parameter 
$\epsilon \geq  0$ is
\begin{align*}
    \mathds W_\epsilon(\PP_z, \hat \PP_{z}) = \min\limits_{\pi \in \Pi(\PP_z, \hat \PP_{z}) } \int_{\R^{d_z} \times \R^{d_z}} \| z - \hat z \|^2 \, \diff \pi(z, \hat z) + \epsilon \mathds H(\pi),
\end{align*}
where $\Pi(\PP_z, \hat \PP_{z})$ is the set of couplings of $\PP_z$ and $\hat \PP_{z}$ as in Definition~\ref{def:wasserstein-distance}, and the entropy of a coupling~$\pi$ is defined as:
\[
    \mathds H(\pi)=\int_{\mathbb{R}^{d_z} \times \mathbb{R}^{d_z}} \frac{\diff \pi}{\diff \mathbb{L}}(z, \hat z) \log\Bigl( \frac{\diff \pi}{\diff \mathbb{L}}(z, \hat z) \Bigr) \, \diff \mathbb{L}(z,\hat z),
\]
where $\frac{\diff \pi}{\diff \mathbb{L}}$ denotes the Radon-Nikodym derivative of $\pi$ with respect to  the Lebesgue measure $\mathbb{L}$ on $\R^{d_z} \times \R^{d_z}$, defined for any $\pi$ that is absolutely continuous with respect to $\mathbb{L}$. If $\pi$ is not absolutely continuous with respect to $\mathbb{L}$, then we set $\mathds H(\pi)=+\infty$.
\end{definition}

Ambiguity sets based on discrepancies similar to $\mathbb{W}_\epsilon$ have appeared before in the DRO literature \citep[see, e.g.,][]{ref:wang2021sinkhorn,azizian2023regularization}. Although these sets are compatible with our framework, a few notable differences arise. First, the divergence $\mathds{W}_\epsilon$ no longer satisfies the identity of indiscernibles principle, i.e., $\mathds{W}(\PP,\PP)$ may not be identically zero for all $\PP \in \setalldist{d}$. Moreover, the ambiguity set could be empty when the ambiguity radius $\rho_z$ is small (this arises due to the regularization term $\epsilon \mathds{H}(\pi)$, for a sufficiently large $\epsilon$). As a result, small changes are required in our assumptions to ensure the problem is non-trivial. We point these out subsequently, as we verify each assumption.

\subsection{Verifying Assumption~\ref{ass:general_assumption_about_M_function}}
\label{appendix:assumption3_entropy_reg_KL}
Mirroring our earlier developments, for two non-degenerate distributions with finite second moments, we define a distance function between two mean and covariance matrix pairs as follows.
\begin{definition}
The entropy-regularized Bures-Wasserstein distance between $(\mu_z, \Sigma_z) \in \R^{d_z} \times \mathbb S_{++}^{d_z}$ and $(\hat\mu_{z}, \hat \Sigma_{z})  \in\R^{d_z} \times \mathbb S_{++}^{d_z}$ with regularization parameter~$\epsilon\in \R_+$ is given by  
\begin{align*} 
&\mathds G_\epsilon((\mu_z, \Sigma_z), (\hat \mu_{z}, \hat \Sigma_{z})) \\
&\qquad =\sqrt{\| \mu_z  - \hat \mu_{z}\|^2  + \Tr(\Sigma_z) + \Tr(\Sigma_2)  -  2 \Tr( X_{\epsilon}(\Sigma_z)) - \frac{\epsilon}{2} \log\left((2\pi e)^{2d} \left(\frac{\epsilon}{2}\right)^{d}| X_\epsilon(\Sigma_z)|\right)},
\label{eq:B_eps}
\end{align*}
where
\begin{equation}X_{\epsilon}(\Sigma_z)=\left(\hat \Sigma_z^{1 / 2} \Sigma_{z} \hat \Sigma_z^{1 / 2}+\left(\frac{\epsilon}{4}\right)^2 I\right)^{1 / 2} -\frac{\epsilon}{4} I.
\label{eq:X_eps}\end{equation}
\end{definition}
The following proposition shows that the entropy-regularized OT between any non-degenerate distribution and a non-degenerate Gaussian distribution is bounded below by the entropy-regularized OT between their respective mean vectors and covariance matrices. Moreover, if both distributions are Gaussian, then this bound is tight.
\begin{proposition}[\citet{ref:del2020statistical}]
\label{prop:entropic-ot-gauss-ineqs}
For any distribution $\PP_z$ with mean and covariance matrix $
(\mu_z, \Sigma_z) \in\R^{d_z} \times \mathbb S_{++}^{d_z}$ and any Gaussian distribution $\hat \PP_z$ on $\R^{d_z}$ with mean and covariance matrix $
(\hat \mu_z, \hat \Sigma_z) \in\R^{d_z} \times \mathbb S_{++}^{d_z}$, we have
\begin{enumerate}
    \item[i)] $\mathds W_\epsilon (\PP_z, \hat \PP_z) \geq  \mathds{G}_{\epsilon}((\mu_z, \Sigma_{z}),(\hat \mu_z, \hat \Sigma_z))$
    \item[ii)] $\mathds W_\epsilon (\PP_z, \hat \PP_z) =  \mathds{G}_{\epsilon}((\mu_z, \Sigma_{z}),(\hat \mu_z, \hat \Sigma_z))$ if $\PP_z$ is Gaussian.
\end{enumerate}
\end{proposition}
The results in~(i) and~(ii) of \Cref{prop:entropic-ot-gauss-ineqs} follow from Theorem~2.3 and Theorem 2.2 in~ \citet{ref:del2020statistical}, respectively. Notably, for the nominal Gaussian distribution $\hat{\PP}_z = \mathcal{N}(\hat \mu_z, \hat{M}_z)$, the minimum value of the square of~$\mathds{G}_{\epsilon}(\PP_z, \mathcal N(\hat \mu_z, \hat \Sigma_z))$ over $\mathbb P_z\in \setalldist{d_z}$ is attained by $\mathcal N(\hat\mu_z, \hat \Sigma_z + \epsilon/2 I)$~\cite[Theorem~2.4]{ref:del2020statistical}. Importantly, this value is nonzero (and is not attained by the nominal distribution $\hat{\PP}_z$), so to ensure that the ambiguity set is nonempty, we require 
\begin{align}
    \rho_z \geq \underline{\rho}_z := \mathds G_\epsilon \bigl( (\revision{\hat\mu_z}, \hat \Sigma_z + \epsilon/2 I),(\revision{\hat \mu_z}, \hat \Sigma_z) \bigr), ~\mbox{for all}~ z \in \mathcal{Z}.
    \label{eq:radius_requirement_EROT}
\end{align}
Under this premise, we have the following result.
\begin{proposition}
    \label{prop:entropic-ot-cov-set-convex-compact}
    If $\rho_z \geq \underline{\rho}_z$, the divergence $\mathds{W}_\epsilon$ satisfies
    Assumption~\ref{ass:general_assumption_about_M_function}.
\end{proposition}
\begin{proof}[Proof of \Cref{prop:entropic-ot-cov-set-convex-compact}.]
    If $\rho_z \geq \underline{\rho}_z$, \Cref{prop:entropic-ot-gauss-ineqs} implies that Assumption~\ref{ass:general_assumption_about_M_function}-(i) holds. To see that~Assumption~\ref{ass:general_assumption_about_M_function}-(ii) also holds, note that:
    \begin{align*}
    \mathcal{M}_{(\mu_z, M_z)}^{\mathds G_\epsilon} 
    &= \{(\mu_z , M_z) \in \setmoments{d_z} : M_z \succ \mu_z\mu_z^\top,~  \mathds W_\epsilon (\mathcal N(\mu_z, M_z), \hat \PP_z) \leq \rho_z\}  \\
    &= \{(\mu_z, M_z) \in \setmoments{d_z} : M_z \succ \mu_z\mu_z^\top,~ \bigl( \mathds G_\epsilon((\mu_z, M_z - \mu_z\mu_z^\top),(\hat\mu_{z}, \hat \Sigma_z ) \bigr)^2 \leq \rho_z^2\},
    \end{align*}
    where the second equality follows from~\Cref{prop:entropic-ot-gauss-ineqs}-(ii). Note that $\mathcal{M}_{(\mu_z, M_z)}^{\mathds G_\epsilon}$ contains only non-degenerate distributions because otherwise, $\mathds{G}_\epsilon (\mathcal{N}(\mu_z, \Sigma_z), \hat{\mathbb{P}}_z)$ would not be finite. 
    
    We next establish that the set $\mathcal{M}_{(\mu_z, M_z)}^{\mathds G_\epsilon}$ is convex and compact. As a first step, we prove that the map $\Sigma_z \mapsto X_\epsilon(\Sigma_z)$ is matrix concave and operator monotone on the cone of positive definite matrices, $\mathbb S^{d_z}_{++}$. By Theorem~1.5.9 and Theorem~4.2.3 in~\citet{ref:bhatia2009positive}, the map $X \mapsto X^{\frac{1}{2}}$ is operator monotone and matrix concave on $\mathbb S_+^{d_z}$. The map $\Sigma_z \mapsto \hat{\Sigma}_z^{1 / 2} \Sigma_z \hat{\Sigma}_z^{1 / 2}+\left(\frac{\epsilon}{4}\right)^2 I$ is linear and operator monotone for all $\hat{\Sigma}_z \in \mathbb S^{d_z}_{++}$. Because matrix convexity/concavity is preserved under linear maps\footnote{For an elementary proof, assume $f : \mathbb S_{+}^{d_z} \rightarrow \mathbb {S}_{+}^{d_z}$ is matrix convex 
    and consider an affine map $X(Y)$. Then, $f\left( X(\lambda Y_1 + (1-\lambda)Y_2) \right) = f\left( \lambda X(Y_1) + (1-\lambda) X(Y_2) \right) \preceq \lambda f(X(Y_1)) + (1-\lambda) f(X(Y_2))$, where the first equality follows because $X(Y)$ is affine and the inequality follows because $f$ is matrix convex.} and the composition of operator monotone functions is operator monotone, it follows that $X_\epsilon(\Sigma_z)$ is matrix concave and operator monotone.

    Next, we prove that the function $\mathds{G}_\epsilon((\mu_z, M_z - \mu_z \mu_z^\top), (\hat \mu_z, \hat \Sigma_z))^2$ is jointly convex in $(\mu_z, M_z) \in \setmoments{d_z}$ for any fixed $(\hat \mu_z, \hat \Sigma_z) \in \setmoments{d_z}$. As a function of $(\mu_z, M_z)$ (ignoring constant terms), 
    this can be rewritten as $\|\mu_z - \hat \mu_z\|^2 + \Tr(M_z) - \|\mu_z\|^2 - 2 \Tr(X_\epsilon(M_z - \mu_z \mu_z^\top)) -\epsilon/2 \log(|X_\epsilon(M_z - \mu_z\mu_z^\top)|)$. Expanding the first term and canceling the third term, it can be seen that the first three terms yield a jointly convex function of $(\mu_z, M_z)$. So it suffices to prove that $f( X_\epsilon(M_z - \mu_z \mu_z^\top))$ is jointly concave in $(\mu_z,M_z)$ on $\setmoments{d_z}$, where $f(X) = 2 \Tr(X) + \epsilon/2 \log(|X|)$.

    To analyze this, we first prove that $f(X_\epsilon(\Sigma_z))$ is operator monotone and concave in $\Sigma_z$ on $\mathbb {S}_{++}^{d_z}$. Note that the function $f(X)$ is operator monotone on $\mathbb {S}_+^{d_z}$ because $\Tr(\cdot)$ and $\log \det (\cdot)$ are both operator monotone. Because $X_\epsilon(\Sigma_z)$ is also operator monotone (as argued above) and the composition of operator monotone functions remains operator monotone, we conclude that $f( X_\epsilon(\Sigma_z) )$ is operator monotone on $\mathbb {S}_{++}^{d_z}$. To prove concavity, consider any $\Sigma_z,\Sigma_z' \in \mathbb {S}_+^{d_z}$ and any $\lambda \in [0,1]$. We have:
    \begin{align*}
        f\left( X_\epsilon(\lambda \Sigma_z + (1-\lambda) \Sigma_z') \right) &\geq  f\left( \lambda X_\epsilon(\Sigma_z) + (1-\lambda) X_\epsilon(\Sigma_z') \right)
        \geq  \lambda f\left( X_\epsilon(\Sigma_z) \right) + (1-\lambda) f \left( X_\epsilon(\Sigma_z') \right),
    \end{align*}
    where the first inequality follows because  $f(X)$ is operator monotone on $\mathbb {S}_+^{d_z}$ and $X_\epsilon(\Sigma_z)$ is matrix concave on $\mathbb {S}_+^{d_z}$, and the second inequality follows because $f(X)$ is concave on $\mathbb {S}_+^{d_z}$. This shows that $f( X_\epsilon(\Sigma_z))$ is concave in $\Sigma_z$. 
    
    To prove that $f( X_\epsilon(M_z - \mu_z \mu_z^\top))$ is jointly concave in $(\mu_z,M_z)$ on $\setmoments{d_z}$, consider any $(\mu_z,M_z)$, $(\mu_z',M_z') \in \setmoments{d_z}$ and $\lambda \in [0,1]$. With $\mu_\lambda = \lambda \mu_z + (1-\lambda) \mu_z'$, we can verify the identity:
    \begin{align}
        \mu_\lambda \mu_\lambda^\top = \lambda \mu_z \mu_z^\top + (1-\lambda) \mu_z' (\mu_z')^\top - \lambda (1-\lambda) (\mu_z - \mu_z')(\mu_z- \mu_z')^\top.
        \label{eq:inequality_rank_one_trick}
    \end{align}
    This implies that:
    \begin{align*}
       f\left( X_\epsilon( \lambda M_z + (1-\lambda) M_z' - \mu_\lambda \mu_\lambda^\top ) \right) & \geq  
       f\left( X_\epsilon( \lambda M_z + (1-\lambda) M_z' - \lambda \mu_z \mu_z^\top - (1-\lambda) \mu_z' (\mu_z')^\top ) \right) \\
       & \geq  \lambda f\left( X_\epsilon( M_z - \mu_z \mu_z^\top ) \right) + (1-\lambda) f\left( X_\epsilon( M_z' - \mu_z'( \mu_z')^\top) \right),
    \end{align*}
    where the first inequality follows from~\eqref{eq:inequality_rank_one_trick} and because $f( X_\epsilon(\Sigma_z) )$ is operator monotone in $\Sigma_z$. The second inequality follows because $f( X_\epsilon(\Sigma_z) )$ is concave in $\Sigma_z$.
    
    These arguments imply that $\mathds{G}_\epsilon((\mu_z, M_z - \mu_z \mu_z^\top), (\hat \mu_z, \hat \Sigma_z))^2$ is jointly convex in $(\mu_z,\revision{M_z})$ \revision{on} $\setmoments{d_z}$. 
    
    To conclude the proof, we argue that $\mathcal{M}_{\Sigma_z}^{\mathds G_\epsilon}$ is convex and compact. $\mathcal{M}_{\Sigma_z}^{\mathds G_\epsilon}$ is convex because it  corresponds to the $\rho^2_z > 0$ sublevel set of the convex function $\mathds{G}^2_\epsilon((\mu_z, M_z - \mu_z \mu_z^\top), (\hat \mu_z, \hat \Sigma_z))$. Moreover, $\mathds{G}^2_\epsilon((\mu_z, M_z - \mu_z \mu_z^\top), (\hat\mu_z, \hat \Sigma_z))$ is a continuous and coercive function
    in $(\mu_z, M_z)$, and thus its sublevel sets are closed and bounded, and thus are compact (see, e.g., Lemma~1.24 and Proposition~11.12 in \citealp{ref:BauschkeCombettes2017}, respectively). 
\end{proof}

\smallskip
\subsection{Verifying Assumption~\ref{ass:setting_mean_zero_feasible}}
\begin{proposition}
    Fix $\hat{\mu}_z=0$. If $\rho_z \geq \underline{\rho}_z$, the divergence $\mathds{W}_\epsilon$ satisfies Assumption~\ref{ass:setting_mean_zero_feasible}.
    \label{prop:assumption3_for_EROT}
\end{proposition}
\begin{proof}{Proof.}
We simplify notation by omitting the subscript $z$. Recall that $\hat{\mu}_{}=0$, which implies that $\hat{M}_{} = \hat{\Sigma}_{}$, and the nominal distribution is $\hat{\PP}_{} = \mathcal{N}(0,\hat{\Sigma}_{})$. As in the proof of \Cref{prop:assumption3_for_W_KL_M}, consider any two Gaussian distributions $\PP_{}, \PP'_{} \in \mathcal{B}_{}$ such that $\EE_{\PP}[zz^\top] = \EE_{\PP'_{}}[zz^\top] = M_{}$ and $\mu'_{} = \hat{\mu}_{} = 0$, which implies that $\Sigma'_{} =\EE_{\PP'_{}}[(z-\mu'_{})(z-\mu'_{})^\top]= M_{}$. 
For the subsequent arguments, it helps to note that
\begin{align}
    \| \mu_{} - \EE_{\hat{\PP}_{}}[z] \|^2 + \Tr ( \Sigma_{} ) - \| \mu'_{} -  \EE_{\hat{\PP}_{}}[z] \|^2 - \Tr ( \Sigma'_{} ) &= \Tr (M_{}) - \Tr (M_{}) = 0.     \label{eq:equality_for_normals_implied_by_assump_4_EROT}
\end{align}
Also, we recall from the proof of \Cref{prop:entropic-ot-cov-set-convex-compact} that the maps $\Sigma_z \mapsto X_\epsilon(\Sigma_z)$ and $X \mapsto X^{\frac{1}{2}}$ are matrix-concave and operator monotone on the cone of positive semidefinite matrices $\mathbb S^{d_z}_{++}$, that is, are functions $f : \mathbb{S}_{+}^{d_z} \rightarrow \mathbb{S}_{+}^{d_z}$ that satisfy $f(X) \succeq f(Y)$  for any $X,Y \in \mathbb S_{+}^{d_z}$ with $X \succeq Y$.

Recall from \Cref{prop:entropic-ot-gauss-ineqs} that the set of moments $\mathcal{M}_{(\mu_z, M_z)}$ can be written in this case as:
\begin{align}
    \mathcal{M}_{(\mu_z, M_z)} = \{(\mu_z, M_z) \in \setmoments{d_z}: ~ \left( \mathds{G}_\epsilon\bigl((\mu_z, M_z), (\hat{\mu}_{}, \hat{M}_{}) \bigr) \right)^2 \leq \rho_z^2\}.
    \label{eq:first_condition_assumpt_3}
\end{align}
To prove that $(\mu,M) \in  \mathcal{M}_{(\mu_{}, M_{})}$ implies that $(0,M) \!\in  \!\!\mathcal{M}_{(\mu_{}, M_{})}$, it suffices to show that 
\[\mathds{G}_\epsilon^2((0, M_{}), (0, \hat{M}_{}) )\leq \mathds{G}_\epsilon^2((\mu_z, M_{}),(0, \!\hat{M}_{}) ).\]
By \Cref{prop:entropic-ot-gauss-ineqs} and~\eqref{eq:equality_for_normals_implied_by_assump_4_EROT}, we have:
\begin{align*}
    & \mathds{G}_\epsilon^2\bigl((0, M_{}), (0, \hat{M}_{}) \bigr) - \mathds{G}_\epsilon^2\bigl((\mu_{}, M_{}), (0, \hat{M}_{}) \bigr) \\
    & \qquad \qquad = - 2 ( \Tr\left(X_\epsilon(M_{}) \right) - \Tr ( X_\epsilon(M_{} -\mu_{} \mu_{}^\top) )) - \frac{\epsilon}{2} \left(\log( |X_\epsilon(M_{})|)  - \log \left(|X_\epsilon(M_{} - \mu_{} \mu_{}^\top)|\right) \right)\leq 0,
\end{align*}
where the inequality follows because $M_{} \succeq M_{} - \mu_{} \mu_{}^\top$, the maps $X_\epsilon$ and $X^{1/2}$ are operator monotone on $\mathbb{S}^{d_z}_{++}$, and $\Tr(X) \geq \Tr(Y)$ and $\log \det (X) \geq \log \det (Y)$ if $X \succeq Y$.  
\end{proof}

\subsection{\Cref{thm:optimal-covs-are-higher} for the $\mathds{W}_\epsilon$ Divergence}
For the $\mathds{W}_\epsilon$ divergence, Assumption~\ref{ass:M-gradients} as stated in the main text is not satisfied. Instead, we require the following slightly modified assumption.
\begin{assumption}
\label{ass:M-gradients-for-W_epsilon}
Fix $\mu_z = \hat{\mu}_z=0$ for all $z\in \mathcal{Z}$. There exists $\gz: \mathbb S_+^{d_z} \to \R$ such that the sets defined in ~Assumption~\ref{ass:general_assumption_about_M_function} can be represented as \(\mathcal{M}_{\Sigma_z} = \{ \Sigma_z \in \mathbb{S}_+^{d_z} \,: \, \gz(\Sigma_z) \leq  0\}\) for any $z \in \mathcal{Z}$, where $\gz$ is convex on $\mathbb{S}_+^{d_z}$, differentiable on $\mathbb{S}_{++}^{d_z}$, and satisfies: (i) $\gz(\bar{\Sigma}_z) < 0$, (ii)~$\bar{\Sigma}_z \in \argmin_{\Sigma_z \succeq 0} \gz(\Sigma_z)$, and (iii) $\gradgz(\Sigma_1) \succeq \gradgz(\Sigma_2)$ implies $\Sigma_1 \succeq \Sigma_2$, for any $\Sigma_1, \Sigma_2 \in \mathbb S_+^{d_z}$, where $\bar{\Sigma}_z = \hat \Sigma_z + \epsilon/2 I$.
\end{assumption}
The key difference compared to Assumption~\ref{ass:M-gradients} is that conditions~(i) and~(ii) are required for the matrix $\bar{\Sigma}_z$ rather than $\hat{\Sigma}_z$. We claim that \Cref{thm:optimal-covs-are-higher} holds for the $\mathds{W}_\epsilon$-based ambiguity sets if Assumptions~\ref{ass:Gaussian}-\ref{ass:setting_mean_zero_feasible} and Assumption~\ref{ass:M-gradients-for-W_epsilon} hold: one can follow the same steps in the proof of \Cref{thm:optimal-covs-are-higher} to argue that $\Sigma_z\opt \succeq \bar{\Sigma}_z = \hat \Sigma_z + \epsilon/2 I$, which implies that $\Sigma_z\opt \succeq \hat{\Sigma}_z$, as desired.

What remains is to show that Assumption~\ref{ass:M-gradients-for-W_epsilon} is satisfied. This is always the case if the ambiguity set is not a singleton, as summarized in the next result.
\begin{proposition}
    Fix $\hat{\mu}_z=0$. If $\rho_z > \underline{\rho}_z$, the divergence $\mathds{W}_\epsilon$ divergence satisfies Assumption~\ref{ass:M-gradients-for-W_epsilon}.
    \label{prop:assumption4_for_EROT}
\end{proposition}
\begin{proof}{Proof.}
We show that $\gz (\Sigma_z) = ( \mathds{G}_\epsilon( (0, \Sigma_z), (0, \hat{\Sigma}_z) ) )^2 - \rho_z^2$ satisfies Assumption~\ref{ass:M-gradients-for-W_epsilon} for $\rho_z > \underline{\rho}_z$. 

\Cref{prop:entropic-ot-cov-set-convex-compact} (and its proof) show that $\gz$ is a convex function. That $\gz$ is differentiable in $\Sigma_z$ on $\mathbb{S}_{++}^{d_z}$ and achieves its minimum value at $\bar{\Sigma}_z = \hat \Sigma_z + \epsilon/2 I$ follows from the same properties that hold for $\mathds{G}_\epsilon$ (see Lemma~1.24 and Proposition~11.12 in \citealp{ref:BauschkeCombettes2017}.) As $\underline{\rho}_z$ was chosen so that $g(\bar{\Sigma}_z) = 0$ for $\rho_z = \underline{\rho}_z$, the continuity of $\gz$ implies that $\gz(\bar{\Sigma}_z) < 0$ for $\rho_z > \underline{\rho}_z$. 

These arguments show that requirements~(i)-(ii) of~Assumption~\ref{ass:M-gradients-for-W_epsilon} are satisfied. For~(iii), write the gradient of $\gz$ as:
\[ \gradgz(\Sigma_z) = I - \hat\Sigma_z^\frac{1}{2}\left( 
\frac{\epsilon}{4} I + \Bigl( \frac{\epsilon^2}{16} I + \hat\Sigma_z^\frac{1}{2} \Sigma_z \hat\Sigma_z^\frac{1}{2}  \Bigr)^\frac{1}{2} \right)^{-1}\hat\Sigma_z^\frac{1}{2}.\]
Then, consider $\Sigma_1, \Sigma_2 \in \mathbb{S}_+^{d_z}$ and note that $\gradgz(\Sigma_1) \succeq \gradgz(\Sigma_2)$  implies that $\Sigma_1 \succeq \Sigma_2$ because the mapping $X \mapsto X^{1/2}$ preserves ordering and the mapping $X \mapsto X^{-1}$ reverses ordering on $\mathbb{S}_{++}^{d_z}$. 
\end{proof}

%
%
\bigskip
\section{Extension to Fisher Divergence}
\label{sec:Fisher_divergence_extension}
Our final example involves another information-theoretic divergence inspired by the Fisher information matrix. To formalize it, let $\tilde{\mathcal{P}}(\R^d)$ denote the set of probability distributions $\PP$ on $\R^d$ that admit densities $p(z)$ that are continuously differentiable, everywhere positive and satisfy the condition:
\begin{align}
  \exists \epsilon > 0, ~ \bar{p} > 0 ~~:~~ \|z\|^{d_z+\epsilon} p(z) \leq \bar{p}, ~ \forall \, z \in \R^{d_z}.
  \label{eq:Fisher_fast_poly_decay}
\end{align}
All distributions with sub-exponential/sub-Gaussian tails or tails with a sufficiently fast polynomial decay satisfy the requirement; this includes many common distributions such as Gaussian, Laplace, or (with minor parameter restrictions) Student-$t$ or Pareto. The assumption rules out heavy-tailed distributions whose density decays slower than $\|z\|^{-d_z}$.

For the rest of the section, we will be interested in a divergence $\mathds{D}$ corresponding to the \emph{relative Fisher divergence} (or \emph{score-matching distance}), which we formalize next.
\begin{definition}[Fisher divergence]
  \label{def:Fisher_divergence}
 Consider $\PP_z, \hat{\PP}_z \in \setalldist{d_z}$. If $\PP_z, \hat{\PP}_z \in \tilde{\mathcal{P}}(\R^{d_z})$, we define the \emph{Fisher divergence} (also called the \emph{relative Fisher information} or \emph{score‑matching distance}) between $\PP_z$ and $\hat{\PP}_z$ as
 \[
 \mathds F(\PP_z, \hat \PP_z) =  \frac{1}{2} \int_{\R^{d_z}} \norm{\nabla  \log p(z )-\nabla\log \hat{p}(z)}_2^{2} \, p(z) \, \diff z ,
 \]
 where $p$ and $\hat{p}$ denote the densities of $\PP_z$ and $\hat{\PP}_z$, respectively, and $\nabla$ denotes the gradient with respect to $z$. If $\PP_z \notin \tilde{\mathcal{P}}(\R^{d_z})$ or $\hat{\PP}_z \notin \tilde{\mathcal{P}}(\R^{d_z})$, 
 we set $\mathds{F}(\PP_z,\hat{\PP}_z)=\infty$ if $\PP_z \neq \hat{\PP}_z$ and $\mathds{F}(\PP_z,\hat{\PP}_z) = 0$ if $\PP_z = \hat{\PP}_z$.
\end{definition}

Originally considered in information theory and applied probability \citep{johnson_2004}, this divergence has recently been used to derive novel concentration inequalities \citep{Rioul2010} and in many applications in machine learning and computer vision \citep{hyvarinen2005,SongErmon2019ScoreGenerative}. However, to the best of our knowledge, it has never been considered in distributionally robust optimization or control models before.

If can be readily seen that $\mathds{F}(\PP_z,\hat \PP_z)$ is non-negative and  equals~$0$ iff $\PP_z=\hat{\PP}_z$. Our definition extends the Fisher divergence to all distributions from $\setalldist{d_z}$, although the distributions of interest are from $\tilde{\mathcal{P}}(\R^{d_z})$. (Note that this also requires the nominal distribution $\hat{\PP}_z$ to belong to $\tilde{\mathcal{P}}(\R^{d_z})$, which means that setting $\PP_z$ as an empirical distribution from a finite set of samples is not possible.) The quantity $s(z) =\nabla\log p(z)$ is commonly referred to as the \emph{score} of the distribution $\PP$, which explains the alternative naming for the divergence. We are interested in a Gaussian nominal distribution $\hat \PP_z$, and we subsequently use $\hat s(z) =\nabla\log\hat p(z)=-\hat\Sigma_z^{-1}(z-\hat\mu_z)$ to denote its score. 

\subsection{Verifying Assumption~\ref{ass:general_assumption_about_M_function}}

\begin{proposition}
\label{prop:Fisher_convex_compact}
  If $\rho_z \geq 0$, the Fisher divergence $\mathds{F}$ satisfies Assumption~\ref{ass:general_assumption_about_M_function}.
\end{proposition}
\begin{proof}{Proof.}
To simplify notation, we subsequently drop the subscript $z$. We verify parts~(i) and~(ii) separately.

\smallskip
\noindent{\bf Part (i).} Define the ambiguity set 
\(
\mathcal C(\mu,M)=
\{\PP \in \tilde{\mathcal{P}}(\R^d): \EE_{\PP}[z]=\mu, ~\EE_{\PP}[zz^{\top}]=M\}.
\)
It suffices to prove that:
\[
\mathds{F}(\PP,\hat \PP) \geq \mathds{F}\bigl(\mathcal{N}(\mu,M), \hat \PP \bigr) ~ \forall \, \PP \in \mathcal C(\mu,M),
\]
and equality holds if only if $~\PP_z = \mathcal{N}(\mu_z,M_z)$. To prove this inequality, consider any distribution $\PP_{} \in \mathcal C(\mu_{},M_{})$ and expand the expression for $\mathds{F}(\PP_{},\hat \PP_{})$:
\[
\mathds{F}(\PP_{},\hat \PP_{})
=\frac12\,
\EE_{\PP_{}}\bigl[\norm{\nabla\log p(z)}_2^{2}\bigr]
+\frac12\,\EE_{\PP_{}}\bigl[\norm{\hat s(z)}_2^{2}\bigr]
-\EE_{\PP_{}}\bigl[ (\nabla\log p(z)) ^\top \hat s(z)\bigr].
\]
We evaluate each term separately. The first term is exactly equal to $\tfrac12 \Phi(\PP_{})$, where $\Phi(\PP_{}) =\EE_{\PP_{}}\bigl[\norm{\nabla\log p(z)}_{2}^{2}\bigr]$ is the trace of the Fisher information matrix of $\PP_{}$. 

Because $\hat s(z)$ is linear in~$z$, the second term $\EE_{\PP_{}}[\norm{\hat s(z)}_2^{2}]$ only depends on the first two moments of the distribution $\PP_{}$, so is a constant on $\mathcal C(\mu_{},M_{})$.

For the last term, we use integration by parts to prove that $\EE_{\PP}[(\nabla\log p(z)) ^\top \hat s(z)]$ is also constant on $\mathcal C(\mu,\Sigma)$. Recall that for any scalar function \( f \) and vector field \( g \),
\[
(\nabla f)^\top g = \operatorname{div}(fg) - f\, \operatorname{div} g,
\]
where $\operatorname{div}(g) = \nabla \cdot g = \sum_i \frac{\partial g_i(z)}{\partial z_i}$. Take \( f(z) = p(z) \) and \( g(z) = \hat{s}(z) \). Then, integrating over all \( \mathbb{R}^d \),
\begin{align}   
\EE_{\PP}[ (\nabla\log p(z))^\top \hat s(z)] = \underbrace{\int_{\mathbb{R}^d} \operatorname{div}(p \hat{s}) \diff z}_{A}
- \underbrace{\int_{\mathbb{R}^d} p(z)\, \operatorname{div} \hat{s}(z)\diff z}_{B}.
\label{eq:AB_terms_def_fisher}
\end{align}
We claim that $A=0$ and $B$ is a constant for all distributions in $\mathcal{C}$. We first argue the latter: note that $\hat{s}(z) = -\hat{\Sigma}^{-1}(z - \hat{\mu})$, so $\operatorname{div} \hat{s}(z) = -\operatorname{Tr}(\hat{\Sigma}^{-1})$. Integrating this constant over the density $p(z)$ thus proves that $B$ is a constant. To prove that $A=0$, note that $A$ can be expressed via the Gauss-Ostrogradski Theorem as:
\[
A = \lim_{R \rightarrow \infty} A_R, ~\mbox{where}~ A_R = \int_{B_R} \operatorname{div} F(z) \diff z = \int_{S_R} (F(z))^\top n(z) \diff S(z).
\]
Above, $F(z) = p(z) \hat{s}(z)$ and for $R>0$, we define $B_R=\{z\in \R^{d} : \|z\|\le R\}$ and $S_R=\partial B_R=\{z \in \R^d :\|z\|=R\}$ as the sphere with radius $R$, with outward unit normal $n(z)=z/\|z\|$ and unit volume/area given by $\diff S(z)$. On $S_R$, we have
\[
|(\hat{s}(z)^\top n(z)| = \Bigl| \bigl(-\hat{\Sigma}^{-1}(z - \hat{\mu}) \bigr)^\top n(z)\Bigr| \leq \| \hat{\Sigma}^{-1}\| \big( \|z\| + \|\hat{\mu}\| \big)
= \| \hat{\Sigma}^{-1}\| R + \|\hat{\Sigma}^{-1}\| \|\hat{\mu}\|, ~\forall \, z \in S_R.
\]
Therefore,
\begin{equation*}
|A_R| \leq \Bigl( \|\hat{\Sigma}^{-1}\| R + \| \hat{\Sigma}^{-1}\| \|\hat{\mu}\| \Bigr) \underbrace{\int_{S_R} p(z) \diff S(z)}_{=I_R}.
\end{equation*}
We claim that requirement~\eqref{eq:Fisher_fast_poly_decay} implies that $R I_R \rightarrow 0$ as $R\rightarrow \infty$, which in turn implies that $I_R \rightarrow 0$ and completes the argument that $A = 0$. We have: 
\begin{align} 
R \cdot I_R = R \int_{S_R} p(z)\, dS(z) \leq R \int_{S_R}  \frac{\bar{p}}{\|z\|^{d + \epsilon}} \diff S(z) = \omega_{d} R^{-\epsilon} \bar{p},
\label{eq:Fisher_inequality_IR}
\end{align}
where the inequality follows from~\eqref{eq:Fisher_fast_poly_decay} and the last step follows by recalling that the sphere $S_R$ in $\R^{d}$ has surface $\omega_{d} R^{d-1}$, for a constant $\omega_{d}$ that depends on $d$. Then, we readily conclude that the right-hand-side in~\eqref{eq:Fisher_inequality_IR} converges to 0 as $R\rightarrow \infty$, proving that $A=0$.

The arguments above show that $\mathds{F}(\PP,\hat \PP) =\tfrac12\,\Phi(\PP)$ plus a term that is constant over $\mathcal C(\mu,\Sigma)$. By the Stam Inequality \citep{stam1959} and its multidimensional generalizations (see inequality~(16) and the related discussion in~\cite{Rioul2010}), it is well known that the Fisher information $\Phi(\PP)$ for any random variable with covariance $\Sigma_\PP$ satisfies the inequality:
\begin{align*}
    \Phi(\PP) \geq \Tr( \Sigma_\PP^{-1} ), 
\end{align*}
with equality only if the random variable is Gaussian. This implies that among all distributions $\PP \in \mathcal{C}$, the one that minimizes $\mathds{F}(\PP, \hat \PP)$ is a Gaussian distribution. 

\medskip
\noindent{\bf Part (ii).} Consider the set $\mathcal{M}^{\mathds{F}}_{\mu,M} = \bigl\{(\mu,M)\in \setmoments{d} : \mathds{F}\bigl(\mathcal{N}(\mu,M), \hat{\PP}\bigr)\le\rho \bigr\}$. The Fisher divergence between two Gaussian distributions with mean-covariance pairs $(\mu,\Sigma)$ and $(\hat \mu, \hat \Sigma)$, respectively, has the closed form expression \cite[Eq.~(2.2)]{chafai2021}
\begin{equation}\label{eq:D_Fisher_Gaussians}
\norm{\hat\Sigma^{-1}(\mu-\hat\mu)}_2^{2}
+\Tr\bigl(\hat\Sigma^{-2}\Sigma
               -2\hat\Sigma^{-1}+\Sigma^{-1}\bigr),
\end{equation}
We can therefore rewrite:
\begin{align*}
\mathds{F}\bigl(\mathcal{N}(\mu,M), \hat \PP\bigr) 
 & =\norm{\hat\Sigma^{-1}(\mu-\hat\mu)}_2^{2} + \Tr\bigl(\hat\Sigma^{-2} (M - \mu \mu^\top) -2\hat\Sigma^{-1}+ (M - \mu \mu^\top)^{-1}\bigr) \\
 &=\Tr \Bigl( \hat\Sigma^{-2} \big(  (\mu - \hat{\mu})(\mu - \hat{\mu})^\top + M - \mu \mu^\top \bigr) \Bigr) + \Tr\bigl( -2\hat\Sigma^{-1}+ (M - \mu \mu^\top)^{-1}\bigr).
\end{align*}
Recognizing that quadratic terms in $\mu$ cancel out, the first trace term above is linear in $(\mu,M)$. Thus, the expression is jointly convex in $(\mu,M)$ if we can argue that the mapping $(\mu,M) \mapsto \Tr ( (M - \mu \mu^\top)^{-1} )$ is jointly convex in $(\mu,M)$ on $\setmoments{d}$. We claim that for any \( S \in \mathbb{S}_{++}^d \), we have the identity:
\begin{equation}
\Tr(S^{-1}) = \sup_{W \in \mathbb{S}_+^d} 2 \Tr(W) - \Tr(S W^2).
\label{eq:variational_rep_trace_inverse}
\end{equation}
To prove this, fix \( S \succ 0 \) and set $\phi(W) = 2\, \Tr(W) - \Tr(S W^2)$ for $W \succeq 0$. Because \( \phi \) is concave in \( W \), its maximizers must satisfy the first-order condition \( 0 = 2I - SW - WS \). The choice $W\opt = S^{-1}$ satisfies this equation and is feasible, and evaluating the objective for $W\opt$  proves~\eqref{eq:variational_rep_trace_inverse}. Applying~\eqref{eq:variational_rep_trace_inverse} in our case, with \( S = M - \mu\mu^\top \), yields:
\[
 \Tr \bigl( (M - \mu \mu^\top)^{-1} \bigr) = \sup_{W \in \mathbb{S}_{+}^d}  2\, \Tr(W) - \Tr \bigl( (M - \mu\mu^\top) W^2 \bigr) .
\]
It can be readily seen that the expression maximized above is jointly convex in $(\mu,M)$ for any $W$. Because the supremum of convex functions preserves convexity, we readily obtain that $(\mu,M) \mapsto \Tr ( (M - \mu \mu^\top)^{-1} )$ is jointly convex in $(\mu,M)$ on the set $\{(\mu, M): M- \mu \mu^\top \succ 0\}$, which proves that the set $\mathcal{M}^\mathds{F}_{(\mu,\Sigma)}$ is convex for any $\rho \geq 0$. 

That $\mathcal{M}^\mathds{F}_{(\mu,\Sigma)}$ is compact follows because the term $\norm{\hat\Sigma^{-1}(\mu-\hat\mu)}_2^{2}$ is a coercive function in $\mu\in\mathbb R^{n}$ and $\Tr( \hat{\Sigma}^{-2} M )$ is coercive in $M\in\mathbb S^n_+$. 
\end{proof}

%
%
\subsection{Verifying Assumption~\ref{ass:setting_mean_zero_feasible}}
\begin{proposition}
    Fix $\hat{\mu}_z = 0$. For any $\rho_z \geq 0$, the Fisher Divergence $\mathds{F}$ satisfies Assumption~\ref{ass:setting_mean_zero_feasible}.
\end{proposition}
\begin{proof}{Proof.}
Again, we omit subscripts $z$ for notational simplicity. The set of feasible moments $\mathcal{M}^\mathds{F}_{(\mu,M)}$ exactly corresponds to the 0 sublevel set of the function $h : \setmoments{d} \times \setmoments{d} \rightarrow \R$ defined as:
\begin{align*}
h\bigl( (\mu,M), (\hat{\mu},\hat{M})\bigr) &= \mathds{F}\bigl(\mathcal{N}(\mu,M), \mathcal{N}(\hat{\mu}, \hat{M}) \bigr) - \rho \\
  &= \| \hat{M}^{-1}  (\mu-\hat{\mu}) \|^2 +  \Tr \Bigl( \hat{M}^{-2} \bigl( M - \mu \mu^\top \bigr) \Bigr) -2 \Tr (\hat{M}^{-1}) + \Tr\bigl( (M - \mu \mu^\top)^{-1}\bigr) - \rho.
\end{align*}
To complete the proof, note that: 
\begin{align*}
   h\bigl( (\mu,M), (0,\hat{M})\bigr) - h\bigl( (0,M), (0,\hat{M})\bigr) 
   = \Tr\bigl( (M - \mu \mu^\top)^{-1}\bigr) - \Tr\bigl( M^{-1}\bigr) \geq 0,
\end{align*}
where the inequality follows because the operator $X \mapsto X^{-1}$ reverses order on $\mathbb{S}_{++}^d$ and $M - \mu \mu^\top \preceq M$ for any $M \in \mathbb{S}_+^d$.
\end{proof}

%
%
\subsection{Verifying Assumption~\ref{ass:M-gradients}}
\begin{proposition}
    Fix $\hat{\mu}_z=0$. For any $\rho_z > 0$, the Fisher Divergence $\mathds{F}$ satisfies Assumption~\ref{ass:M-gradients}.
\end{proposition}
\begin{proof}{Proof.}
    We omit the subscript $z$. Let $\gz(\Sigma) = \mathds{F}\bigl(\mathcal{N}(0,\Sigma), \mathcal{N}(0, \hat{\Sigma}) \bigr) - \rho$.  By~\eqref{eq:D_Fisher_Gaussians}, the expression for $\gz$ is:
    \begin{align*}
        \gz\bigl(\Sigma) =\Tr\bigl(\hat\Sigma^{-2}\Sigma
                       -2\hat\Sigma^{-1}+\Sigma^{-1}\bigr) - \rho.
    \end{align*}
    That $\gz$ is differentiable on $\mathbb{S}_{++}^d$ is immediate. That $\gz$ is minimized at $\Sigma=\hat{\Sigma}$ follows because $\mathds{F}$ is a divergence, so $\mathds{F}\bigl(\mathcal{N}(0,\Sigma),\mathcal{N}(0, \hat{\Sigma}) \bigr)$ is minimized for $\Sigma=\hat{\Sigma}$. That $\gz$ is convex follows readily from \Cref{prop:Fisher_convex_compact} and its proof. These arguments imply that properties~(i)-(ii) are readily satisfied. To verify~(iii), note that the gradient is given by $\gradgz(\Sigma) = \hat{\Sigma}^{-2} - \Sigma^{-2}$. Therefore, for $\Sigma_1,\Sigma_2 \in \mathbb{S}^d_{++}$,
    \begin{align*}
        \gradgz(\Sigma_1) \succeq \gradgz(\Sigma_2) ~\Leftrightarrow~ \Sigma_2^{-2} \succeq \Sigma_1^{-2} ~\Leftrightarrow~ \Sigma_1 \succeq \Sigma_2,
    \end{align*}
    where the equivalences follow because $X \mapsto X^{-2}$ is order reversing on $\mathbb{S}_{++}^{d}$.  
\end{proof}

%
%
\bigskip
\section{Extension to Elliptical Nominal Distribution}
\label{sec: elliptical}
In this section, we focus on an ambiguity set where $\mathds{D}$ is the 2-Wasserstein distance~$\mathds{W}$. We will show that our results from \Cref{sec:Nash} continue to hold for any \emph{elliptical} nominal distribution $\hat{\PP}$ with finite second moments. To this end, we recall the following definition of elliptical distributions. 
\begin{definition}[Elliptical distribution]\label{def:elliptical distribution}
The distribution $\PP_z$ for $z \in \R^{d_z}$ is called elliptical if the characteristic function~$\Phi_{\PP_{z}}(t) = \EE_{\PP_{z}}[\exp(it^\top z)]$ of $\PP_{z}$ has a representation of the form $\Phi_{\PP_z}(t) = \exp(i t^\top \mu) \psi(t^\top S t)$ for some $\mu_z \in \R^{d_z}$ (location parameter), some~$S_z \in \mathbb S_+^{d_z}$ (dispersion matrix), and some function $\psi: \R_+ \to \R$ (characteristic generator). 
\end{definition}
This class of distributions includes many non-Gaussian distributions, such as the 
Laplace, logistic, or hyperbolic distributions, among others~\citep{ref:fang2018symmetric}. Gaussian distributions are a special case of elliptical distributions, with the characteristic generator~$\psi(r) = \exp(-r/2)$. 


In the remainder of this section, we focus on elliptical distributions with finite second moments. Such distributions can be re-parameterized to ensure that the dispersion matrix $S_z$ exactly equals the covariance matrix $\Sigma_z$, which will be convenient for our subsequent discussion. The following remark formalizes this idea.

\begin{remark}[Technical remark on re-parametrization]\label{rem: Technical remark on re-parametrization}
Denote by~$\mathcal{E}$ an arbitrary elliptical distribution of $z$ with finite second moments. Recall that the characteristic function~$\Phi_{\mathcal{E}}$ always exists and is finite even if some moments of $\mathcal{E}$ do not exist. Denote the right derivative of $\psi(u)$ at~$u=0$ as $\psi'(0)$. The mean and covariance of~$\mathcal{E}$ exist if and only if $\psi'(0)$ exists and is finite, and they can be expressed as $\mu_z$ and $-2\psi'(0) S_z$~\cite[Theorem~4]{ref:cambanis1981theory}, respectively. Denote by $\mathcal{E}^\psi(\mu_z, S_z)$ the elliptical distribution with mean $\mu_z$, dispersion matrix~$S_z$, and characteristic generator $\psi$. It can be checked that for any admissible~$\psi$ with $|\psi'(0)|< \infty$, we have $\mathcal{E}^\psi(\mu_z, S_z) = \mathcal{E}^{\tilde{\psi}}(\mu_z, \tilde{S}_z)$, where $\tilde \psi(u) = \psi({-u}/ ({2 \psi'(0)})) $ and $\tilde S_z = -2 \psi'(0) S_z$. But then, a choice that ensures $\tilde \psi'(0) = -1/2$ will also imply that $\tilde S_z = \Sigma_z$, so we can re-parameterize any elliptical distribution so that its dispersion matrix equals its covariance matrix. (The re-parametrization has no effect on the actual distribution.) 
\end{remark}
\Cref{rem: Technical remark on re-parametrization} illustrates that any elliptical distribution can be defined by a characteristic generator $\psi$, a mean $\mu_z$, and a covariance matrix $\Sigma_z$ -- or equivalently, a second-moment matrix $M_z$ -- where, without loss of generality, we have $S_z = \Sigma_z$. In the following, we use $\mathcal{E}^\psi(\mu_z, M_z)$ to denote an elliptical distribution with characteristic generator $\psi$, mean $\mu_z$, and second-moment matrix $M_z$ 
(and where no confusion can arise, we also use $\mathcal{E}^\psi(\mu_z, \Sigma_z)$ to denote an elliptical distribution with characteristic generator $\psi$, mean $\mu_z$, and covariance matrix $\Sigma_z$). Henceforth, we use the terms \textit{dispersion matrix} and \textit{covariance matrix }interchangeably. 

\subsection{Assumptions}
We now formally present the counterparts of the tractability assumptions from the main text for the setup considered in this appendix section.
\begin{assumption}
\label{ass: elliptical}
$\hat {\PP}$ is an elliptical distribution with finite second moments.
\end{assumption}
Requiring $\hat{\PP}$ to be elliptical renders the model studied in this appendix section computationally tractable and is consistent with the assumptions of the so-called linear quadratic-elliptical (LQE) model, which assumes that the exogenous uncertainties follow a known, elliptical distribution and are uncorrelated. Note that if the joint distribution of the random variables is elliptical, mutual uncorrelatedness does not imply independence as in the case of Gaussians, but linear conditioning preserves ellipticity. 
Similarly to the LQG model, the minimum mean-squared-error state estimators~$\hat x_t=\mathbb E_{\PP}[x_t|y_0,\ldots,y_t]$ for the LQE model can be obtained through Kalman filtering and dynamic programming techniques \citep{ref:basu1994bayesian, ref:chu1972estimation}. In fact, the Kalman filter equations and the expressions for the optimal control inputs discussed in Appendix~\S\ref{sec: appx: optimal-solution-classic-LQG} remain unchanged for the LQE case, with the sole difference being that $\Sigma_t$ and $\Sigma_{t+1 | t}$ do not necessarily represent the covariance matrices of the estimated state in the LQE case. 

For the rest of our discussion, we also recall that marginal distributions of elliptical distributions are elliptical with the same characteristic generator \cite[Corollary 3.1]{hult2002multivariate}.

\begin{assumption}
\label{ass: elliptical Wasserstein ambiguity}
$\mathds{D}$ is the 2-Wasserstein distance~$\mathds{W}$.
\end{assumption}
Assumption~\ref{ass: elliptical Wasserstein ambiguity} guarantees that a variant of Assumption~\ref{ass:general_assumption_about_M_function} holds, where Gaussian distributions are replaced by elliptical distributions that share the same characteristic generator as the elliptical nominal distribution $\hat{\PP}_z$. This is formalized in \Cref{prop: elliptical variant of Assumption 2 holds}. 
%
%
\begin{proposition}\label{prop: elliptical variant of Assumption 2 holds}
The 2-Wasserstein distance~$\mathds{W}$ satisfies the following properties:
\begin{enumerate}[noitemsep,nolistsep,label=(\roman*)]    
\item For every $(\mu_z, M_z) \in \setmoments{d_z}$, an elliptical distribution that shares the same characteristic generator ${\psi}$ as $\hat \PP_z$ minimizes the distance~$\mathds{W}$ from $\hat \PP_z$ among all distributions with mean~$\mu_z$ and second moment matrix $M_z$, that is,
    \begin{equation*}
    \mathds W(\mathcal E^{{\psi}}(\mu_z, M_z), \hat \PP_z) = \left\{
    \begin{array}{ccll}  
         &\inf\limits_{\PP_z \in \setalldist{d_z}} &\mathds W(\PP_z, \hat \PP_z) \\
         &\st&\EE_{\PP_z}[z] = \mu_z, &\EE_{\PP_z}[z z^\top ] = M_z.
    \end{array}\right.
   \end{equation*}
    \item The set $\mathcal{M}^{\mathcal{E}^\psi}_{(\mu_z, M_z)} = \{(\mu_z, M_z) \in \setmoments{d_z} : \, \mathds W(\mathcal E^{{\psi}}(\mu_z, M_z), \hat{\PP}_z) \leq \rho_z \}$ is convex and compact. 
    \end{enumerate}
\end{proposition}
\Cref{prop: elliptical variant of Assumption 2 holds} follows from Theorem~2.1 in~\cite{gelbrich1990formula}, which shows that, for any two distributions $\PPz{1}$ and $\PPz{2}$ on $\R^{d_z}$ with first and second moments given by $(\mu_1, M_1) \in \setmoments{d_z}$ and $( \mu_2, M_2) \in \setmoments{d_z}$, respectively,
\begin{align*} 
\mathds{W}(\PPz{1}, \PPz{2}) \geq  \mathds{G}\left((\mu_1, M_1 - \mu_1 \mu_1^\top),(\mu_2, M_2 - \mu_2 \mu_2^\top)\right),
\end{align*}
with equality holding if $\PPz{1}$ and $\PPz{2}$ are elliptical with the same characteristic generator. Here, $\mathds{G} ((\mu_1, \Sigma_1), (\mu_2, \Sigma_2))$ is the Gelbrich distance between two mean-covariance pairs, defined in~\eqref{eq: Gelbrich distance definition}. This implies that \Cref{prop: elliptical variant of Assumption 2 holds}-(i) is satisfied. \Cref{prop: elliptical variant of Assumption 2 holds}-(ii) is also satisfied because the set $\mathcal{M}^{\mathcal{E}^\psi}_{(\mu_z, M_z)}$ exactly corresponds to the set of pairs of first and second moments with Gelbrich distance at most $\rho_z$ from the mean-covariance pair for the nominal distribution, which is known to be convex and compact \citep[Proposition~3.17]{ref:nguyen2019adversarial}.

Our treatment considers only 2-Wasserstein distance~$\mathds{W}$ because we are not aware of other divergences $\mathds{D}$ that satisfy the requirements in~\Cref{prop: elliptical variant of Assumption 2 holds} for elliptical distributions $\hat{\PP}$.
Subsequently, we refer to the DRLQ problem formulation with elliptical nominal and 2-Wasserstein distance as the ``elliptical DRLQ problem" for conciseness.

\subsection{Results for the Elliptical DRLQ Problem}
\label{sec:results_elliptical_DRLQ}
It is worth emphasizing that the elliptical DRLQ problem exactly mirrors the Gaussian case with 2-Wasserstein distance discussed in \Cref{sec:examples}. The reason is that the 2-Wasserstein distance between two elliptical distributions $\PP$ and $\hat{\PP}$ with the same generator $\psi$ exactly equals the 2-Wasserstein distance between two Gaussian distributions with the same first two moments, and both distances are given by the Gelbrich distance between the respective mean-covariance pairs:
\begin{align}
    \mathds{W}\bigl(\mathcal{E}^{{\psi}}(\mu_z, M_z), \mathcal{E}^{{\psi}}(\hat{\mu}_z, \hat{M}_z) \bigr) 
    &= \mathds{W}\bigl(\mathcal{N}(\mu_z, M_z), \mathcal{N}(\hat{\mu}_z, \hat{M}_z) \bigr) \notag \\
    &= \mathds{G}\bigl( (\mu_z, M_z - \mu_z \mu_z^\top), (\hat{\mu}_z, \hat{M}_z - \hat{\mu}_z \hat{\mu}_z^\top) \bigr).
    \label{eq:equivalence_2Wass_elliptical}
\end{align}
This directly parallels the Gaussian case in \Cref{sec:examples} (under the assumption that the nominal distributions have the same means and second moments). For instance, the set of valid pairs of first and second moments $\mathcal{M}^{\mathcal{E}^\psi}_{(\mu_z, M_z)}$ defined in~\Cref{prop: elliptical variant of Assumption 2 holds} exactly equals the set $\mathcal{M}_{(\mu_z, M_z)}$ defined in \Cref{sec:examples}. Moreover, because most of our constructions and results rely on the set $\mathcal{M}_{(\mu_z, M_z)}$, we can mirror all the main results in \Cref{sec:Nash}. For conciseness, we refrain from formally restating and reproving all results and instead just discuss how they apply to the elliptical DRLQ case.

We consider the reformulation \eqref{eq:DRLQG} of the DRLQ problem in terms of purified observations, and its dual \eqref{DRCPdual}. These reformulations are applicable here because they do not rely on Assumption~\ref{ass: elliptical} and Assumption~\ref{ass: elliptical Wasserstein ambiguity} (or their counterparts in the main text).

\subsubsection*{Upper Bound for Primal.}
Similar to the construction of the primal upper bound in \Cref{sec:Nash}, we can obtain an upper bound for $p^\star$ by enlarging the ambiguity set~$\mathcal{B}$ and restricting the policies $u_t$ to affine dependencies. 
We define the following outer approximation for the ambiguity set:
\begin{equation*}
    \begin{aligned}
        &\overline{\mathcal{B}} = \left\{\PP \in \mathcal P(\R^{n+ T(n+ p)}): \PP_{z} \in \overline{\mathcal{B}}_z, ~ \EE_{\PP}[z' z^\top] = 0 ~\forall z ,z'\in \mathcal Z, ~ z'\neq z \right\},
    \end{aligned}
\end{equation*}
where, for all $z \in \mathcal Z$,
\begin{align*}
\overline{\mathcal{B}}_z\!=\! \left\{\PP_z  \in  \setalldist{d_z} : \exists (\mu_z , M_z) \in \setmoments{d_z} \text{ with } \EE_{\PP_z}[z]\! =\! \mu_z,  \EE_{\PP_z}[z z^\top ]\! =\! M_z,  \mathds{W}\left(\mathcal E^{{\psi}}(\mu_z, M_z), \hat{\PP}_z\right) \!\leq\! \rho_z \right\},
\end{align*}
where $\psi$ is the characteristic generator of $\hat{\PP}_z$ (the same for all $z \in \mathcal Z$ because the characteristic generator of the marginals is the same as that of $\hat{\PP}$). Affine policies take the same form as before, i.e., $u = q + U   \eta = q + U(Dw + v)$, where $q = (q_0, \dots, q_{T-1})$, and $U$ is a block lower triangular matrix as defined in \eqref{eq:block-lower-triangular}. The optimal value of Problem~\eqref{upper bound problem 1} now constitutes an upper bound for the primal here. \Cref{prop: min max SDP upper bound problem} also holds, and the proof follows the same arguments. 

\subsubsection*{Lower Bound for Dual.} We restrict nature's feasible set to the family~$\mathcal{B}_{\mathcal E^{\psi}}$ of all elliptical distributions from the ambiguity set~$\mathcal{B}$ that have the same characteristic generator $\psi$ as $\hat{\PP}$. Paralleling the construction in \Cref{sec:Nash}, we formulate a lower bound problem:
\begin{equation}
\label{eq: elliptical dual distributionally robust control problem restriction}
    \underline{d}^\star = \left\{
    \begin{array}{ccl} \max\limits_{\mathbb P \in \mathcal{B}_{\mathcal E^{\psi}}} &\min\limits_{x,   u} &\mathbb E_{\mathbb P} \left[   u^\top R   u +   x^\top Q  x \right]\\
    &\st &  u \in \mathcal U_{\eta}, {x} = H u + G w.
    \end{array}\right.
\end{equation}
We can now prove that a variant of \Cref{prop:dual-sdp} holds.
\begin{proposition}
    Under the standing assumptions in this section, problem~\eqref{eq: elliptical dual distributionally robust control problem restriction} has the same optimal value as problem~\eqref{eq:dr-affine-max-min}.
    \label{prop:elliptical dual-sdp}
\end{proposition}
\begin{proof}[Proof of~\Cref{prop:elliptical dual-sdp}.]
The proof follows from arguments similar to those in the proof of \Cref{prop:dual-sdp}. Noting that the inner minimization problem in \eqref{eq: elliptical dual distributionally robust control problem restriction} is solved by an affine policy according to standard LQE theory, we can reformulate problem~\eqref{eq: elliptical dual distributionally robust control problem restriction} as
\begin{equation}
\begin{array}{cccll}
&\max\limits_{\substack{\mu_w, {M_w}, \\ \mu_v, {M_v},\mathbb P}} &\min\limits_{\substack{q \in \R^{pT}\\ U \in \mathcal U}}  &
\Tr\Big( \bigl( (UD)^\top R UD + (G+HUD)^\top Q (G+HUD) \bigr) M_w  +  U^\top \bar R U M_v\Big ) \\[-1ex]
&&&\qquad +2 q^\top(\bar R U D + G^\top QH) \mu_w + 2 q^\top \bar R U \mu_v + q^\top \bar R q,\\[5pt]
&&\st &\mathbb P \in \mathcal{B}_{\mathcal E^{\psi}}, ~\mu_w = \EE_{\PP}[w], ~M_w= \mathbb{E}_{\mathbb P}[w w^\top],
~\mu_v = \EE_\PP[v],~M_v = \mathbb{E}_{\mathbb P}[  v v^\top].
\end{array}
\label{eq:elliptical simplified_rewriting_dual_bnd}
\end{equation}
It remains to prove that problem~\eqref{eq:elliptical simplified_rewriting_dual_bnd} has the same optimal value as problem~\eqref{eq:dr-affine-max-min}. Consider any $(\mu_w, {M_w}), (\mu_v, {M_v}), \mathbb P$ feasible in the outer maximization problem in~\eqref{eq:elliptical simplified_rewriting_dual_bnd}. Because $\PP\in \mathcal \mathcal{B}_{\mathcal E^{\psi}}$ is an elliptical distribution with characteristic generator $\psi$ and marginal distributions of elliptical distributions are also elliptical with the same characteristic generator, we can write the requirement $\mathds{W}(\PP_z, \hat{\mathbb{P}}_{z}) \leq \rho_{z}$ in the definition of $\mathcal{B}_{\mathcal E^{\psi}}$ equivalently as $\mathds{W}({\mathcal E^{\psi}}(\mu_{z},M_{z}), \hat{\mathbb P}_{z}) \leq \rho_{z}$, for any $z \in {\cal Z}$. By~\eqref{eq:equivalence_2Wass_elliptical}, this implies that $(\mu_w, {M_w}), (\mu_v, {M_v})$ is feasible in the outer maximization problem in~\eqref{eq:dr-affine-max-min}, proving that the optimal value in~\eqref{eq:dr-affine-max-min} is at least as large as that in~\eqref{eq:elliptical simplified_rewriting_dual_bnd}. However, for any~$(\mu_w, {M_w}), (\mu_v, {M_v})$ feasible in the outer maximization in~\eqref{eq:dr-affine-max-min}, the elliptical distribution with characteristic generator $\psi$, mean $(\mu_w, \mu_v)$ and covariance matrix $\operatorname{diag}({M_w} - \mu_w \mu_w^\top, {M_v} - \mu_v \mu_v^\top)$
is feasible in~\eqref{eq:elliptical simplified_rewriting_dual_bnd}, proving that the two problems have the same optimal value. 

\end{proof}
\subsubsection*{Optimality of Affine Policies and Elliptical Distributions.} A counterpart of \Cref{theorem:lower-equal-upper} -- i.e., the strong duality of the primal~\eqref{eq:DRLQG} and its dual~\eqref{DRCPdual} -- holds under the assumptions in this section, and the proof relies on mirroring arguments. This implies that an affine policy $u = q + U \eta$ is optimal in the elliptical DRLQ problem and that an elliptical distribution with the same characteristic generator as $\hat{\PP}$ is optimal for the dual of the elliptical DRLQ problem.
\subsubsection*{Optimality of Linear Policies and Zero-Mean Distributions.} Mirroring our discussion in \Cref{sec:Nash}, we can readily argue that if $\hat{\mu}_z=0$ for every $z \in \mathcal{Z}$, Assumption~\ref{ass:setting_mean_zero_feasible} readily holds for the elliptical DRLQ problem. This follows because the distances between two \emph{elliptical} distributions $\PPz{1},\PPz{2} \in \setalldist{d_z}$ (with the same characteristic generator as $\hat{\PP}$) is exactly given by the Gelbrich distance between their pairs of first two moments, by~\eqref{eq:equivalence_2Wass_elliptical}. So an argument that mirrors the one for the 2-Wasserstein distance can be used to show that Assumption~\ref{ass:setting_mean_zero_feasible} holds.

This allows us to extend \Cref{thm:worst-case-mean-zero} to this case, with an identical line of arguments. We can therefore conclude that the dual of the elliptical DRLQ problem admits an optimal solution $\PP\opt$ that is elliptical and has the same mean as the nominal mean (i.e., zero), and under this distribution, there is an optimal \emph{linear} control policy, $u\opt = U\opt \eta$. 
\subsubsection*{Efficient Numerical Solution.} The results in \Cref{sec:DR-LQG-algorithm} can also be readily extended to the elliptical DRLQ problem because our algorithms rely on solving problem~\eqref{eq: dual distributionally robust control problem restriction} and the objective and feasible set of that problem are identical in the elliptical DRLQ case and in the Gaussian case due to~\eqref{eq:equivalence_2Wass_elliptical}. 


\bibliographystyle{abbrvnat}
\bibliography{bib}
\end{document}

%% file: style.tex
\usepackage{algorithm}
\usepackage{hyperref}
\usepackage[square,numbers]{natbib}

\setcitestyle{number}

\hypersetup{colorlinks, breaklinks,
linkcolor={[RGB]{184,86,103}},
citecolor={[RGB]{113,100,198}},
urlcolor={green!50!black}}

\usepackage{dsfont} 
\usepackage{bbm} 
\usepackage{footnote}

\usepackage{enumerate}
\usepackage{enumitem}    
\usepackage{comment}
\usepackage{enumerate}

\usepackage{multicol}

\usepackage{tcolorbox}
\usepackage{amsthm}
\usepackage{amsmath}
\usepackage{amssymb}
\usepackage{algorithm}
\newtheorem{theorem}{Theorem}[section]
\newtheorem{lemma}{Lemma}
\newtheorem{proposition}{Proposition}
\newtheorem{remark}{Remark}
\newtheorem{definition}{Definition}
\newtheorem{corollary}{Corollary}
\newtheorem{assumption}{Assumption}

\DeclareMathOperator{\Tr}{Tr}
\DeclareMathOperator*{\argmin}{argmin}

\newcommand{\R}{\mathbb{R}}

\newcommand{\st}{\mathrm{s.t.}}
\newcommand{\diff}{\mathrm{d}}
\newcommand{\opt}{^\star}

\newcommand{\EE}{\mathbb{E}}
\newcommand{\PP}{\mathbb{P}}
\newcommand{\QQ}{\mathbb{Q}}

\graphicspath{{figures/}}

\usepackage{xcolor}
\newcount\Comments
\newcommand{\kibitz}[2]{\ifnum\Comments=1{\textcolor{#1}{\textsf{\footnotesize #2}}}\fi}

\usepackage[margin=1in]{geometry}

\allowdisplaybreaks

\usepackage[dvipsnames]{xcolor}

\usepackage{graphicx} 